\documentclass[12pt]{amsart}
\usepackage{amssymb,latexsym,eucal}

\usepackage{hyperref}
\usepackage[letterpaper,centering,text={6.5in,9in}]{geometry}
\usepackage{graphicx}
\usepackage{breqn}
\usepackage{amsaddr}

\usepackage{latexsym}
\usepackage[utf8]{inputenc}
\usepackage{amsmath}
\usepackage{amsthm}
\usepackage{bigints}
\usepackage{amsfonts}
\usepackage{amssymb}
\usepackage{epsfig}
\usepackage{nicefrac}
\usepackage[all]{xy}
\usepackage{graphicx}
\usepackage{enumerate}
\usepackage{enumitem}
\usepackage{mathtools}    
\DeclarePairedDelimiterX\setc[2]{\{}{\}}{\,#1 \;\delimsize\vert\; #2\,}
\usepackage{float}
\usepackage{pgfplots}

\usepackage{caption,subcaption}
\usepackage{tikz}
\usetikzlibrary{
intersections, arrows.meta,
automata,er,calc,
backgrounds,
mindmap,folding,
patterns,
decorations.markings,
fit,decorations.pathmorphing,
shapes,matrix,
positioning,
shapes.geometric,
arrows,through, 
}
\usepackage{tikz-cd}
\usepackage{indentfirst}

\newtheorem{theorem}{Theorem}[section]
\newtheorem{corollary}[theorem]{Corollary}
\newtheorem{proposition}[theorem]{Proposition}
\newtheorem{lemma}[theorem]{Lemma}

\theoremstyle{definition}

\newtheorem{remark}[theorem]{Remark}
\newtheorem{example}[theorem]{Example}

\newtheorem{definition}[theorem]{Definition}

\def\bigmid{\ \rule[-3.5mm]{0.1mm}{9mm}\ }

\newcommand{\CC}{{\mathbb C}}

\newcommand{\RR}{{\mathbb R}}
\newcommand{\ZZ}{{\mathbb Z}}
\newcommand{\NN}{{\mathbb N}}
\newcommand{\QQ}{{\mathbb Q}}

\newcommand{\ton}{{\otimes_{\Lambda_{\geq 0}}}}
\newcommand{\shn}{{SH^{S^1, -, >L}_M(K)}}
\newcommand{\lr}{{\,\,\longrightarrow\,\,}}
\pgfplotsset{compat=1.18}

\begin{document}

\title[$S^1$-equivariant relative symplectic cohomology]{$S^1$-equivariant relative symplectic cohomology and relative symplectic capacities}
\author[Jonghyeon Ahn]{Jonghyeon Ahn}
\newcommand{\Addresses}{{
\bigskip
\bigskip
\footnotesize
\textsc{Department of Mathematics,
University of Illinois Urbana-Champaign, Urbana, IL, 61801, USA.}\par\nopagebreak
\textit{E-mail address}: \texttt{ja34@illinois.edu}}}

\date{}
\begin{abstract}
    In this paper, we construct an $S^1$-equivariant version of the relative symplectic cohomology developed by Varolgunes. As an application, we construct a relative version of Gutt-Hutchings capacities and a relative version of symplectic (co)homology capacity. We will see that these relative symplectic capacities can detect the diplaceability and the heaviness of a compact subset of a symplectic manifold. We compare the first relative Gutt-Hutchings capacity and the relative symplectic (co)homology capacity and prove that they are equal to each other under a convexity assumption.

\end{abstract}

\maketitle

\tableofcontents

\section{Introduction}
\subsection{Motivation}
The symplectic (co)homology of a compact symplectic manifold with contact type boundary was introduced by Floer and Hofer in \cite{fh} and developed in papers with Cieliebak and Wysocki in \cite{cfh,cfhw, fhw}. A different version of symplectic (co)homology was later introduced by Viterbo in \cite{vi}. This work also contains an $S^1$-equivariant version of symplectic (co)homology $SH^{S^1}(W)$ which was later developed by Bourgeois and Oancea in \cite{bo09, bo13, bo} motivated, in part by its connections to contact homology. %Gutt(\cite{gutt}), Gutt and Hutchings(\cite{gh}), Ginzburg and Shon(\cite{gs}) and so on. 
Recently, Varolgunes introduced a relative version of symplectic cohomology $SH_M(K)$ in \cite{v,vt} which is defined for any subset $K$ of a closed symplectic manifold $(M,\omega)$. Different relative versions of symplectic cohomology were also introduced by Groman in \cite{g} and Venkatesh in \cite{ven}. Varolgunes's relative symplectic cohomology has been applied and extended in several recent works such as \cite{dgpz}, \cite{msv}, \cite{s} and \cite{s24}. In this paper, we develop and apply an $S^1$-equivariant version of Varolgunes's relative symplectic cohomology.
\subsection{Results}
Let 
\begin{align*}
 \Lambda  = \left\{ \sum_{i=0}^\infty c_i T^{\lambda_i} \bigmid c_i \in \QQ, \lambda_i \in \RR\,\,\text{and}\,\, \lim_{i \to \infty} \lambda_i = \infty \right\}   
\end{align*}
 be the Novikov field. The Novikov ring $\Lambda_{\geq 0}$ is the subring of $\Lambda$ given by 
 \begin{align*}
   \Lambda_{\geq 0}  = \left\{ \sum_{i=0}^\infty c_i T^{\lambda_i} \in \Lambda \bigmid \lambda_i \geq 0 \right\}.  
 \end{align*}
Let $(M, \omega)$ be a closed symplectic manifold and $K \subset M$ be a compact domain\footnote{By domain, we mean a submanifold  with boundary of codimension zero.} with contact type boundary. Briefly speaking, $SH^{S^1}_{M}(K)$ is graded module over the Novikov ring $\Lambda_{\geq 0}$ and is defined by
\begin{align*}
    SH^{S^1}_{M}(K) = H \left( \widehat{\varinjlim_{H \in \mathcal{H}_K^{\text{Cont}}}} CF^{S^1}_w(H) \right)
\end{align*}
where $ CF^{S^1}_w(H) = \Lambda_{\geq 0}[u] \ton CF_w(H)$ is the weighted $S^1$-equivariant Floer complex of $H$ and $\mathcal{H}_K^{\text{Cont}}$ is the set of suitably chosen Hamiltonian functions. The $S^1$-equivariant relative symplectic cohomology of $K$ in $M$ with coefficient in $\Lambda$ is defined by
\begin{align*}
    SH^{S^1}_M(K ; \Lambda) = SH^{S^1}_M(K ) \ton \Lambda.
\end{align*}
Precise definition will be provied in \S2 and \S3.

Our first results relate $SH^{S^1}_M(K)$ to the $SH^{}_M(K)$. Arguing as in \cite{bo}, we get the following.
\begin{theorem}\label{similars1} 
Let $(M, \omega)$ be a closed symplectic manifold.  Let $K \subset M$ be compact domain with contact type boundary. 
\begin{enumerate}[label=(\alph*)]
 \item (Proposition \ref{gysin}) There exists a Gysin exact triangle
    \begin{center}
    \includegraphics[scale=1.21]{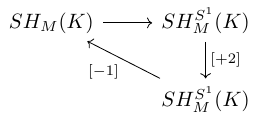}
    \end{center}
        \item (Corollary \ref{gysinaction}) If $\omega|_{\pi_2(M)} = 0$, then there exists a Gysin exact triangle
    \begin{center}
        \begin{tikzcd}
    SH^{>L}_M(K) \arrow{r} &SH^{S^1,>L}_M(K) \arrow{d}{[+2]} \\
    &SH^{S^1,>L}_M(K) \arrow{lu}{[-1]}    
\end{tikzcd}
    \end{center}
    for each $L \in \RR$. Here, $SH^{>L}_M(K)$ is the relative symplectic cohomology of $K$ in $M$ generated by Hamiltonian orbits of action greater than $L$ and $SH^{S^1, >L}_M(K)$ is similarly defined. \\

    \item (Proposition \ref{spectral}) There exists a spectral sequence $E_r^{p,q}(M, K)$ converging to $SH^{S^1}_M(K)$ such that its second page is given by
    \begin{align*}
        E_2^{p,q}(M,K) \cong H^p(BS^1 ; \Lambda_{\geq 0}) \otimes SH^q_M(K) 
    \end{align*}
    where $BS^1$ stands for the classifying space of $S^1$, that is, $BS^1 = \mathbb{CP}^{\infty}$.
\end{enumerate}
  Moreover, (a), (b) and (c) still hold if we replace the coefficient ring $\Lambda_{\geq0}$ by $\Lambda$.\\
\end{theorem}
\begin{corollary}\label{bothdie}
    Let $(M, \omega)$ be a closed symplectic manifold and let $K \subset M$ be compact domain with contact type boundary. Then relative symplectic cohomology $SH_M(K) $ vanishes if and only if $S^1$-equivariant symplectic cohomology $SH^{S^1}_M(K)$ vanishes. Also, $SH_M(K ; \Lambda) = 0$ if and only if $SH^{S^1}_M(K ; \Lambda) = 0$.
\end{corollary}
\begin{proof}
If $SH^{S^1}_M(K) = 0$, then the Gysin sequence given in (a) of Theorem \ref{similars1} implies that $SH_M(K) = 0$. Conversely, if $SH_M(K) = 0$, then the second page of the spectral sequence in (c) of Theorem \ref{similars1} vanishes and hence $SH^{S^1}_M(K) = 0$. The statement concerning coefficient in $\Lambda$ can be proved analogously. \\
\end{proof}
Next we consider the features of Varolgunes's relative symplectic cohomology, $SH_M(K)$, that are inherited by $SH^{S^1}_M(K)$. One of the remarkable properties of this relative symplectic cohomolgy is that it detects the displaceability of compact subsets. A subset $K$ of $M$ is said to be  \textbf{displaceable} if there exists a Hamiltonian diffeomorphism $\phi : M \to M$ such that $\phi(K) \cap K = \emptyset$.
\begin{theorem}[Varolgunes \cite{vt}]\label{dis}
    Let $(M, \omega)$ be a closed symplectic manifold and $K \subset M$ be a compact subset. If $K$ is displaceable, then $SH_M(K; \Lambda) = 0$.
\end{theorem}
Corollary \ref{bothdie} implies that $S^1$-equivariant relative symplectic cohomology inherits the ability to detect displaceability for compact domains with contact type boundary.
\begin{corollary}
   Let $(M, \omega)$ be a closed symplectic manifold and let $K \subset M$ be a compact domain with contact type boundary. If $K$ is displaceable, then $SH^{S^1}_M(K; \Lambda) = 0$.\\ \qed
\end{corollary}

%\begin{proof}
 %   If $K$ is displaceable, then $SH_M(K ; \Lambda) = 0$ by Theorem \ref{dis}. Since $SH_M(K; \Lambda)$ being zero is equivalent to $SH^{S^1}_M(K;\Lambda)$ being zero by Corollary \ref{bothdie}, we have $SH^{S^1}_M(K;\Lambda)=0$. \\
%\end{proof}
Another remarkable feature of Varolgunes's relative symplectic cohomology is that it satisfies a Mayer-Vietoris property under suitable assumptions.
Let $K_1$ and $K_2$ be a compact subset of $M$. We say that $K_1$ and $K_2$ \textbf{Poisson-commute} if there exist Poisson-commuting smooth functions $f_1, f_2 : M \to [0, 1]$ such that $f_1^{-1}(0) = K_1$ and $f_2^{-1}(0) = K_2$. 
\begin{theorem}[Varolgunes \cite{v}]
    Let $(M, \omega)$ be a closed symplectic manifold and $K_1, K_2 \subset M$ be compact subsets. If $K_1$ and $K_2$ Poisson-commute, then there exsits a Mayer-Vietoris exact triangle as follows.
    \begin{align}
        \includegraphics[scale=1.2]{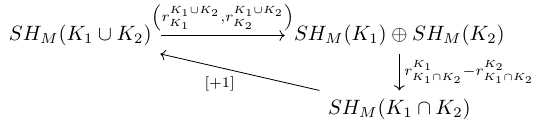}
    \end{align}
    Moreover, this Mayer-Vietoris sequence still holds if we replace coefficient ring $\Lambda_{\geq 0}$ by $\Lambda$.\\
\end{theorem}
The maps $r^{K_1\cup k_2}_{K_1}$, $r^{K_1\cup k_2}_{K_2}$, $r^{K_1}_{K_1 \cap K_2}$ and $r^{K_2}_{K_1 \cap K_2}$ are \textit{restriction maps} defined in \cite{v}. Generally speaking, if $K_1 \subset K_2$, there exists a restriction map $r^{K_1}_{K_2} : SH_M(K_1) \lr SH_M(K_2)$. Note that restriction maps are induced by continuation maps and hence they increase the action. 
The Mayer-Vietoris property still holds for $S^1$-equivariant relative symplectic cohomology. For $SH^{S^1}_M(K_1 \cup K_2)$ and $SH^{S^1}_M(K_1 \cap K_2)$ to make sense, let us make a following definition.
 \begin{definition}
        Let $(M, \omega)$ be a closed symplectic manifold and let $K_1, K_2 \subset M$ be compact domains with contact type boundary. We say that $(K_1, K_2)$ is a \textbf{contact pair} in $M$ if $K_1 \cup K_2$ and $K_1 \cap K_2$ are compact domains with contact type boundary.
       \end{definition}

\begin{theorem}[Proposition \ref{mvs1}]
     Let $(M, \omega)$ be a closed symplectic manifold and $K_1, K_2 \subset M$ be a compact domains with contact type boundaries. If $(K_1, K_2)$ is a contact pair and $K_1$ and $K_2$ Poisson-commute, then there exists a Mayer-Vietoris exact triangle for $S^1$-equivariant relative symplectic cohomology pertaining to $K_1$ and $K_2$.
    \begin{align*}
        \includegraphics[scale=1.23]{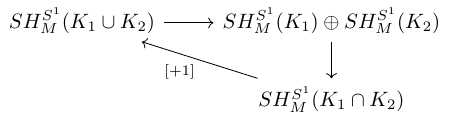}
    \end{align*}
   Moreover, this triangle stays exact if we replace coefficient ring $\Lambda_{\geq 0}$ by $\Lambda$. \\
\end{theorem}
Relative symplectic cohomology is also developed and applied by Dickstein, Ganor, Polterovich and Zapolsky in \cite{dgpz}. One of the results established is that when the ambient symplectic manifold is symplectically aspherical\footnote{We say that a symplectic manifold $(M, \omega)$ is \textbf{symplectically aspherical} if $\omega|_{\pi_2(M)} = 0$ and $c_1(TM)|_{\pi_2(M)} = 0$ where $c_1(TM)$ is the first Chern class of the tangent bundle $TM$ of $M$.} the relative symplectic cohomology $SH_M(K; \Lambda)$ does not depend on $K$ under some natural conditions. More precisely, they prove the following.
 %In \cite{dgpz}, Dickstein, Ganor, Polterovich and Zapolsky proved that sometimes the relative symplectic cohomology $SH_M(K; \Lambda)$ does not depend on the ambient manifolds $M$.
\begin{theorem}[Dickstein, Ganor, Polterovich and Zapolsky \cite{dgpz}]
   Let $(M, \omega)$ be a closed symplectic manifold which is symplectically aspherical. Let $K \subset M$ be a compact domain with contact type boundary. If $\partial K$ is index-bounded and the map $\pi_1(\partial K) \to \pi_1(M)$ induced by the inclusion $\partial K \xhookrightarrow{} M$ is injective\footnote{We usually say that $\partial K$ is \textbf{incompressible} in M.}, then there exists a canonical isomorphism 
   \begin{align*}
       SH_M(K; \Lambda) \cong SH(K; \Lambda)
   \end{align*}
   where $SH(K; \Lambda)$ stands for the symplectic cohomology of $K$ as a symplectic manifold with contact type boundary.\\
\end{theorem}
A similar result holds for $S^1$-equivariant relative symplectic cohomology.
\begin{theorem}[Proposition \ref{relcl}]\label{identity}
   Let $(M, \omega)$ be a closed symplectic manifold which is symplectically aspherical. Let $K \subset M$ be a Liouville domain. If $\partial K$ is index-bounded and the map $\pi_1(\partial K) \to \pi_1(M)$ induced by the inclusion is injective, then there is an isomorphism
    \begin{align}\label{isom}
        SH^{S^1}_M(K ; \Lambda) \cong SH^{S^1}(K; \Lambda)
    \end{align}
    where $SH^{S^1}(K; \Lambda)$ stands for the $S^1$-equivariant symplectic cohomology of $K$ as a symplectic manifold with contact type boundary. Moreover, we have a following commutative diagram
    \begin{center}
    \includegraphics[scale=1.2]{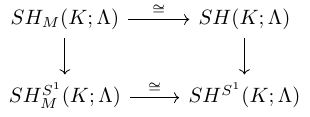}
    \end{center}
    Here, top horizontal arrow corresponds to the canonical isomorphism from \cite{dgpz} and the vertical maps are both induced by the chain level map $x \mapsto 1 \otimes x$.\\
\end{theorem}
Given Theorem \ref{identity}, it is natural to ask if the identity \eqref{isom} holds for every compact domain K with  contact type boundary. That is, one might ask if $SH^{S^1}_M(K ; \Lambda)$ is truly relative. The following example, which follows from a calculation in the thesis of Varolgunes, \cite{vt}, confirms that it is.

\begin{example}
    Let $(S^2,\omega)$ be the 2-dimensional sphere equipped with an area form $\omega$ with total area 1. Let $D_{\Delta} \subset S^2$ be a smooth disk of area $\Delta$. 
    \begin{itemize}
        \item The symplectic cohomology $SH(D_{\Delta};\Lambda)$ and $S^1$-equivariant symplectic cohomology $SH^{S^1}(D_{\Delta}; \Lambda)$ of $D_{\Delta}$ as a Liouville domain are zero regardless of its size.\\
        \item The relative symplectic cohomology $SH_{S^2}(D_{\Delta} ; \Lambda)$ is given by 
        \begin{align*}
             SH_{S^2}(D_{\Delta}; \Lambda) = \begin{cases}
            0 &\text{if}\,\, \Delta < \frac{1}{2}\\
            \Lambda \oplus \Lambda &\text{if}\,\, \Delta \geq \frac{1}{2}.
        \end{cases}\\
        \end{align*}
        \item The $S^1$-equivariant relative symplectic cohomology $SH^{S^1}_{S^2}(D_{\Delta}; \Lambda)$ is given by
        \begin{align*}
        SH^{S^1}_{S^2}(D_{\Delta}; \Lambda) = \begin{cases}
            0 &\text{if}\,\, \Delta < \frac{1}{2}\\
            \Lambda[u] \oplus \Lambda[u] &\text{if}\,\, \Delta \geq \frac{1}{2},
        \end{cases}
    \end{align*}
    where $u$ is a formal variable of degree 2. For more detailed exposition, see Proposition \ref{diskinsphere}.\\
    \end{itemize}
    
\end{example}
\subsection{Symplectic capacities}

One nice application of symplectic cohomology is to construct a symplectic capacity. Let us recall the definition of symplectic capacity. A \textbf{symplectic capacity} $c$ assigns to each symplectic manifold $(X, \omega)$ a number $c(X, \omega) \in [0, \infty]$ satisfying
\begin{itemize}
    \item (Monotonicity) if $(X, \omega) $ and $(X', \omega')$ have the same dimension and there exists a symplectic embedding $(X, \omega) \hookrightarrow (X', \omega')$, then $c(X, \omega) \leq c(X', \omega')$, and\\
    \item (Conformality) if $r > 0$, then $c(X, r \omega) = r c(X, \omega)$.\\
    
\end{itemize}

We can also define relative symplectic capacity as follows. 
\begin{definition}
Let $(M, \omega)$ be a symplectic manifold and let $K \subset M$ be a subset. A \textbf{relative symplectic capacity} $c$ assigns to each triple $(M, K, \omega)$ a number $c(M, K, \omega) \in [0, \infty]$ satisfying
\begin{itemize}
    \item (Monotonicity) if $(M, \omega) $ and $(M', \omega')$ have the same dimension and if there exists a symplectic embedding $\phi : (M, \omega) \hookrightarrow    (M', \omega')$ such that $\text{int}(\phi(K)) \subset K'$, then $c(M, K, \omega) \leq c(M', K', \omega')$, and\\
    \item (Conformality) if $r > 0$, then $c(M, K, r \omega) = r c(M, K, \omega)$.
    
\end{itemize}
We will usually drop the symplectic form $\omega$ in $c(M, K, \omega)$ if it is clear from the context. \\
\end{definition}
\subsubsection{Basic properties}

The first symplectic capcity defined using symplectic homology can be found in the paper \cite{fhw} of  Floer, Hofer and Wysocki. They defined symplectic capacities of open sets in $\CC^n$, which is usually called a symplectic homology capacity. We denote this symplectic homology capacity by $c^{SH}$. This definition is extended to a compact symplectic manifold with contact type boundary $W$ also utilizing the symplectic homology $SH(W)$ of $W$. This generalization is well explained in \cite{gs}. Gutt and Hutchings constructed in \cite{gh} a sequence of symplectic capacities, denoted by $c_k^{GH}$ for $k \geq 1$, of Liouville domains as an application of $S^1$-equivariant symplectic homology. Here, we show that relative symplectic cohomology and $S^1$-equivairant relative symplectic cohomology can be used, in an analogous way, to define relative symplectic capacities. 
\begin{theorem}
   Let $(M, \omega)$ be a closed symplectic manifold which is symplectically aspherical. Let $K \subset M$ be a compact domain with contact type and index-bounded boundary.
    \begin{enumerate}[label=(\alph*)]
        \item (Proposition \ref{ghc}) There exists a relative Gutt-Hutchings capacity $c_k^{GH}(M, K)$ for each $k=1,2,3,\cdots$.\\
        \item (Proposition \ref{shc}) There exists a relative symplectic (co)homology capacity $c^{SH}(M,K)$.\\
    \end{enumerate}
\end{theorem}
Basic properties of $c_k^{GH}(M,K)$ and $c^{SH}(M, K)$ are given as follows.
\begin{theorem}\label{propertyofcap}
Let $(M, \omega)$ be a closed symplectic manifold which is symplectically aspherical. Let $K \subset M$ be a compact domain with contact type and index-bounded boundary.
\begin{enumerate}[label=(\alph*)]
    \item (Proposition \ref{ghc}) $c_k^{GH}(M,K) \leq c_{k+1}^{GH}(M,K)$ for all $k \geq 1$.\\
    \item (Proposition \ref{disjoint}, \ref{disjointsh}) If $K_1, \dots, K_{\ell} \subset M$ are disjoint Liouville domain with index-bounded boundaries, then
    \begin{align*}
        c_k^{GH}\left(M,\,\, \coprod_{i=1}^l K_i\right) = \max\left\{c_k^{GH}(M, K_1), \cdots, c_k^{GH}(M, K_{\ell}) \right\}
    \end{align*}
    for each $k\geq 1$ and 
    \begin{align*}
        c^{SH}\left(M,\,\, \coprod_{i=1}^l K_i\right) = \max\left\{c^{SH}(M, K_1), \cdots, c^{SH}(M, K_{\ell}) \right\}.\\
    \end{align*}
    
    \end{enumerate}    
\end{theorem}
       
       \begin{theorem}[Proposition \ref{nondijoint}]\label{nondisjointintro}
       Let $(M, \omega)$ be a closed symplectic manifold which is symplectically aspherical. Let $(K_1, K_2)$ be a contact pair in $M$. Assume that $\partial K_1, \partial K_2, \partial(K_1 \cup K_2)$ and $\partial (K_1 \cap K_2)$ are index-bounded and that $K_1$ and $K_2$ satisfy descent.
               \begin{enumerate}[label=(\alph*)]
                   \item If the restriction map 
                \begin{align*}
                    r^{K_1}_{K_1 \cap K_2} : SH_M(K_1 ; \Lambda) \lr SH_M(K_1 \cap K_2; \Lambda) 
                \end{align*}
                is injective, then $c^{SH}(M, K_1) = c^{SH}(M, K_1 \cap K_2) $. \\
                %\item If the restriction map 
                %\begin{align*}
                 %   r^{K_2}_{K_1 \cap K_2} : SH_M(K_2 ; \Lambda) \lr SH_M(K_1 \cap K_2; \Lambda) 
                %\end{align*}
                %is injective, then $c^{SH}(M, K_2) = c^{SH}(M, K_1 \cap K_2) $.\\
                \item If the restriction map 
                \begin{align*}
                    r^{K_1 \cup K_2}_{K_1} : SH_M(K_1 \cup K_2; \Lambda)  \lr SH_M(K_1 ; \Lambda)  
                \end{align*}
                is injective, then $c^{SH}(M, K_1) = c^{SH}(M, K_1 \cup K_2)$.\\
                %\item If the restriction map 
               % \begin{align*}
                %    r^{K_1 \cup K_2}_{K_2} : SH_M(K_1 \cup K_2; \Lambda)  \lr SH_M(K_2 ; \Lambda)  
                %\end{align*}
               % is injective, then $c^{SH}(M, K_2) = c^{SH}(M, K_1 \cup K_2)$.\\
               \end{enumerate}
       \end{theorem}
 The precise definition of \textit{descent} will be given later in \S 2 but roughly speaking, descent is the position of two compact subsets that makes them satisfy Mayer-Vietoris property.

%\begin{corollary}
 %   Let $(M, \omega)$ be a closed symplectic manifold which is symplectically aspherical. Let $K \subset M$ be a compact domain with contact type and index-bounded boundary.
  %  \begin{enumerate}[label=(\alph*)]
   % \item If $K$ is displaceable or stably-displaceable, then $c^{SH}(M, K) < \infty$. Equivalently, if $c^{SH}(M, K) = \infty$, then $K$ is neither displaceable nor stably-displaceable.\\
    %     \item If $c^{SH}(M, K) < \infty$, then $K$ is not heavy. Equivalently, if $K$ is heavy, then $c^{SH}(M, K) = \infty$.\\
     %    \item If $c^{SH}(M, K) < \infty$, then $K$ is not SH-heavy. Equivalently, if $K$ is SH-heavy, then $c^{SH}(M, K) = \infty$.\\
    %\end{enumerate}
%\end{corollary}
%\begin{proof}
 %   (a) If $K$ is displaceable, then $SH_M(K;\Lambda) = 0$ by Theorem \ref{dis}. Theorem \ref{finitezero} says that $c^{SH}(M, K) < \infty$. If $K$ is stably displaceable, then by Theorem \ref{enpo} $K$ is not heavy. Since heaviness is equivalent to non-vanishinig relative symplectic cohomology by Theorem \ref{strongheavy}, we can conclude that $c^{SH}(M, K) < \infty$ again by  Theorem \ref{finitezero}.\\
  %  (b) If $c^{SH}(M, K) < \infty$, then $SH_M(K;\Lambda)$ vanishes by Theorem \ref{finitezero}. The first statement of \ref{strongheavy} implies that $K$ is not heavy.\\
   % (c) If $c^{SH}(M, K) < \infty$, then $SH_M(K;\Lambda)$ vanishes by Theorem \ref{finitezero}. The first statement of \ref{strongheavy} implies that $K$ is not heavy. And the second statement of \ref{strongheavy} implies that $K$ is not SH-heavy.\\
%\end{proof}

\subsubsection{Comparison of relative symplectic capacities}
In \cite{gh}, Gutt and Hutchings claimed a conjecture that the $k$-th Gutt-Hutchings capacity is equal to the $k$-th Ekeland-Hofer capacity defined in \cite{eh} for each $k = 1, 2, 3, \cdots$. Let us denote $k$-th Ekeland-Hofer capcity by $c_k^{EH}$. It is proved by Abbondandolo and Kang in \cite{ak} and by Irie in \cite{i} that, in some cases,  $c^{SH} = c_1^{EH}$. So, it is interesting to compare $c_1^{GH}$ and $c^{SH}$.

\begin{theorem}\label{comparisonintro}
Let $(M, \omega)$ be a closed symplectic manifold which is symplectically aspherical. Let $K \subset M$ be a compact domain with contact type and index-bounded boundary.
    \begin{enumerate}[label=(\alph*)]
        \item (Lemma \ref{comm}) $c_1^{GH}(M, K) \leq c^{SH}(M, K)$.\\
        \item (Theorem \ref{shgheq}) If the canonical contact form on $\partial K$ is dynamically convex, then $c_1^{GH}(M, K) = c^{SH}(M, K)$.\\
    \end{enumerate}
\end{theorem}
In the original definition of the relative symplectic cohomology in \cite{v} and \cite{vt}, the ambient space $M$ is assumed to be a closed symplectic manifold and we followed this assumption so far. But this assumption can be extended because Sun proved in \cite{s} that the assumption for the ambient space can be generalized to being the completion of a compact symplectic manifold with contact type boundary. Therefore, for a Liouville domain with index-bounded boundary $K$, we can easily deduce that $c_k^{GH}(\widehat{K}, K) = c_k^{GH}(K)$ and $c^{SH}(\widehat{K}, K) = c^{SH}(K)$ where $c_k^{GH}(K)$ and $c^{SH}(K)$ are $k$-th Gutt-Hutchings capacity of $K$ and symplectic (co)homology capacity of $K$, respectively. This discussion proves the following corollary.
\begin{corollary}\label{generalintro}
    Let $W$ be a Liouville domain with index-bounded boundary. If the canonical contact form on $\partial W$ is dynamically convex, then $c_1^{GH}(W) = c^{SH}(W)$.  \\ \qed
\end{corollary}
\begin{remark}
    If we take a more careful look at the definition of $SH^{S^1}_M(K ; \Lambda)$ and the proof of Theorem \ref{comparisonintro}, we can in fact drop the index-boundedness assumption. Therefore, Corollary \ref{generalintro} still holds for general Liouville domains.\\
\end{remark}
\subsubsection{Heaviness}
As mentioned earlier, relative symplectic cohomology is related to the displaceability of a compact subset $K$. The displaceability of $K$ is also related to the notion of heaviness introduced by Entov and Polterovich in \cite{ep09}. We now consider the relationship between relative symplectic capacity and heaviness.
We begin with a brief review of heaviness. A detailed exposition can be found in \cite{ep09}. Let $(M, \omega)$ be a closed symplectic manifold and let $H : M  \to \RR$ be a Hamiltonian function on $M$. Consider the PSS isomorphism of $H$, introduced in \cite{pss},
\begin{align*}
    PSS^H : QH(M ; \Lambda) \lr HF(H ; \Lambda)
\end{align*}
where $QH(M ; \Lambda)$ is the quantum cohomology of $M$ and $HF(H ; \Lambda)$ is the Floer cohomology of $H$. For each nonzero $a \in QH(M ; \Lambda)$, the \textbf{spectral invariant} of $a$ is defined by
\begin{align*}
    c(a ; H) = \sup \left\{ L \in \RR \mid PSS^H(a) \in \text{Im} i_L \right\}
\end{align*}
where $i_L : HF^{>L}(H) \to HF(H)$ is a natural map induced by the inclusion. Here, $HF^{>L}(H)$ stands for the action filtration of $HF(H)$ generated by Hamiltonian orbits of $H$ of action greater than $L$. Also, the \textbf{homogenized spectral invariant} of a nonzero idempotent $a \in QH(M ; \Lambda)$ can be defined by
\begin{align*}
    \mu(a; H) = \lim_{\ell \to \infty} \frac{c(a, \ell H)}{\ell}.
\end{align*}
Using this homogenized spectral invariant, we can define the heaviness of a compact subset.
\begin{definition}[Entov and Polterovich \cite{ep09}]
    Let $(M, \omega)$ be a closed symplectic manifold and let $a \in QH(M; \Lambda)$ be a nonzero element.
    \begin{enumerate}[label=(\alph*)]
        \item A compact set $K \subset M$ is called $a$-\textbf{heavy} if $\mu(a ; H) \leq \displaystyle\max_{K} H$ for all $H \in C^{\infty}(M)$. If $a = 1$ is the unit, then we say $K$ is heavy.\\
        \item A compact set $K \subset M$ is called $a$-\textbf{superheavy} if $\mu(a ; H) \geq \displaystyle\min_{K} H$ for all $H \in C^{\infty}(M)$. If $a = 1$ is the unit, we say that it is \textbf{superheavy}.\\
    \end{enumerate} 
\end{definition}
 We say that $K$ is \textbf{stably displaceable} if $K \times Z_{S^1}$ is displaceable from itself in $M \times T^*S^1$ where $Z_{S^1}$ is the zero section of $T^*S^1$.
\begin{theorem}[Entov and Polterovich \cite{ep09}]\label{enpo}
    Let $(M, \omega)$ be a closed symplectic manifold and let $K \subset M$ be a compact subset.
    \begin{enumerate}[label=(\alph*)]
    \item If $K$ is superheavy, then it is heavy.\\
        \item If $K$ is heavy, then it is not stably displaceable.\\
        \item If $K$ is superheavy, then it intersects with every heavy set in $M$.\\
    \end{enumerate} 
\end{theorem}
%For a closed symplectic manifold $(M, \omega)$ and a compact subset $K \subset M$, Dickstein, Ganor, Polterovich and Zapolsky defined the quantum cohomology ideal-valued measure $\tau(K)$ of $K$ in \cite{dgpz} by
%\begin{align*}
%    \tau(K) = \bigcap_{K \subset U} \text{Ker}\left( r^M_{M - U} : SH_M(M ; \Lambda) \to SH_M(M - U ; \Lambda) \right)
%\end{align*}
%where $U$ runs over all open subsets of $M$ containing $K$. Here, $r^M_{M - U}$ is a restriction map defined in \cite{v}. Utilizing this quantum cohomology ideal-valued measure, they defined a new type of heaviness.
%\begin{definition}[Dickstein, Ganor, Polterovich and Zapolsky, \cite{dgpz}]
 %   Let $(M, \omega)$ be a closed symplectic manifold and let $K \subset M$ be a compact subset. If $\tau(K) \neq 0$, then $K$ is called \textbf{SH-heavy}.\\
%\end{definition}
In \cite{msv}, it is proved by Mak, Sun and Varolgunes that relative symplectic cohomology can give us a criterion for heaviness.
%there is a strong connection between these notions of heaviness.
\begin{theorem}[Mak, Sun and Varolgunes \cite{msv}]\label{strongheavy}
    Let $(M, \omega)$ be a closed symplectic manifold and let $K \subset M$ be a compact subset. Then $SH_M(K ; \Lambda) \neq 0$ if and only if $K$ is heavy.
   % \begin{enumerate}[label=(\alph*)]
      %  \item $SH_M(K ; \Lambda) \neq 0$ if and only if $K$ is heavy.\\
      %  \item If $K$ is SH-heavy, then $K$ is heavy.\\
    %\end{enumerate}
\end{theorem}
Following theorem implies that relative symplectic capacity can say something about heaviness. 
\begin{theorem}[Proposition \ref{finitesh}]\label{finitezero}
   Let $(M, \omega)$ be a closed symplectic manifold which is symplectically aspherical. Let $K \subset M$ be a compact domain with contact type and index-bounded boundary. If $SH_M(K ; \Lambda) = 0$, then $c^{SH}(M, K) < \infty$.
\end{theorem}
%\begin{theorem}[Proposition \ref{finitesh}]\label{finitezero}
  % Let $(M, \omega)$ be a closed symplectic manifold which is symplectically aspherical. Let $K \subset M$ be a Liouville domain with index-bounded boundary. Then
   % $c^{SH}(M, K) < \infty$ if and only if $SH_M(K ; \Lambda) = 0$.\\
%\end{theorem}
Combining results above, we can easily prove the following corollary.
\begin{corollary}
    Let $(M, \omega)$ be a closed symplectic manifold which is symplectically aspherical. Let $K \subset M$ be acompact domain with contact type and index-bounded boundary.
    \begin{enumerate}[label=(\alph*)]
\item If $K$ is displaceable, then $c^{SH}(M, K) < \infty$.  Equivalently, if $c^{SH}(M, K) = \infty$, then $K$ is not displaceable.\\
\item If $K$ is not heavy, then $c^{SH}(M, K) < \infty$. Equivalently, if $c^{SH}(M, K) = \infty$, then $K$ is heavy.\\
\item If $K$ is stably displaceable, then $c^{SH}(M, K) < \infty$. Equivalently, if $c^{SH}(M, K) = \infty$, then $K$ is not stably displaceable.\\ \qed

    \end{enumerate}
\end{corollary}
%\begin{proof}
   % (a) If $K$ is displaceable, then $SH_M(K;\Lambda) = 0$ by Theorem \ref{dis}. Then Theorem \ref{finitezero} implies that $c^{SH}(M, K) < \infty$.\\
   % (b) If $K$ is not heavy, then by Theorem \ref{strongheavy}, $SH_M(K;\Lambda) = 0$. Then Theorem \ref{finitezero} also tells us that $c^{SH}(M, K) < \infty$.\\
  %  (c) If $K$ is stably displaceable, then $K$ is not heavy by Theorem \ref{enpo}. Then (b) implies the desired result.\\
  
%\end{proof}
\subsection{Organization of the paper} In section 2, we will briefly review $S^1$-equivariant symplectic cohomology and relative symplectic cohomology. In section 3, the definition of $S^1$-equivariant symplectic cohomology will be introduced and some basic properties of it will be provided. In section 4, the background for the relative symplectic capacity will be built. And at the end of section 4, we will compare the first relative Gutt-Hutchings capacity and the relative symplectic (co)homology capacity.\\

\noindent\textbf{Acknowledgements.} The author thanks Ely Kerman for motivating this project and for his helpful discussions and valuable comments. \\

\section{Preliminaries}
\subsection{Conventions} We put our conventions together here for Hamiltonian Floer theory. A more detailed exposition will be provided later. Let $(M, \omega)$ be a symplectic manifold and $J$ be a generic almost complex structure that is compatible with $\omega$. For any Hamiltonian function $H : S^1 \times M \to \RR$, we denote the Hamiltonian vector field of $H$ by $X_H$.
\begin{itemize}
    \item (Riemannian metric) $g(v, w) = \omega (v, Jw)$.\\
    \item (Hamiltonian equation) $\iota_{X_H} \omega = d H$.\\
    \item (Action functional) For a contractible Hamiltonian orbit $x$, let $\Tilde{x}$ be a disk capping of $x$. The action functional $\mathcal{A}_H(x,\Tilde{x})$ of $(x, \Tilde{x})$ is defined by
    \begin{align*}
        \mathcal{A}_H(x,\Tilde{x}) = \int_{D^2} \Tilde{x}^* \omega + \int_{S^1} H_t(x(t)) dt.\\
    \end{align*}
    If $\omega|_{\pi_2(M)}=0$, then $ \mathcal{A}_H(x,\Tilde{x}) =  \mathcal{A}_H(x)$ does not depend on the choice of a disk capping $\Tilde{x}$ of $x$.
    \item (Floer equation) $\displaystyle\frac{\partial u}{\partial s}(s,t) + J(u(s,t))\left(\frac{\partial u}{\partial t}(s,t) - X_H(u(s,t))\right) = 0$ for a smooth map $u : \RR \times S^1 \to M$.\\
    \item (Grading) Let $x$ be a contractible Hamiltonian orbit and $\Tilde{x}$ be its disk capping. The symplectic trivialization along $\Tilde{x}$ determines a unique (up to homotopy) trivialization along $x$. This can be used to define the Conley-Zehnder index $\text{CZ}(x, \Tilde{x})$ of $(x, \Tilde{x})$. Grading of $(x, \Tilde{x})$ is given by $\text{CZ}(x, \Tilde{x}).$ If $c_1(TM)|_{\pi_2(M)}=0$, then $\text{CZ}(x, \Tilde{x}) = \text{CZ}(x)$ does not depend on the choice of a symplectic trivialization and is $\ZZ$-valued.\\
    \end{itemize}
\begin{remark}\label{convention} Here are some consequences that we want to emphasize due to the conventions above.
    \begin{enumerate}[label=(\alph*)]
        \item Let $x$ and $y$ be contractible orbits of $H$ and let $\Tilde{x}$ and $\Tilde{y}$ be disk cappings of $x$ and $y$, respectively. If there exists a Floer trajectory $u$ such that $\displaystyle\lim_{s \to -\infty} u(s,t) = x(t)$, $\displaystyle\lim_{s \to \infty} u(s,t) = y(t)$ and $[\Tilde{x}\,\#\,u] = - [\Tilde{y}]$ where $\#$ denotes the connected sum and the bracket denotes the homotopy class, then $\mathcal{A}_H(x,\Tilde{x}) \leq \mathcal{A}_H(y,\Tilde{y})$ and $\text{CZ}(x, \Tilde{x}) < \text{CZ}(y, \Tilde{y})$.\\
       \item If $H$ is $C^2$-small, then the Floer equation becomes
        \begin{align*}
            \frac{du}{ds} = J(u)X_H(u) = \nabla_g H(u)
        \end{align*}
        where $\nabla_g H$ is the gradient vector field of $H$ with respect to the Riemannian metric $g$. Therefore, the Floer theory in this case becomes the Morse theory of $-H$. Moreover, for critical point $x$ of $H$, choose $\Tilde{x}$ to be the constant capping of $x$. Then we have
        \begin{align*}
            \text{CZ}(x, \Tilde{x}) = n- \text{ind}(x, -H)
        \end{align*}
        where $\text{ind}(x, -H)$ is a Morse index of a critical point $x$ of $-H$.\\
    \end{enumerate}
\end{remark}
\subsection{$S^1$-equivariant symplectic cohomology}
In this subsection, we briefly recall the definition of $S^1$-equivariant symplectic cohomology and its properties. Detailed exposition can be found in \cite{bo}, \cite{gh} and \cite{vi}. Let $(M, \omega)$ be a compact \textbf{symplectic manifold with contact type boundary}, that is, there exists a Liouville vector field $X$ defined on a neighborhood of $\partial M$ satisfying $\mathcal{L}_X \omega = \omega$ and $X$ is transvere to $\partial M$. Let $\lambda = \iota_X \omega$. Then $\lambda_M := \lambda|_{\partial M}$ is a contact form on $\partial M$. If $\lambda$ is defined on $M$ and $d \lambda = \omega$, then we call $(M, \omega, \lambda)$ a \textbf{Liouville domain}. The symplectic completion $\widehat{M}$ of $M$ is the symplectic manifold
\begin{align*}
    \widehat{M} = (M \amalg (\partial M \times [0, \infty]))/\sim
\end{align*}
with its symplectic form
\begin{align*}
    \widehat{\omega} = \begin{cases}
        \omega &\text{on}\,\, M\\
        d(e^{\rho} \lambda_M) & \text{on} \,\,\partial M \times [0, \infty).
        \end{cases}
\end{align*}
where $\rho$ is a coordinate on $[0, \infty)$. The equivalence relation $\sim$ is given by the diffeomorphism 
\begin{align*}
    \partial M \times [0, \infty) \to U, \,\,(p, \rho) \mapsto \phi^X_{\rho}(p)
\end{align*}
where $\phi^X_{\rho}$ is the flow of the Liouville vector field $X$ and $U$ is a neighborhood of $\partial M$ in $M$.
\begin{definition}
    Let $(M, \omega)$ be a compact symplectic manifold with contact type boundary. An \textbf{admissible Hamiltonian function} is a smooth function $H : S^1 \times \widehat{M} \to \RR$ satisfying the following conditions:
    \begin{enumerate}[label=(\alph*)]
        \item $H$ is negative and $C^2$-small on $S^1 \times M$.\\
        \item There exists $\rho_0 \geq 0$ such that $H(t, p, \rho) = \beta e^{\rho} +\beta'$ on $S^1 \times \partial M \times [ \rho_0, \infty)$ where $\beta \notin \text{Spec}(\partial M, \lambda_{ M})$. Here, $\text{Spec}(\partial M, \lambda_{ M})$ is the set of periods of Reeb orbits on $(\partial M, \lambda_{ M})$.\\
        \item $H(t, p, \rho)$ is $C^2$-close to $h(e^{\rho})$ on $S^1 \times \partial M \times [0, \rho_0]$ for some strictly convex and increasing function $h$. 
    \end{enumerate}
    We denote the set of all admissible Hamiltonian functions by $\mathcal{H}$.\\
\end{definition}
Note that the Hamiltonian vector field of $H \in \mathcal{H}$ on the cylinder is given by 
\begin{align*}
    X_H(p, \rho) = -h'(e^{\rho}) R_{\lambda_M}(p)
\end{align*}
where $R_{\lambda_M}$ is the Reeb vector field of $\lambda_M$. Hence, a 1-periodic Hamiltonian orbit of $H$ corresponds to a Reeb orbit of $(\partial M, \lambda_M)$ of period $h'(e^{\rho})$ and traversed in the opposite direction of the Reeb orbit.
Now we define $S^1$-equivariant Floer complex of $H \in \mathcal{H}$ (over $\QQ$). For simplicity, we assume that $(M, \omega)$ is symplectically atoroidal\footnote{A symplectic manifold $(M, \omega)$ is called \textbf{symplectically atoroidal} if $\int_{T^2} f^*\omega = 0$ for any smooth map $f : T^2 \to M$.}. Let $CF(H)$ be a usual Floer complex of $H$, that is,
\begin{align*}
    CF(H) = \bigoplus_{x \in \mathcal{P}(H)} \QQ \langle x \rangle
\end{align*}
where $\mathcal{P}(H)$ is the set of contractible 1-periodic orbits of $H$. Define
\begin{align*}
    CF^{S^1}(H) = \QQ[u] \otimes_{\QQ} CF(H)
\end{align*}
where $u$ is a formal variable of degree 2. The differential has the form
\begin{align*}
    d^{S^1} (u^k \otimes x) = \sum_{i=0}^k u^{k-i} \otimes \phi_i(x).
\end{align*}
In particular, $\phi_0$ is the Floer differential of $CF(H)$. For the precise definition of $\phi_i$, see \cite{bo}. The grading on $CF^{S^1}(H)$ is defined by 
\begin{align*}
    | u^k \otimes x | = \text{CZ}_{\tau}(x) - 2k
\end{align*}
where $\text{CZ}_{\tau}(x)$ is the Conley-Zehnder index of $x$ computed using a symplectic trivialization $\tau$. For $H_1, H_2 \in \mathcal{H}$ with $H_1 \leq H_2$, there is a continuation map $CF^{S^1}(H_1) \lr CF^{S^1}(H_2)$ induced by a monotone homotopy $H_s$ connecting $H_1$ and $H_2$. Note that continuation maps increase the action and preserve the grading.\\
\begin{definition}
    Let $(M, \omega)$ be a compact symplectic manifold with contact type boundary. The \textbf{$S^1$-equivariant symplectic cohomology} $SH^{S^1}(M)$ of $M$ is defined by 
    \begin{align*}
        SH^{S^1}(M) &= H\left( \varinjlim_{H \in \mathcal{H}} CF^{S^1}(H)\right)\\& = H\left( \varinjlim_{n \to \infty} CF^{S^1}(H_n)\right) 
    \end{align*}
    where $H_n \in \mathcal{H}$ and $H_1 \leq H_2 \leq H_3 \leq \cdots$ and 
    \begin{align*}
        \lim_{n \to \infty} H_n(p) = \begin{cases}
            0 \,\,&\text{if}\,\, p \in M\\
            \infty \,\,&\text{if}\,\, p \notin M.
        \end{cases}\\
    \end{align*}
\end{definition}
We can understand the symplectic cohomology better if we utilize the techniques due to Bourgeois and Oancea \cite{bo09}.
\begin{definition}[Bourgeois and Oancea \cite{bo09}]\label{mb}
    Let $(M, \omega)$ be a compact symplectic manifold with contact type boundary. An \textbf{admissible Morse-Bott Hamiltonian function} is an autonomous function $H : \widehat{M} \to \RR$ satisfying the following conditions.
    \begin{enumerate}[label=(\alph*)]
        \item $H$ is negative and $C^2$-small on $M$\\
        \item There exists $\rho_0 > 0$ such that $H(p, \rho) = \beta e^{\rho} + \beta'$ on $\partial M \times [\rho_0, \infty)$ where $\beta \notin \text{Spec}(\partial M, \lambda_M)$.\\
        \item $H(p, \rho)$ is $C^2$-close to $h(e^{\rho})$ on $\partial M \times [0, \rho_0]$ for some convex and increasing function satisfying $h^{''} - h^' >0$.
    \end{enumerate}
    We denote the set of all admissible Morse-Bott Hamiltonian functions by $\mathcal{H}^{MB}$.\\
\end{definition}
\begin{theorem}[Bourgeois and Oancea \cite{bo09}, Cieliebak, Floer, Hofer and Wysocki \cite{cfhw}]\label{reeb} 
    An admissible Morse-Bott Hamiltonian function $H \in \mathcal{H}^{MB}$ can be perturbed to an adimissible Hamiltonian function $H' \in \mathcal{H}$ whose 1-periodic orbits satisfy the following.
    \begin{enumerate}[label=(\alph*)]
        \item The constant periodic orbits of $H'$ are the critical points of $H$.\\
        \item For each Reeb orbit $\gamma$ of $(\partial M, \lambda_M)$, there are two nondegenerate Hamiltonian orbits $\gamma_{\text{Max}}$ and $\gamma_{\text{min}}$ of $H'$ such that for a given symplectic trivialization $\tau$ along $\gamma$,
        \begin{align*}
            - \text{CZ}_{\tau}(\gamma_{\text{Max}}) = \text{CZ}_{\tau}(\gamma) \,\, \text{and}\,\, - \text{CZ}_{\tau}(\gamma_{\text{min}}) = \text{CZ}_{\tau}(\gamma) +1 
        \end{align*}
        where $\text{CZ}_{\tau}$ stands for the Conley-Zehnder index calculated under the trivialization $\tau$.\\
    \end{enumerate}
\end{theorem}
\subsection{Relative symplectic homology}
In this subsection, a brief background on the relative symplectic cohomolgy will be provided. We refer the readers to \cite{dgpz}, \cite{v} and \cite{vt} for more details. The \textbf{Novikov field} $\Lambda$ is defined by
\begin{align*}
    \Lambda = \left\{ \sum_{i=0}^\infty c_i T^{\lambda_i} \bigmid c_i \in \QQ, \lambda_i \in \RR\,\,\text{and}\,\, \lim_{i \to \infty} \lambda_i = \infty \right\} 
\end{align*}
where $T$ is a formal variable. There is a valuation map $val : \Lambda \to \RR \cup \{\infty\}$ given by
\begin{align*}
    val (x) = 
    \begin{cases}
      \displaystyle\min_{i} \{\lambda_i \mid c_i \neq 0 \} \,&\text{if}\,\, x = \displaystyle\sum_{i=0}^\infty c_i T^{\lambda_i} \neq 0\\
      \infty \,&\text{if}\,\, x = 0.
    \end{cases}   
\end{align*}
For any $r \in \RR$, define $\Lambda_{>r} = val^{-1}((r,\infty])$ and $\Lambda_{\geq r} = val^{-1}([r,\infty])$. In particular, we call $\Lambda_{\geq 0}$ the \textbf{Novikov ring}. For any $\Lambda_{\geq 0}$-module $A$, we can define its \textbf{completion} $\widehat{A}$ by
\begin{align*}
    \widehat{A} = \varprojlim_{r \to 0} A \otimes \Lambda_{\geq 0} / \Lambda_{\geq r}.
\end{align*}
More concrete description of $\widehat{A}$ explained in \cite{v} is the following: Let $\{a_n\}$ be a sequence in $A$. This sequence $\{a_n\}$ is called a \textbf{Cauchy sequence} if, for every $r \geq 0$, there exists $N \in \NN$ such that $a_n - a_m \in T^r A$ for all $n, m > N$. Also, $\{a_n\}$ is said to \textbf{converge to zero} if, for every $r \geq 0$, there exists $N \in \NN$ such that $a_n \in T^r A$ for all $n > N$. Then the completion $\widehat{A}$ of $A$ is isomorphic to the set of all Cauchy sequences in $A$ modulo the set of all sequences converging to zero in $A$. 
\par Let $(M, \omega)$ be a closed symplectic manifold. Let $H : S^1 \times M \to \RR$ be a Hamiltonian function on $M$. Let $x \in \mathcal{P}(H)$ and $\Tilde{x}$ be a disk capping of $x$. The action of $(x, \Tilde{x})$ is defined by
\begin{align*}
    \mathcal{A}_H(x, \Tilde{x}) = \int_{D^2} \Tilde{x}^* \omega + \int_{S^1} H_t(x(t))dt.
\end{align*}
We can associate to $(x, \Tilde{x})$ an integer using the Conley-Zehnder index. Define an equivalence relation on the set of pairs of orbits and its capping by
\begin{align*}
    (x, \Tilde{x}) \sim (y, \Tilde{y}) \,\,\text{if and only if}\,\, x = y, \,\mathcal{A}_H(x, \Tilde{x}) = \mathcal{A}_H(y, \Tilde{y}) \,\,\text{and}\,\, CZ(x, \Tilde{x}) = CZ(y, \Tilde{y}).
\end{align*}
We denote the equivalence class of $(x, \Tilde{x})$ by $[x,\Tilde{x}]$ and denote the set of all equivalence classes by $\Tilde{\mathcal{P}}(H)$. The \textbf{weighted Floer complex of $H$} is defined by
\begin{align*}
    CF_w(H) = \left\{\sum_{i=1}^{\infty}c_i [x_i, \Tilde{x_i}] \bigmid c_i \in \QQ,  [x_i, \Tilde{x_i}] \in  \Tilde{\mathcal{P}}(H)\,\,\text{and}\,\, \lim_{i \to \infty} \mathcal{A}_H([x_i, \Tilde{x_i}]) = \infty \right\}
  \end{align*}
and is graded by the Conley-Zehnder index. Explicitly, the weighted Floer complex of $H$ in degree $k$ is 
\begin{align*}
    CF_w^k(H) = \left\{\sum_{i=1}^{\infty}c_i [x_i, \Tilde{x_i}] \in CF_w(H) \bigmid \text{CZ}\left([x_i, \Tilde{x_i}]\right) = k\,\,\text{for all}\,\,i=1,2,3,\cdots. \right\}
  \end{align*}
  For $x, y \in \mathcal{P}(H)$, let $\pi_2(M, x, y)$ be the set of homotopy classes of smooth maps from $\RR \times S^1$ to $M$ which are asymptotic to $x$ and $y$ at $-\infty$ and $\infty$, respectively. Consider, for two equivalence classes $[x, \Tilde{x}], [y, \Tilde{y}] \in \Tilde{\mathcal{P}}(H)$ and $A \in \pi_2(M,x,y)$, the moduli space $\mathcal{M}(H; [x, \Tilde{x}], [y, \Tilde{y}]; A)$ of Floer  trajectories of $H$ connecting $x$ and $y$. More precisely, $\mathcal{M}(H; [x, \Tilde{x}], [y, \Tilde{y}]; A)$ is the set of smooth maps $u : \RR \times S^1 \to M$ satisfying the following.
\begin{itemize}
    \item The homotopy class of $u$ represents $A \in \pi_2(M,x, y)$.\\
    \item $\displaystyle\frac{\partial u}{\partial s} + J(u) \left( \frac{\partial u}{\partial t} - X_{H_t}(u) \right) = 0$ where $J$ is a generic almost complex structure on $M$.\\
    \item $\displaystyle\lim_{s \to -\infty} u(s,t) = x(t)$ and $\displaystyle\lim_{s \to \infty} u(s,t) = y(t)$.\\
    \item $[y, \Tilde{y}] = [y, \Tilde{x}\,\#\, (-u)] = [y, \Tilde{x}\,\#\, (-A)] $ where $\#$ denotes the connected sum.\\
\end{itemize}
Note that if $u \in \mathcal{M}(H; [x, \Tilde{x}], [y, \Tilde{y}]; A)$, then $\mathcal{A}_H([y,\Tilde{y}]) \geq \mathcal{A}_H([x, \Tilde{x}])$ because we can easily show that 
\begin{align*}
    \frac{\partial }{\partial s} \mathcal{A}_H(u(s,t)) = \int_{\RR \times S^1} \omega \left(\frac{\partial u}{\partial s}, J(u) \frac{\partial u}{\partial s}\right) ds \wedge dt = \int_{\RR \times S^1} \left\|\frac{\partial u}{\partial s}\right\|_g^2 ds \wedge dt \geq 0.
    \end{align*}
Moreover, the topological energy of $u$ is given by
\begin{align*}
    E_{\text{top}}(u) :=   \int_{S^1} H_t(y(t)) dt - \int_{S^1} H_t(x(t)) dt + \omega(A) = \mathcal{A}_H([y,\Tilde{y}]) - \mathcal{A}_H([x, \Tilde{x}]).
\end{align*}
The \textbf{weighted Floer differential} of $CF_w(H)$ is given by
\begin{align*}
    d_w [x, \Tilde{x}] = \sum_{\substack{[y, \Tilde{y}] \in \Tilde{\mathcal{P}}(H) \\ A \in \pi_2(M, x, y)}} \#\mathcal{M}(H ; [x, \Tilde{x}], [y, \Tilde{y}]; A) T^{\mathcal{A}_H([y,\Tilde{y}]) - \mathcal{A}_H([x, \Tilde{x}])} [y, \Tilde{y}].
\end{align*}
And $\#\mathcal{M}(H ; [x, \Tilde{x}], [y, \Tilde{y}]; A)$ is a suitable virtual count defined in \cite{p}. For two Hamiltonian functions $H_1$ and $H_2$ with $H_1 \leq H_2$, choose a monotone homotopy $H_s$ of Hamiltonian functions connecting them. Then there is a weighted continuation map $CF_w(H_1) \lr CF_w(H_2)$ defined by the suitable count of Floer trajectories of $H_s$ connecting the orbits of $H_1$ and $H_2$. This continuation map is also weighted by $T^{E_{\text{top}}(u)}$ in a similar way that we define the weighted differential above. Note that continuation maps increase the actions and respect the grading. \\
\begin{remark}
    If $(M, \omega)$ is a closed symplectic manifold with $\omega|_{\pi_2(M)}=0$, we can simplify some of the definitions. For $x, y \in \mathcal{P}(H)$, the moduli space space $\mathcal{M}(H;x,y)$ is the set of smooth maps $u : \RR \times S^1 \to M$ satisfying the following.
\begin{itemize}
    \item $\displaystyle\frac{\partial u}{\partial s} + J(u) \left( \frac{\partial u}{\partial t} - X_{H_t}(u) \right) = 0$ where $J$ is a generic almost complex structure on $M$.\\
    \item $\displaystyle\lim_{s \to -\infty} u(s,t) = x(t)$ and $\displaystyle\lim_{s \to \infty} u(s,t) = y(t)$.\\
\end{itemize}
The weighted Floer differential of $CF_w(H)$ is given by
\begin{align*}
    d_w x = \sum_{y\in \mathcal{P}(H)} \#\mathcal{M}(H ; x, y) T^{\mathcal{A}_H(y) - \mathcal{A}_H(x)} y
\end{align*}
Here, $\#\mathcal{M}(H ; x, y)$ is a suitable virtual count defined in \cite{p}.\\

\end{remark}
\begin{remark}\label{clw}
    Let $(M, \omega)$ be a closed symplectic manifold with $\omega|_{\pi_2(M)}=0$. Let $H : S^1 \times M \to \RR$ be a Hamiltonian function and let $x \in \mathcal{P}(H)$. The weighted Floer cohomology of $H$ with coefficient in $\Lambda$ coincides with the classical Floer cohomology of $H$ with coefficient in $\Lambda$. The isomorphism between them is induced by
    \begin{align*}
        CF(H; \Lambda) \longrightarrow CF_w(H) \ton \Lambda, \,\,x \mapsto T^{\mathcal{A}_H(x)} x.
    \end{align*}
    One special case is when $H$ is $C^2$-small. Denote the classical Morse complex and weighted Morse complex of $H$ by $CM(H)$ and $CM_w(H)$, respectively. If $H$ is $C^2$-small, then every 1-periodic orbit of $H$ is a critical point of $H$ and $CF(H) = CM(-H)$ and $CF_w(H) = CM_w(-H)$. See Remark \ref{convention}. In this case, $\mathcal{A}_H(x) = H(x)$ for each critical point $x$ of $-H$ and 
        \begin{align*}
            CM(-H ; \Lambda) \lr CM_w(-H; \Lambda),\,\,x \mapsto T^{H(x)}x
        \end{align*}
        induces an isomorphism between classical and weighted Morse cohomology.\\
   
\end{remark}
\begin{definition}
   Let $(M, \omega)$ be a closed symplectic manifold. Let $K \subset M$ be a compact subset of $M$. A \textbf{$K$-admissible Hamiltonian function} is a smooth function $H : S^1 \times M \to \RR$ satisfying $H$ is negative and $C^2$-small on $S^1 \times K$. We denote the set of all $K$-admissible Hamiltonians by $\mathcal{H}_K$. \\
    
\end{definition}
Note that the cofinal family $\{H_n\}$ of $\mathcal{H}_K$ satisfies the following.
    \begin{itemize}\label{ham}
    \item $H_n$ is negative on $S^1 \times K$ for all $n=1,2,3,\cdots$.\\
    \item $H_1 \leq H_2 \leq H_3 \leq \cdots$.\\
    \item $\displaystyle\lim_{n \to \infty} H_n(t, x) = 0$ for $x \in K$ and $\displaystyle\lim_{n \to \infty} H_n(t, x) = \infty$ for $x \in M-K$.\\
\end{itemize}
Now we are ready to define relative symplectic cohomology.
\begin{definition}
     Let $(M, \omega)$ be a closed symplectic manifold. Let $K \subset M$ be a compact subset. The \textbf{relative symplectic chomology of $K$ in $M$} is defined by
    \begin{align*}
        SH_M(K) &= H\left(\widehat{\varinjlim_{H \in \mathcal{H}_K}}CF_w(H)\right)\\& = H\left(\widehat{\varinjlim_{n \to \infty}}CF_w(H_n)\right)
    \end{align*}
    where $\{H_n\}$ is a cofinal family of $\mathcal{H}_K$. We can also define the \textbf{relative symplectic cohomology of $K$ in $M$ with coefficient in $\Lambda$} by
   
\begin{align}\label{coeff}
    SH_M(K; \Lambda) &= 
    H \left(\widehat{\varinjlim_{n \to \infty}}CF_w(H_n) \ton \Lambda\right) \nonumber  \\&= H\left(\widehat{\varinjlim_{n \to \infty}}CF_w(H_n)\right) \ton \Lambda\\&= SH_M(K) \otimes_{\Lambda_{\geq 0}} \Lambda.\nonumber \\\nonumber
\end{align}

\end{definition}
\begin{remark}
    \begin{enumerate}[label=(\alph*)]
        \item It is proved in \cite{v} that the definition of relative symplectic cohomology does not depend on the choice of cofinal family of $\mathcal{H}_K$.\\ 
        \item In the definition of relative symplectic cohomology with coefficient in $\Lambda$, the reason why we can exchange homology and tensor is as follows. Since $\Lambda$ is the field of fractions of $\Lambda_{\geq 0}$, it is flat over $\Lambda_{\geq 0}$. Then the fact that $\text{Tor}^1_{\Lambda_{\geq 0}}\left(-, \Lambda\right) = 0$ and the universal coefficient theorem justify the commutativity of \eqref{coeff}.\\
    \end{enumerate}
\end{remark}
The definition of relative symplectic cohomology can be better understood if we impose some assumption on $\partial K$. Let us recall a definition.
\begin{definition}\label{indbdd}
Let $K \subset M$ be a compact domain with contact type boundary. Assume that $c_1(TM)|_{\pi_2(M)} = 0$. We say that $\partial K$ is \textbf{index-bounded} if, for each $k \in \ZZ$, the set of periods of contractible Reeb orbits of $\partial K$ of Conley-Zehnder index $k$ is bounded.  \\
\end{definition}
The index-boundedness of $\partial K$ in Definition \ref{indbdd} is well-defined because $c_1(\xi) = \iota^* c_1(TM)$ and $c_1(TM)|_{\pi_2(M)} = 0$ where $\iota : \partial K \to M$ is the inclusion and $\xi$ is the canonical contact structure on $\partial K$.

\begin{remark}\label{kill}
    %Let $(M, \omega)$ be a symplectic manifold and $D \subset M$ be hypersurface with contact type boundary. Suppose that $D$ is index-bounded, that is, for each $k \ZZ$, the set of periods of the contractible Reeb orbits of Conley-Zehnder index $k$ is bounded. Then the actions of upper orbits of the Hamiltonian function in Figure \ref{fig:completion} is strictly greater than those of lower orbits. 

   Let $(M,\omega)$ be a closed symplectically aspherical symplectic manifold. The Hamiltonian orbits of the Hamiltonian function given in Figure \ref{fig:completion} can be divided into two parts; upper orbits and lower orbits. If $K$ is a Liouville domain and $\partial K$ is index-bounded, then the completion process kills the upper orbits. To see this, let $\{H_n\}$ be a cofinal family of $\mathcal{H}_K$. Let $x \in \mathcal{P}_{H_n}$ be any orbit and $y \in \mathcal{P}_{H_{n+1}}$ be an upper orbit. Suppose that $x$ and $y$ are connceted by the continuation map and $u$ is a Floer trajectory connecting $x$ and $y$. Then
    \begin{align*}
        E_{\text{top}}(u) &= \int_{S^1} H_{n+1}(y(t)) dt - \int_{S^1} H_n(x(t)) dt + \omega(A) \\&= \int_{S^1} H_{n+1}(y(t)) dt - \int_{S^1} H_n(x(t)) dt + \int_{y} \theta - \int_x \theta
    \end{align*}
    where $\theta$ is a 1-form on the neighborhood of $\partial K$ satisfying $d \theta = \omega$. Here, we use the fact in \cite{gt} that every Floer trajectory lies in the union of $K$ and the cylindrical neighborhood of $\partial K$. Since $\displaystyle\int_{y} \theta - \int_x \theta$ is bounded due to the index-boundedness of $\partial K$ and we can choose a cofinal sequence $\{H_n\}$ so that the difference of $H_n$ and $H_{n+1}$ is large on $M - K$, the exponent of $T$ in the continuation map connecting $x$ and $y$ is greater than or equal to some positive constant for sufficiently large $n$. This can be removed by the completion. So, in the completed chain complex of relative symplectic cohomology, only lower orbits can survive.  \\  
\end{remark}

\begin{figure}
    \centering
    \caption{Completion process kills the upper orbits}
    \label{fig:completion}
    \begin{tikzpicture}
    \draw[domain=-1.4:2.3]
    plot (\x, {2.4*(1+tanh(5*(\x-0.5)))});
    \draw (-1.43,0.4)--(2.2,0.4);
    \draw[dotted] (-0.27,0.4)--(-0.27,0);
    \node at (-0.8,0.7) {$K$};
    \draw[dotted] (-0.3,0) ellipse (1.2 and 0.8);
    \draw[dotted] (1.25, 4.7) ellipse (1.2 and 0.8);
    \node at (2, 0) {Lower orbits};
    \node at (-1, 4.7){Upper orbits};
\end{tikzpicture}
\end{figure}
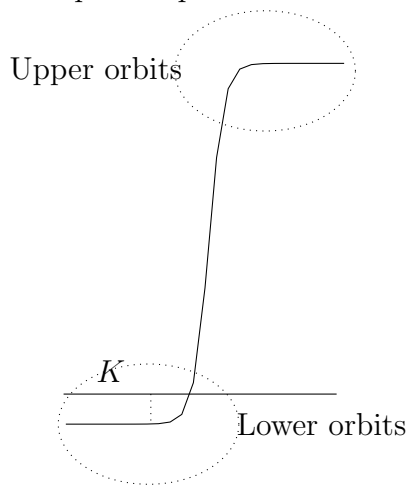
We have so far seen the definition of relative symplectic cohomology using only homological language. But once we equip ourselves with homotopical tools, we can view this in a different perspective. To this end, let us briefly review some homotopical ingredients that we need. For more detailed explanation, see \cite{dgpz}, \cite{v} and \cite{vt}. Let $n$ be a nonnegative integer and consider
\begin{align*}
    [0, 1]^n = \{(x_1, \cdots, x_n) \mid 0 \leq x_i \leq 1 \,\, \text{for all}\,\,i = 1, \cdots, n.\}.
\end{align*}
For $0 \leq k \leq n$, we call a subset of $[0,1]^n$ a \textbf{$k$-dimensional face} of $[0,1]^n$ if it is given by setting $n-k$ coordinates to either 0 or 1. In particular, we call a $0$-dimensional face a \textbf{vertex}. We denote the dimension of $F$ by $|F|$. For a face $F$ of $[0,1]^n$, the \textbf{initial vertex} of $F$ is the vertex which has the maximal number of zeros among vertices in $F$ and the \textbf{terminal vertex} of $F$ is the vertex which has the minimal number of zeros among vertices in $F$. In other words, the initial vertex of $F$ is the closest vertex from the origin in $F$ and the terminal vertex of $F$ is the farthest vertex from the origin in $F$. We denote the initial and terminal vertex of $F$ by ini$F$ and ter$F$, respectively. We say that two faces $F$ and $F'$ of $[0,1]^n$ are \textbf{adjacent} if ter$F$ = ini$F'$. If $F$ and $F'$ are adjacent, we write $F = F' \cdot F''$ to denote the situation where ini$F$ = ini$F'$, ter$F'$ = ini$F''$ and ter$F''$ = ter$F$.
\begin{definition}
    Let $R$ be a commutative ring.
    \begin{enumerate}[label=(\alph*)]
        \item An \textbf{$n$-cube (of chain complexes) over $R$} is given by a pair $$(\{C_v\}_{v \,\,\text{vertex of}\,\, [0,1]^n}, \{ f_F\}_{F \,\,\text{face of}\,\, [0,1]^n}) $$ where for every vertex $v$ of $[0,1]^n$, $C_v$ is a $\ZZ /2$-graded $R$-module and for every face $F$ of $[0,1]^n$, we have a graded module morphism $f_F : C_{\text{ini}F} \to C_{\text{ter}F}$ of degree $|F| + 1$ modulo 2. Morever, these maps are subject to the condition that
    \begin{align}\label{cubemap}
        \sum_{F= F' \cdot F''} (-1)^{|F'|} \text{sgn}(F', F'') f_{F''} f_{F'} = 0
    \end{align}
     for each face $F$. Here, sgn$(F', F'')$ is the sign of $(k, l)$-shuffle where $|F'| = k$ and $|F''| = l$.\\
     \item A \textbf{partially defined $n$-cube} is given by a collection of modules associated to some of vertices of $[0,1]^n$ and maps between them corresponding to some of faces of $[0,1]^n$ satisfying \eqref{cubemap} whenever it makes sense.\\
     \item An $n$-cube that agrees with the given partially defined cube $\mathcal{C}$ is called a \textbf{filling} of $\mathcal{C}$.\\
    \end{enumerate}
      
\end{definition}
\begin{definition}
    Let $\mathcal{C}$ and $\mathcal{C}'$ be $n$-cubes. A \textbf{map between $n$-cubes} $\mathcal{C}$ and $\mathcal{C}'$ is defined to be a filling of the partially defined $(n+1)$-cube where the $n$-dimensional faces given by $\{x_{n+1} = 0\}$ and $\{x_{n+1} = 1\}$ are the $n$-cubes $\mathcal{C}$ and $\mathcal{C}'$, respectively.\\
\end{definition}
\begin{example}\label{ex1}
\begin{enumerate}[label=(\alph*)]
    \item A 0-cube is a (co)chain complex.\\
    \item Let $C_1$ and $C_2$ be 0-cubes, namely, (co)chain complexes. Then a map between $C_1$ and $C_2$ is a (co)chain map.\\
    \item  Let $\mathcal{C} : C_0 \xrightarrow[]{c} C_1$ and $\mathcal{C'} : C_0' \xrightarrow[]{c'} C_1'$ be $1$-cubes. A map $\mathcal{C} \to \mathcal{C}'$ can be depicted as following 2-cube
    \begin{center}
       \includegraphics[scale=1.2]{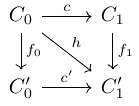}

    \end{center}
    satisfying $c' f_0 - f_1 c + h d + d h = 0$ where $d$ stands for the differential of $C_i$ and $C_i'$ for $i = 0, 1$.\\
\end{enumerate}
    
\end{example}
Now we define some operation on cubes. Before introducing it, let us fix some notations. For $1 \leq i \leq n$, let $\RR^{n-1}_i = \{(x_1, \cdots, x_n) \mid x_i = 0 \}$. Define a projection map $\pi_i : \RR^n \to \RR^{n-1}_i$ by
\begin{align*}
    \pi_i(x_1, \cdots, x_n) = (x_1, \cdots,x_{i-1}, 0, x_{i+1}, \cdots, x_n)
\end{align*}
and an inclusion map $\iota_{i,j} : \RR^{n-1}_i \to \RR^n$ by 
\begin{align*}
    \iota_{i,j}(x_1, \cdots,x_{i-1}, 0, x_{i+1}, \cdots, x_n) = (x_1, \cdots,x_{i-1}, j, x_{i+1}, \cdots, x_n) 
\end{align*} for $j = 0, 1$. For $\Bar{F} \subset [0,1]^{n-1} \subset \RR^{n-1}_i$, let $F = (\pi_i|_{[0,1]^n})^{-1}(\Bar{F})$ and $F_{i,j} = \iota_{i,j}(\Bar{F})$.
\begin{definition}
Let $\mathcal{C} = (\{C_v\}, \{f_F\})$ be an $n$-cube. 
\begin{enumerate}[label=(\alph*)]
    \item The \textbf{cone of $\mathcal{C}$ in the $i$-th direction}, denoted by $\text{cone}_i \mathcal{C}$, is an $(n-1)$-cube defined as follows. For each vertex $v$ of $[0,1]^{n-1} \subset \RR^{n-1}_i$, the corresponding module is 
\begin{align*}
    (\text{cone}_i \mathcal{C})_v = \mathcal{C}_{\iota_{i,0}(v)}[1] \oplus \mathcal{C}_{\iota_{i,1}(v)}.
\end{align*}
For each face $\Bar{F} \subset [0,1]^{n-1} \subset \RR^{n-1}_i$, the corresponding map is 
\begin{align*}
    \begin{pmatrix}
        (-1)^{|F_0| +1}f_F & 0\\
        (-1)^{\#(i,F)+1}f_F & f_{F_{i,1}}
    \end{pmatrix}
\end{align*}\
where $\#(i, F)$ denotes the position of $i$ in the set of free coordinates of $F$ relative to the order induced by $\{1, \cdots,n\}$.\\
\item We say that an $n$-cube $\mathcal{C}$ is \textbf{acyclic} if its maximally iterated cone of $\mathcal{C}$ is an acyclic complex, that is, \begin{align*}
\underbrace{\text{cone}_1 \circ \cdots \circ \text{cone}_1 }_{n\,\, \text{times}}\mathcal{C}    
\end{align*} is an acyclic complex.\\
\end{enumerate}
 
\end{definition}
\begin{example}
Consider the following 2-cube given in Example \ref{ex1}.
       \begin{center}
        \includegraphics[scale=1.2]{diagrams/example.pdf}
    \end{center}
    The cone of this 2-cube is given by
    \begin{align*}
        \left(C_0 \oplus C_0'[1] \oplus C_1[1] \oplus C_1', \,\,\begin{pmatrix}
            d & 0 & 0 & 0\\
            -f_0 &-d&0&0\\
            -c&0&-d&0\\
            h&c'&f_1&d
        \end{pmatrix} \right)
    \end{align*}\\
\end{example}

\begin{definition}\label{raytel}
\begin{enumerate}[label=(\alph*)]
    \item An \textbf{$n$-ray} is a diagram of the form 
\begin{align*}
    \mathcal{C}_1 \xrightarrow[]{\mathcal{D}_1} \mathcal{C}_2 \xrightarrow[]{\mathcal{D}_2} \mathcal{C}_3 \xrightarrow{\mathcal{D}_3} \cdots
\end{align*}
where $\mathcal{C}_i$ is an $n$-cube and $\mathcal{D}_i$ is a map from $\mathcal{C}_i$ to $\mathcal{C}_{i+1}$ for all $i = 1, 2, 3, \cdots$. We call each $\mathcal{C}_i$ a \textbf{slice} of an $n$-ray.\\
\item Let $\mathcal{C} : C_1 \xrightarrow{f_1} C_2 \xrightarrow{f_2} C_3 \xrightarrow{f_3} \cdots$ be a 1-ray. The \textbf{telescope} tel\,$\mathcal{C}$ of $\mathcal{C}$ is a chain complex with underlying module is given by 
\begin{align*}
    \displaystyle\bigoplus_{i=1}^{\infty} \left(C_i \oplus C_i[1]\right)
\end{align*} and its differential $\delta$ is given as follows. If $x \in C_i$, then $\delta x = d_i x \in C_i$ where $d_i$ is the differential of $C_i$. If $y \in C_i[1]$, then $\delta y = \Tilde{y} - d_i y + f_i(y) \in C_i \oplus C_i[1] \oplus C_{i+1}$ where $d_i$ is the differential of $C_i$ and $\Tilde{y}$ is the copy of $x$ in $C_i$. In other words, $y = \Tilde{y}$ but we think that $\Tilde{y}$ is an element of $C_i$ instead of $C_{i}[1]$.\\
\end{enumerate}

\end{definition}
\begin{remark}
Even though we only define the telescope of a 1-ray in Definition \ref{raytel}, we can still define the telescope of an $n$-ray in general. For this general construction, see \cite{dgpz} and \cite{v}.\\
\end{remark}

In some cases, we can replace the telescope of 1-ray by something that we are familiar with due to the following lemma.
\begin{lemma}[Varolgunes \cite{v}]\label{altdef}
    Let $\mathcal{C} : C_1 \xrightarrow{f_1} C_2 \xrightarrow{f_2} C_3 \xrightarrow{f_3} \cdots$ be a 1-ray. 
    \begin{enumerate}[label=(\alph*)]
        \item There exists a canonical quasi-isomorphism
    \begin{align*}
        \text{tel}\,\,\mathcal{C} \to \varinjlim_{i \to \infty} C_i.
    \end{align*}
    In other words, $H\left(\text{tel}\,\mathcal{C}\right) \cong H\left(\displaystyle\varinjlim_{i \to \infty} C_i\right)$ where $H$ denotes the homology functor.\\
    \item There exists a canonical quasi-isomorphism
    \begin{align*}
        \widehat{\text{tel}}\,\mathcal{C} \to \widehat{\varinjlim_{i \to \infty}} C_i
    \end{align*}
    where $\widehat{\text{tel}}\,\mathcal{C}$ denotes the completion of $\text{tel}\,\mathcal{C}$.\\
    \end{enumerate}
\end{lemma}
With the aid of Lemma \ref{altdef}, we have an alternative definition of relative symplectic cohomology.
\begin{definition}\label{homdef}
    Let $(M, \omega)$ be a closed symplectic manifold and $K \subset M$ be a compact subset. Let $\{H_n\}$ be a cofinal family of $\mathcal{H}_K$.
    \begin{enumerate}[label=(\alph*)]
        \item We have a 1-ray 
    \begin{align*}
        \mathcal{C}(\{H_n\}) : CF_w(H_1) \to CF_w(H_2) \to CF_w(H_3) \to \cdots.
    \end{align*}
    The \textbf{relative symplectic cohomology of $K$ in $M$} is defined by
    \begin{align*}
        SH^{}_M(K) = H \left( \widehat{\text{tel}}\, \mathcal{C}^{}(\{H_n\}) \right).
    \end{align*}
    \item If $\omega|_{\pi_2(M)} = 0$, then we have a 1-ray
    \begin{align*}
         \mathcal{C}^{>L}(\{H_n\}) : CF_w^{>L}(H_1) \to CF_w^{>L}(H_2) \to CF_w^{>L}(H_3) \to \cdots
    \end{align*}
    for each $L \in \RR$. The \textbf{relative symplectic cohomology of $K$ in $M$ with action greater than $L$} is defined by
    \begin{align*}
        SH^{>L}_M(K) = H \left( \widehat{\text{tel}}\, \mathcal{C}^{>L}(\{H_n\}) \right).
    \end{align*}\\
        \end{enumerate}
 \end{definition}
 \begin{remark}\label{homkill}
     We saw in Remark \ref{kill} that if $(M, \omega)$ is a closed symplectically aspherical symplectic manifold and $K \subset M$ is a Liouville domain with index-bounded boundary, then only lower orbits can survive through the completion process. We can ignore upper orbits in a slightly different setting. In \cite{s24}, Sun proved using the homotopy-theoretic definition of relative symplectic cohomology as described in Definition \ref{homdef} that we can still ignore upper orbits after completion in the case that $K$ is merely compact domain with contact type and index-bounded boundary inside symplectically aspherical $(M, \omega)$.\\
 \end{remark}
 Let us recall what it means for two sets to be in descent. Let $(M, \omega)$ be a closed symplectic manifold and $K_1, K_2 \subset M$ be compact sets. Let $\{H_n^{X}\}$ be a cofinal family of $\mathcal{H}_X$ for $X = K_1, K_2, K_1 \cap K_2$ and $K_1 \cup K_2$ such that $H_n^{X} \leq H_n^{X'}$ for all $n$ whenever $X' \subset X$. Then taking telescope and completion yields the following 2-cube.
 \begin{align}\label{descent}
     \includegraphics[scale=1.2]{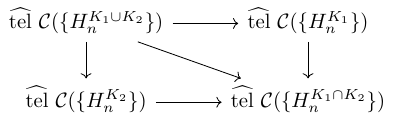}\\\nonumber
 \end{align}
\begin{definition}[Varolgunes \cite{v}]
    Let $(M, \omega)$ be a closed symplectic manifold and $K_1, K_2 \subset M$ be compact sets. We say that $K_1$ and $K_2$ satisfy \textbf{descent} if the 2-cube \eqref{descent} is acyclic.
\end{definition}
\begin{proposition}\label{mv}
    Let $(M, \omega)$ be a closed symplectic manifold with $\omega|_{\pi_2(M)} = 0$ and $K_1, K_2 \subset M$ be compact sets. 
    \begin{enumerate}[label=(\alph*)]
        \item If $K_1$ and $K_2$ satisfy descent, then following 2-cube is acyclic.
    \begin{align}\label{descentaction}
        \includegraphics[scale=1.2]{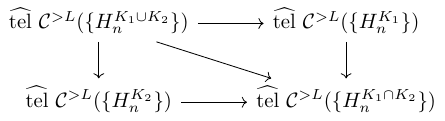}
    \end{align}\\
    \item There exists a Mayer-Vietoris exact triangle of action filtration of relative symplectic cohomology
    \begin{align}\label{mvrest}
        \includegraphics[scale=1.2]{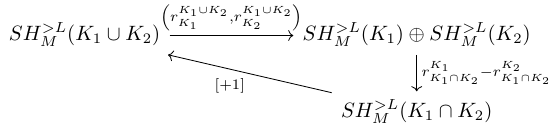}
    \end{align}
    where restriction maps appearing in Mayer-Vietoris exact triangle \eqref{mvrest} are restricted restriction maps to each action filtration. 
 Moreover, this Mayer-Vietoris property still holds if we replace coefficient ring $\Lambda_{\geq 0}$ by $\Lambda$.\\
    \end{enumerate}
\end{proposition}

\begin{proof}
  (a) It is immediate from the facts that arrows in \eqref{descent} are induced by continuation maps and that continuation maps increase the actions of Hamiltonian orbits.\\
  (b) Since the 2-cube \eqref{descentaction} is acyclic, there exsits a quasi-isomorphism 
  \begin{align*}
     \widehat{\text{tel}}\,\, \mathcal{C}(\{H_n^{K_1 \cup K_2}\}) \lr  \text{cone}_1 \left( \widehat{\text{tel}}\,\, \mathcal{C}(\{H_n^{K_1}\}) \oplus
    \widehat{\text{tel}}\,\, \mathcal{C}(\{H_n^{K_2}\}) \lr \widehat{\text{tel}}\,\, \mathcal{C}(\{H_n^{K_1 \cap K_2}\}) \right).
  \end{align*}
 Then the long exact sequence involving cone gives us the Mayer-Vietoris exact triangle. The last assertion follows from the flatness of $\Lambda$ over $\Lambda_{\geq 0}$.\\ 
\end{proof}
\section{$S^1$-equivariant relative symplectic cohomology}
\subsection{Definition} We shall combine the definition of relative symplectic cohomology with that of $S^1$-equivariant symplectic cohomology.
\begin{definition}
   Let $(M, \omega)$ be a closed symplectic manifold and let $K \subset M$ be a compact domain with contact type boundary. A \textbf{contact type $K$-admissible Hamiltonian function} is a smooth function $H : S^1 \times M \to \RR$ satisfying the following conditions.
    \begin{enumerate}[label=(\alph*)]
        \item $H$ is negative and $C^2$-small on $S^1 \times K$. Moreover, $H > - \epsilon$ on $S^1 \times K$ where $\epsilon > 0$ is the half of minimal period of Reeb orbits of $\partial K$.\\
        \item There exists $\eta \geq 0$ such that $H(t, p, \rho)$ is $C^2$-close to $h_1(e^{\rho})$ on $S^1 \times \left(\partial K \times [0, \frac{1}{3}\eta]\right)$ for some convex and increasing function $h_1$. \\
        \item $H(t, p, \rho) = \beta e^{\rho} +\beta'$ on $\partial K \times [\frac{1}{3}\eta, \frac{2 }{3}\eta]$ where $\beta \notin \text{Spec}(\partial K, \lambda_K)$ and $\beta' \in \RR$.\\
        \item $H(t, p, \rho)$ is $C^2$-close to $h_2(e^{\rho})$ on $S^1 \times \left(\partial K \times [\frac{2 }{3}\eta, \eta]\right)$ for some concave and increasing function $h_2$.\\
        \item $H$ is $C^2$-close to a constant function on $S^1 \times \left(M - K \cup \left(\partial K \times [0, \eta]\right)\right)$.
    \end{enumerate}
    We denote the set of all contact type $K$-admissible Hamiltonian functions by $\mathcal{H}^{\text{Cont}}_K$.\\
\end{definition}
Let $K \subset M$ be a compact domain with contact type boundary. For a Hamiltonian function $H \in \mathcal{H}^{\text{Cont}}_K$, define the complex $CF_w^{S^1}(H)$ by
\begin{align*}
    CF_w^{S^1}(H) = \Lambda_{\geq 0}[u] \otimes_{\Lambda_{\geq 0}} CF_w(H)
\end{align*}
where $u$ is a formal variable of degree 2. Its differential is of the form 
\begin{align}\label{equdiff}
    d_w^{S^1} (u^k \otimes x) = \sum_{i=0}^k u^{k-i} \otimes \psi_i(x) 
\end{align}
where $\psi_i$ is a weighted version of $\phi_i$. \\
\begin{definition}
   Let $(M, \omega)$ be a closed symplectic manifold and let $K \subset M$ be a compact domain with contact type boundary. The \textbf{$S^1$-equivariant relative symplectic cohomology of $K$ in $M$} is defined by
    \begin{align*}
        SH^{S^1}_M(K) &= H \left(\widehat{\varinjlim_{H \in \mathcal{H}^{\text{Cont}}_K}} CF_w^{S^1}(H)\right). 
    \end{align*}
    Also, the \textbf{$S^1$-equivariant relative symplectic homology of $K$ in $M$ with coefficient in $\Lambda$} is defined by
    \begin{align*}
        SH^{S^1}_M(K ; \Lambda) = SH^{S^1}_M(K) \ton \Lambda.\\
    \end{align*}
\end{definition}
\begin{remark}\label{alts1def}
    We can define $S^1$-equivariant relative symplectic cohomology utilizing the homotopical tools that we developed earlier: Let $\{H_n\}$ be a cofinal family of $\mathcal{H}_K^{\text{Cont}}$. Then we have a 1-ray of $S^1$-equivariant Floer complexes
    \begin{align*}
        \mathcal{C}^{S^1}(\{H_n\}) : CF_w^{S^1}(H_1) \to CF_w^{S^1}(H_2) \to CF_w^{S^1}(H_3) \to \cdots.
    \end{align*}
    Then
    \begin{align*}
        SH^{S^1}_M(K) = H \left( \widehat{\text{tel}}\, \mathcal{C}^{S^1}(\{H_n\}) \right).\\
    \end{align*}
\end{remark}

%\begin{definition}
   % Let $(M, \omega)$ be a closed symplectic manifold and let $K \subset M$ be a compact subset. The $S^1$-equivariant relative symplectic cohomology of $K$ in $M$ is defined by
   % \begin{align*}
     %   SH^{S^1}_M(K) = H\left( \widehat{\varinjlim_{H \in \mathcal{H}_K}} CF^{S^1}(H) \right).
  %  \end{align*}
  %   Also, the $S^1$-equivariant relative symplectic homology of $K$ in $M$ with coefficient in $\Lambda$ is defined by
   % \begin{align*}
     %   SH^{S^1}_M(K ; \Lambda) = SH^{S^1}_M(K) \ton \Lambda.\\
  %  \end{align*}
%\end{definition}

\begin{proposition}\label{diskinsphere}
    Let $S^2$ be a 2-dimensional sphere with an area form $\omega$ which is normalized for $S^2$ to have area 1. Let $K \subset S^2$ be a smooth disk of area $\Delta$. Then
    \begin{align*}
        SH^{S^1}_{S^2}(K;\Lambda) = \begin{cases}
            0 &\text{if}\,\,\Delta< \frac{1}{2}\\
            \Lambda[u] \oplus \Lambda[u] & \text{if}\,\,\Delta \geq \frac{1}{2}.
        \end{cases}
    \end{align*}
\end{proposition}
\begin{proof}
    This example is outlined in \cite{vt} and can be completed by the help of Varolgunes \cite{vc}. Let $\{h_n : [0, 1] \to \RR\}_{n \in \NN}$ be a sequence of functions satisfying
    \begin{itemize}
        \item $h_n < h_{n+1}$,\\
        \item $h_n$ is negative on $[0, \Delta]$,\\
        \item $h_n$ is linear with slope $c_n$ on $[0, \Delta]$ and on $[\Delta + \epsilon_n, 1]$ where $\displaystyle\lim_{n \to \infty} c_n = 0$ and $\displaystyle\lim_{n \to \infty} \epsilon_n = 0$,\\
        \item $h_n = h_{n1}$ on $[\Delta, \Delta +\frac{1}{3} \epsilon_n]$ where $h_{n1}$ is an increasing and convex function with 
        \begin{align*}
            \left\{h_{n1}'(x) \bigmid \Delta \leq x \leq \Delta + \frac{1}{3}\epsilon_n \right\} \cap \ZZ=\{ -1, -2, -3, \cdots, -n \},
        \end{align*}\\
         \item $h_n$ is linear on $[\Delta + \frac{1}{3} \epsilon_n, \Delta + \frac{2}{3}\epsilon_n]$ with its slope not in $\text{Spec}(\partial K)$, and\\
        \item $h_n = h_{n2}$ on $[\Delta + \frac{2}{3} \epsilon_n, \Delta + \epsilon_n]$ where $h_{n2}$ is an increasing and concave function with 
        \begin{align*}
            \left\{h_{n2}'(x) \bigmid \Delta +\frac{2}{3}\epsilon_n \leq x \leq \Delta + \epsilon_n \right\} \cap \ZZ =\{ -1, -2, -3, \cdots, -n \}.
        \end{align*}
       
            \end{itemize}
    We may think that $S^2 = \left\{ (x,y,z) \in \RR^3 \bigmid 
x^2 +y^2 +(z- \frac{1}{2})^2 = \frac{1}{4} \right\}$. Then $\{ H_n = h_n \circ m \}_{n \in \NN}$ is a cofinal sequence of $\mathcal{H}_K^{\text{Cont}}$ where $m : S^2 \to [0,1] $ is a map sending $(x,y,z)$ to $z$. Then lower orbits of $H_n$ consist of one constant orbit(minimum of $H_n$) and $n$ nonconstant periodic orbits. Similarly, upper orbits consist of one constant orbit(maximum of $H_n$) and $n$ nonconstant periodic orbits. After perturbing $H_n$ as in \cite{bo09} and \cite{cfhw}, we have exactly two orbits for each nonconstant orbit of $H_n$. We have laid out all the generators below. Note that outermost points correspond to minimum and maximum, respectively.
\begin{center}
    \includegraphics[scale=0.8]{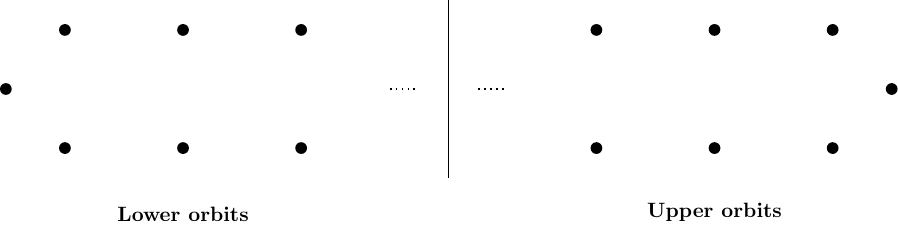}
\end{center}
First, let us consider $S^2 - \{\text{north pole}\}$ and project it to $\RR^2$. Note that this projection is orientation-reversing. Following the computation given in \cite{osurv}, we have a following diagram.
\begin{align}\label{exnorthremoved}
    \includegraphics[scale=0.8]{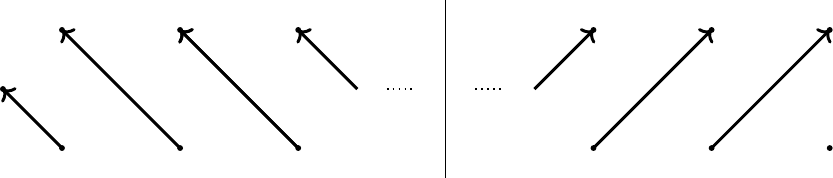}
\end{align}
In \eqref{exnorthremoved}, the rightmost point is removed because we are considering $S^2 - \{\text{north pole}\}$. But \eqref{exnorthremoved} is not sufficient since differential does not square to zero. With the winding numbers of nonconstant orbits in mind, there exist more arrows as follows.
\begin{align}\label{exnorthremoved2}
    \includegraphics[scale=0.8]{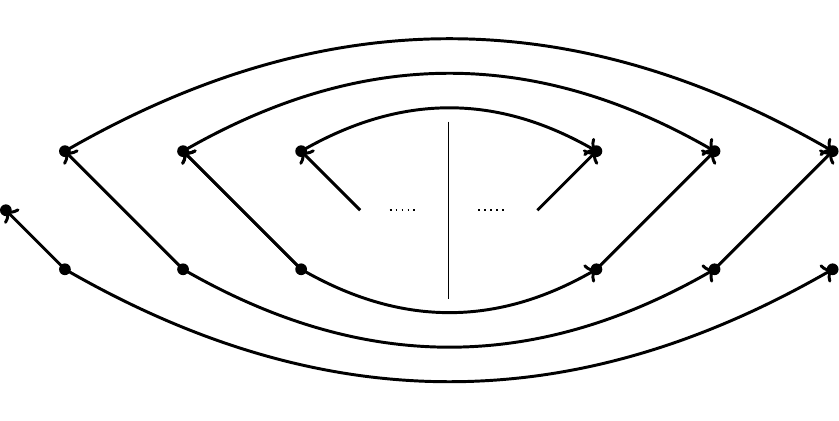}
\end{align}
Now it remains to consider the case $S^2 - \{\text{south pole}\}$. In an analogous way, we can draw a following diagram.
\begin{align}\label{exsouthremoved}
    \includegraphics[scale=0.8]{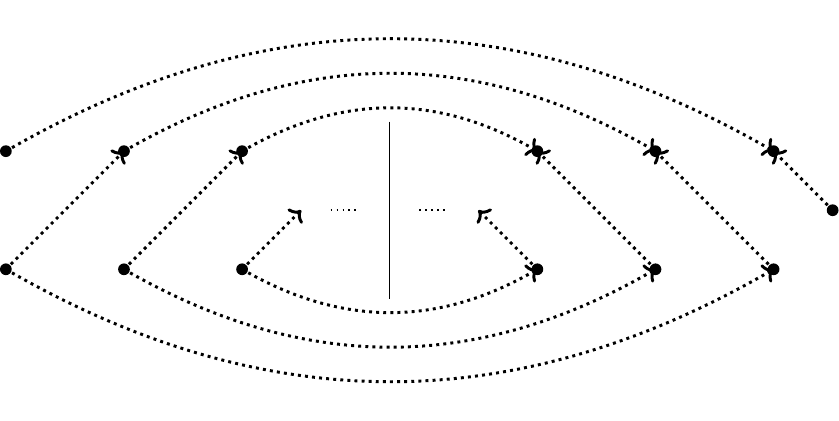}
\end{align}
Putting \eqref{exnorthremoved2} and \eqref{exsouthremoved} together and removing upper orbits as described in Remark \ref{kill}, the chain complex $\displaystyle\widehat{\varinjlim_{n \to \infty}} CF_w(H_n)$ can be depicted as follows.  
\begin{align}\label{exfinal}
    \includegraphics[]{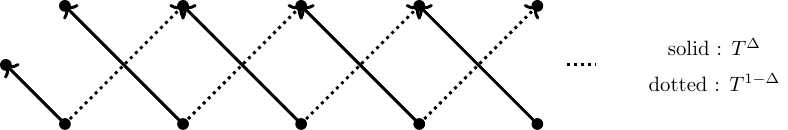}
\end{align}
Note that Floer trajectories represented by solid arrows have topological energy $\Delta$ and those represented by dotted arrows have topological energy $1 -\Delta$. To recapitulate, the chain complex $\displaystyle\widehat{\varinjlim_{n \to \infty}} CF_w(H_n)$ is given by $\widehat{\textbf{C}} \oplus \widehat{\textbf{C}}$ where $(\textbf{C},d)$ is a chain complex satisfying the following:
    \begin{itemize}
        \item \textbf{C} = $\displaystyle\bigoplus_{i \geq 0} \left(\Lambda_{\geq 0} \langle \gamma_i \rangle \oplus \Lambda_{\geq 0} \langle \beta_{i+1} \rangle \right)$ and\\
        \item $d \gamma_i = 0$ and $d \beta_{i+1} = T^{\Delta} \gamma_i + T^{1 - \Delta} \gamma_{i+1}$ for $i = 0, 1, 2, \cdots$.\\
    \end{itemize}
    Since every $H_n$ depends only on $M$, we can deduce from \cite{bo} that the chain complex for $SH^{S^1}_M(K)$ is given by $\widehat{\textbf{C}^{S^1}} \oplus \widehat{\textbf{C}^{S^1}}$ where $(\textbf{C}^{S^1}, d^{S^1})$ is a complex such that 
    \begin{align}\label{diskspheres1}
     \textbf{C}^{S^1} = \Lambda_{\geq0}[u] \otimes_{\Lambda_{\geq 0}} \textbf{C} \,\,\mathrm{and}\,\,d^{S^1} = \text{id} \otimes d   
    \end{align}
    Because of \eqref{diskspheres1}, $SH^{S^1}_M(K; \Lambda) = \Lambda[u] \ton SH_M(K ; \Lambda)$. So, it suffices to find $SH_M(K ; \Lambda)$. By direct calculation, one can check that
    \begin{align*}
      \gamma_0 + (-1)^{n} T^{(n+1)(1-2\Delta)} \gamma_{n+1} = d \left( T^{-\Delta} \sum_{i=0}^n (-1)^i T^{(1-2 \Delta)i} \beta_{i+1} \right)
    \end{align*}
    holds in $\mathbf{C}$. We are given two cases to consider.
    \begin{enumerate}[label=\arabic*.]
        \item $\Delta < \frac{1}{2}$ : Letting $n\to \infty$, $(-1)^{n} T^{(n+1)(1-2\Delta)} \gamma_{n+1}$ vanishes because it can be viewed as a sequence converging to zero in $\mathbf{C}$. Also, $\displaystyle\sum_{i=0}^n (-1)^i T^{(1-2 \Delta)i} \beta_{i+1}$ converges as $n \to \infty$ in $\mathbf{C}$ because it is a Cauchy sequence. Hence, $\displaystyle\sum_{i=0}^{\infty} (-1)^i T^{(1-2 \Delta)i} \beta_{i+1} \in \widehat{\textbf{C}}$. Since $H \left( \widehat{\mathbf{C}} ; \Lambda \right) = H \left( \widehat{\mathbf{C}} \ton \Lambda \right)$ and $\gamma_0$ is exact in $\widehat{\mathbf{C}}$, we have $H \left( \widehat{\mathbf{C}} ; \Lambda \right) = 0$.\\
        \item $\Delta \geq \frac{1}{2}$ : In this case, $[\gamma_0]$ is still a generator of $H \left( \widehat{\mathbf{C}} ; \Lambda \right)$ but it is not exact and therefore $H \left( \widehat{\mathbf{C}} ; \Lambda \right) = \Lambda$.\\ 
    \end{enumerate}
   Consequently, we have
        \begin{align*}
        SH_M(K; \Lambda) = \begin{cases}
            0 &\text{if}\,\, \Delta < \frac{1}{2}\\
            \Lambda \oplus \Lambda &\text{if}\,\, \Delta \geq \frac{1}{2}.
        \end{cases}
    \end{align*}
    Finally,
    \begin{align*}
        SH^{S^1}_M(K; \Lambda) = \begin{cases}
            0 &\text{if}\,\, \Delta < \frac{1}{2}\\
            \Lambda[u] \oplus \Lambda[u] &\text{if}\,\, \Delta \geq \frac{1}{2}.
        \end{cases}\\
    \end{align*}
\end{proof}
Now we take the action into account to form a filtration of $CF_w^{S^1}(H)$. Let $(M, \omega)$ be a closed symplectic manifold with $\omega|_{\pi_2(M)} = 0$. For $L \in \RR$, $CF_w^{S^1, > L}(H)$ be the subset of $CF_w^{S^1}(H)$ generated by $u^k \otimes x$ where $\mathcal{A}_H(x) > L$. Since the differential increases the action, it is actually a subcomplex of $CF_w^{S^1}(H)$. Moreover, continuation maps still work for this subcomplex because they also increase the action.  
\begin{definition}\label{actionhom}
   Let $(M, \omega)$ be a closed symplectic manifold with $\omega|_{\pi_2(M)} = 0$ and let $K \subset M$ be a compact domain with contact type boundary. 
    \begin{enumerate}[label=(\alph*)]
        \item The \textbf{$S^1$-equivariant relative symplectic cohomology of $K$ in $M$ with action greater than $L$} is defined by
    \begin{align*}
        SH^{S^1, > L}_M(K) = H \left(\widehat{\varinjlim_{H \in \mathcal{H}^{\text{Cont}}_K}} CF_w^{S^1, >L}(H)\right).
    \end{align*}
     Also, 
    \begin{align*}
        SH^{S^1, > L}_M(K ; \Lambda) := SH^{S^1, >L}_M(K) \ton \Lambda.
    \end{align*}
    Of course, $SH^{S^1, > -\infty}_M(K) = SH^{S^1}_M(K)$ and $SH^{S^1, > -\infty}_M(K ; \Lambda) = SH^{S^1}_M(K ; \Lambda)$.\\
    \item Note that for small enough $\epsilon > 0$, $CF_w^{S^1, > -\epsilon}(H)$ is generated by the elements $u^k \otimes x$ where $x$ is either a (lower) critical point of $H$ in $K$ or an upper 1-periodic orbit. Therefore, the quotient
    \begin{align*}
        CF^{S^1,-}_w(H) := CF_w^{S^1}(H)/CF_w^{S^1, >-\epsilon}(H)
    \end{align*}
    is generated by the elements $u^k \otimes x$ where $x$ is a lower nonconstant 1-periodic orbit. Define the \textbf{negative $S^1$-equivariant relative symplectic cohomology of $K$ in $M$} by
    \begin{align*}
        SH^{S^1,-}_M(K) = H\left(\widehat{\varinjlim_{H \in \mathcal{H}^{\text{Cont}}_K}} CF_w^{S^1}(H)/CF_w^{S^1, >-\epsilon}(H)\right).
    \end{align*}
    Furthermore, We can define the filtrations $SH^{S^1,-,>L}_M(K)$ of $SH^{S^1,-}_M(K)$ in a similar way. Tensoring $\Lambda$ over $\Lambda_{\geq0}$ with cohomologies defined above, we can replace the coefficients in $\Lambda_{\geq0}$ by the coefficients in $\Lambda$.\\
    \end{enumerate}
    \end{definition}
       \begin{remark} 
       \begin{enumerate}[label=(\alph*)]
           \item  Because our sign convention for the action functional is different from that of \cite{bo09}, \cite{bo13}, \cite{bo} and \cite{gs}, we use the Floer complex generated by Hamiltonian orbits whose actions are \textit{greater than} $L$ instead of less than $L$. More importantly, our \textit{negative} symplectic cohomology corresponds to positive symplectic homology in the literature mentioned earlier. \\
           \item The action filtrations $SH^{S^1,>L}_M(K)$ of $SH^{S^1}_M(K)$ can also be defined using ray and telescope as pointed out in Remark \ref{alts1def}.\\
           \item In a similar way that we define the negative $S^1$-equivariant relative symplectic cohomology, we can define the \textbf{negative relative symplectic cohomology of $K$ in $M$} and we denote this by $SH^{-}_M(K)$.\\
       \end{enumerate}
      
         % The terminology `negative' might be misleading because we still have upper orbits which can have positive actions. But we will consider the cases that we can ignore the upper orbits. Such cases are recapitulated in Remark \ref{homkill}.\\  

    \end{remark}

\subsection{Properties} We investigate some properties of $S^1$-equivariant relative symplectic cohomology. Properties listed below are the analogues of the results in \cite{bo}, \cite{dgpz} and \cite{v}.

%\noindent\textbf{Notation.} From this point on, unless otherwise stated, $(M, \omega)$ will denote a closed symplectic manifold which is symplectically aspherical and $K \subset M$ will mean a Liouville domain with index-bounded boundary.
\begin{proposition}\label{gysin}
    Let $(M, \omega)$ be a closed symplectic manifold and $K \subset M$ be a compact domain with contact type boundary. Then there exists a Gysin exact triangle
    \begin{center}
        \includegraphics[scale=1.24]{diagrams/fullgysin.pdf}
    \end{center}
    And this exact triangle still holds if we replace the coefficient ring $\Lambda_{\geq0}$ by $\Lambda$.
\end{proposition}
\begin{proof}
    Let $\{H_n\}$ be a sequence of cofinal Hamiltonian functions defining $SH^{S^1}_M(K)$. Define a map $U : CF^{S^1}_w(H_n) \,\,\lr\,\, CF^{S^1}_w(H_n)[+2]$ by 
    \begin{align*}
        U(u^k \otimes x) = \begin{cases}
            u^{k-1} \otimes x &\text{if} \,\,k \geq 1\\
            0 &\text{if} \,\, k=0.
        \end{cases}
    \end{align*}
    Note that the map $U$ respects the action. We can construct a short exact sequence of chain complexes
    \begin{align}\label{proofcoro}
        0 \lr CF_w(H_n) \xlongrightarrow{(*)} CF^{S^1}_w(H_n) \xlongrightarrow{U} CF^{S^1}_w(H_n)[+2] \lr 0
    \end{align}
    where the map labelled by $(*)$ is $x \mapsto 1 \otimes x$. Since the direct limit is an exact functor, we have
    \begin{align*}
        0 \lr \varinjlim_{n \to \infty} CF_w(H_n) \lr \varinjlim_{n \to \infty} CF^{S^1}_w(H_n) \lr \varinjlim_{n \to \infty} CF^{S^1}(H_n)[+2] \lr 0.
    \end{align*}
    Note that $CF^{S^1}_{w}(H_n)[+2]$ is a free $\Lambda_{\geq 0}$-module generated by 1-periodic orbits of $H_n$ and hence it is a flat $\Lambda_{\geq 0}$-module. Since the direct limit of flat modules is still flat, $\displaystyle\varinjlim_{n \to \infty}CF^{S^1}_{w}(H_n)[+2]$ is flat over $\Lambda_{\geq 0}$ and this implies
    \begin{align*}
    \text{Tor}_1^{\Lambda_{\geq 0}}\left(\varinjlim_{n \to \infty}CF^{S^1}_{w}(H_n)[+2], \Lambda_{\geq 0}/ \Lambda_{\geq r}\right) = 0
\end{align*}
for each $r > 0$. Therefore, the following sequence of chain complexes is exact.
\begin{align*}
0 \lr \varinjlim_{n \to \infty}CF_{w}(H_n) \ton \Lambda_{\geq 0}/ \Lambda_{\geq r}&\longrightarrow \varinjlim_{n \to \infty}CF^{S^1}_w(H_n)\ton \Lambda_{\geq 0}/ \Lambda_{\geq r} \\&\longrightarrow \varinjlim_{n \to \infty}CF^{S^1}_{w}(H_n)[+2]\ton \Lambda_{\geq 0}/ \Lambda_{\geq r} \lr 0.
\end{align*}
For $r' > r$, the projection $\Lambda_{\geq 0}/ \Lambda_{\geq r'} \to \Lambda_{\geq 0}/ \Lambda_{\geq r}$ is surjective and, for any $\Lambda_{\geq 0}$-module $A$, $A \ton \Lambda_{\geq 0}/ \Lambda_{\geq r'} \to A \ton \Lambda_{\geq 0}/ \Lambda_{\geq r}$ is also surjective due to the right exactness of tensor product. Then by the Mittag-Leffler theorem for the inverse limit, we still have a short exact sequence
\begin{align*}
    0 \longrightarrow \widehat{\varinjlim_{n \to \infty}} CF_{w}(H_n) \longrightarrow \widehat{\varinjlim_{n \to \infty}} CF_w^{S^1}(H_n) \longrightarrow 
\widehat{\varinjlim_{n \to \infty}} CF_{w}^{S^1}(H_n) [+2]\longrightarrow 0.
\end{align*}
    Then the long exact sequence induced by a short exact sequence of chain complexes gives us a Gysin exact triangle of $SH^{S^1}_M(K)$ and $SH_M(K)$. For more detailed explanation, we refer readers to \cite{bo}.  Because $\Lambda$ is a flat $\Lambda_{\geq 0}$-module, Gysin exact triangle still holds if replace coefficient ring $\Lambda_{\geq 0}$ by $\Lambda$.\\
\end{proof}
\begin{corollary}\label{gysinaction}
    Let $(M, \omega)$ be a closed symplectic manifold with $\omega|_{\pi_2(M)} = 0$ and $K \subset M$ be a compact domain with contact type boundary.
    \begin{enumerate}[label=(\alph*)]
        \item There exists a Gysin exact triangle
    \begin{center}
        \begin{tikzcd}
    SH^{>L}_M(K) \arrow{r} &SH^{S^1,>L}_M(K) \arrow{d}{[+2]} \\
    &SH^{S^1,>L}_M(K) \arrow{lu}{[-1]}    
\end{tikzcd}
    \end{center}
    for each $L \in \RR$.\\
    \item There exists a Gysin exact triangle
    \begin{align*}
        \includegraphics[]{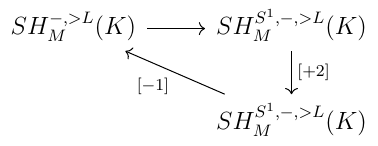}
    \end{align*}
    for each $L \in \RR$.
    \end{enumerate}
    Moreover, (a) and (b) still hold if we replace the coefficient ring $\Lambda_{\geq 0}$ by $\Lambda$.
\end{corollary}
\begin{proof}
(a) Since the maps in the exact sequence \eqref{proofcoro} respect the action filtrations, we have a short exact sequence
\begin{align}\label{proofcoro2}
    0 \lr CF^{>L}_w(H_n) \longrightarrow CF^{S^1, >L}_w(H_n) \longrightarrow{} CF^{S^1,>L}_w(H_n)[+2] \lr 0
\end{align}
    for each $L \in \RR$. This short exact sequence induces the desired exact triangle.\\
    (b) The short exact sequence \eqref{proofcoro2} can be modified to another short exact sequence
    \begin{align*}
        0 \lr CF^{>L}_w(H_n)/CF^{>-\epsilon}_w(H_n) &\longrightarrow CF^{S^1, >L}_w(H_n)/CF^{S^1, >-\epsilon}_w(H_n)\\& \longrightarrow{} CF^{S^1,>L}_w(H_n)[+2]/CF^{S^1,>-\epsilon}_w(H_n)[+2] \lr 0
    \end{align*}
    and this induces the exact triangle.
\end{proof}
\begin{proposition}\label{spectral}
    Let $(M, \omega)$ be a closed symplectic manifold and $K \subset M$ be a compact domain with contact type boundary. Then there exists a spectral sequence $E_r^{p,q}(M, K)$ converging to $SH^{S^1}_M(K)$ such that its second page is given by
    \begin{align*}
        E_2^{p,q}(M,K) \cong H^p(BS^1 ; \Lambda_{\geq 0}) \otimes SH^q_M(K).
    \end{align*}
    Also, the same result holds for $\Lambda$ coefficient.
\end{proposition}
\begin{proof}
    Consider the filtration $(F^p \textbf{C})$ of $\displaystyle\widehat{\varinjlim_{H \in \mathcal{H}_K^{\text{Cont}}}} \Lambda_{\geq_0}[u] \otimes CF_w(H)$ given by
    \begin{align*}
        F^p\textbf{C} = \widehat{\varinjlim_{H \in \mathcal{H}_K^{\text{Cont}}}} \left( \Lambda_{\geq_0}[u] / \langle u^{p+1} \rangle \right) \otimes CF_w(H).
    \end{align*}
    Then the spectral sequence corresponding to the filtration $(F^p \textbf{C})$ is the desired one. Tensoring $\Lambda$ over $\Lambda_{\geq 0}$, we can obtain the same result for $\Lambda$ coefficient. \\
\end{proof}

\begin{proposition}\label{relcl}
Let $(M, \omega)$ be a closed symplectic manifold which is symplectically aspherical. Let $K \subset M$ be Liouville domain. If the map $\pi_1(\partial K) \to \pi_1(M)$ induced by the inclusion is injective and $\partial K$ is index-bounded, then there is an isomorphism
    \begin{align*}
        SH^{S^1}_M(K ; \Lambda) \cong SH^{S^1}(K; \Lambda)
    \end{align*}
    where $SH^{S^1}(K; \Lambda)$ stands for the symplectic cohomology of $K$ as a symplectic manifold with contact type boundary. Moreover, we have a following commutative diagram.
    \begin{center}
    \includegraphics[scale=1.2]{diagrams/square.pdf}
    \end{center}
    Here, the vertical maps are induced by the chain level map $x \mapsto 1 \otimes x$.
\end{proposition}
\begin{proof}
    Note first that the injectivity of the map $\pi_1(\partial K) \to \pi_1(M)$ guarantees that the contractibility in $\partial K$ is equivalent to the contractibility in $M$. Let $\{H_n\}$ be a cofinal sequence of $\mathcal{H}^{\text{Cont}}_K$. Let $CF^{S^1}_{w, U}(H_n)$ be the subset of $CF^{S^1}_w(H_n)$ generated by $u^k \otimes x$ where $x$ is an upper orbit. Since the Floer differential increases the action and upper orbits have larger actions than lower orbits do, there are no Floer trajectories going from upper orbits to lower orbits. Hence, the subset $CF^{S^1}_{w, U}(H_n)$ of $CF^{S^1}_w(H_n)$ is a subcomplex of $CF^{S^1}_w(H_n)$. Define $CF^{S^1}_{w, L}(H_n) := CF^{S^1}_w(H_n) / CF^{S^1}_{w, U}(H_n)$. Then we obtain a short exact sequence of chain complexes
     \begin{align*}
         0 \longrightarrow CF^{S^1}_{w, U}(H_n) \longrightarrow CF^{S^1}_w(H_n) \longrightarrow CF^{S^1}_{w, L}(H_n) \lr 0.
    \end{align*}    
    Since the direct limit is an exact functor, we still have a short exact sequence of chain complexes
    \begin{align*}
        0 \lr \varinjlim_{n \to \infty}CF^{S^1}_{w, U}(H_n) \longrightarrow \varinjlim_{n \to \infty}CF^{S^1}_w(H_n) \longrightarrow \varinjlim_{n \to \infty}CF^{S^1}_{w, L}(H_n) \lr 0.
    \end{align*}
    Note that $CF^{S^1}_{w, L}(H_n)$ is a free $\Lambda_{\geq 0}$-module generated by lower orbits and hence it is a flat $\Lambda_{\geq 0}$-module. Since the direct limit of flat modules is still flat, $\displaystyle\varinjlim_{n \to \infty}CF^{S^1}_{w, L}(H_n)$ is flat over $\Lambda_{\geq 0}$ and this implies
    \begin{align*}
    \text{Tor}_1^{\Lambda_{\geq 0}}\left(\varinjlim_{n \to \infty}CF^{S^1}_{w, L}(H_n), \Lambda_{\geq 0}/ \Lambda_{\geq r}\right) = 0
\end{align*}
for each $r > 0$. Therefore, the following sequence of chain complexes is exact:
\begin{align*}
0 \lr \varinjlim_{n \to \infty}CF^{S^1}_{w, U}(H_n) \ton \Lambda_{\geq 0}/ \Lambda_{\geq r}&\longrightarrow \varinjlim_{n \to \infty}CF^{S^1}_w(H_n)\ton \Lambda_{\geq 0}/ \Lambda_{\geq r} \\&\longrightarrow \varinjlim_{n \to \infty}CF^{S^1}_{w, L}(H_n)\ton \Lambda_{\geq 0}/ \Lambda_{\geq r} \lr 0.
\end{align*}
For $r' > r$, the projection $\Lambda_{\geq 0}/ \Lambda_{\geq r'} \to \Lambda_{\geq 0}/ \Lambda_{\geq r}$ is surjective and, for any $\Lambda_{\geq 0}$-module $A$, $A \ton \Lambda_{\geq 0}/ \Lambda_{\geq r'} \to A \ton \Lambda_{\geq 0}/ \Lambda_{\geq r}$ is also surjective due to the right exactness of tensor product. Then by the Mittag-Leffler theorem for the inverse limit, we still have a short exact sequence
\begin{align}\label{ses}
    0 \longrightarrow \widehat{\varinjlim_{n \to \infty}} CF_{w, U}^{S^1}(H_n) \longrightarrow \widehat{\varinjlim_{n \to \infty}} CF_w^{S^1}(H_n) \longrightarrow 
\widehat{\varinjlim_{n \to \infty}} CF_{w, L}^{S^1}(H_n) \longrightarrow 0.
\end{align}
By Remark \ref{homkill}, $\displaystyle\widehat{\varinjlim_{n \to \infty}} CF_{w, U}^{S^1}(H_n) = 0$. Therefore, the long exact sequnece induced by \eqref{ses} implies that
\begin{align*}
   SH^{S^1}_M(K) &= H \left( \widehat{\varinjlim_{n \to \infty}} CF_w^{S^1}(H_n)\right)&&\text{Definition of relative SH}\\& \cong H\left( \widehat{\varinjlim_{n \to \infty}} CF_{w, L}^{S^1}(H_n)\right) &&\widehat{\varinjlim_{n \to \infty}} CF_{w, U}^{S^1}(H_n) = 0 \,\,\text{and}\,\, \eqref{ses}.
\end{align*}
   Finally, the chain of equalities(isomorphisms) below explains the agreement of two symplectic cohomologies.
    \begin{align*}
        SH^{S^1}_M(K ; \Lambda) &= H\left(\widehat{\varinjlim_{n \to \infty}}CF^{S^1}_w(H_n)\right)\ton \Lambda &&\text{Definition of relative SH} \\
        &\cong H\left(\widehat{\varinjlim_{n \to \infty}}CF^{S^1}_{w,L}(H_n)\right)\ton \Lambda &&\text{Ignore the upper orbits}\\
        &\cong H\left(\varinjlim_{n \to \infty}CF^{S^1}_{w,L}(H_n)\right)\ton \Lambda &&\text{$\varinjlim_{n \to \infty}CF^{S^1}_{w,L}(H_n)$ is complete}\\
        &\cong H\left(\varinjlim_{n \to \infty}CF^{S^1 }_{w,L}(H_n)\ton \Lambda\right)\\
        &\cong H\left(\varinjlim_{n \to \infty}CF^{S^1}(H_n ; \Lambda)\right) &&\text{Remark \ref{clw} and Remark \ref{kill}}\\
        &= SH^{S^1} (K; \Lambda) &&\text{Definition of classical SH}
    \end{align*}
 and thus $SH^{S^1}_M(K ; \Lambda) \cong SH^{S^1}(K ; \Lambda)$. For more detailed exposition, see \cite{dgpz}.    \\
\end{proof}
%\begin{proposition}
 %   Let $(M, \omega)$ be a compact symplectic manifold and $K \subset M$ be compact domain. If $K$ is stably displaceable, then $SH^{S^1}_M(K ; \Lambda) = 0$.
%\end{proposition}
%Now we shall prove Mayer-Vietoris property for $S^1$-equivariant relative symplectic cohomology. To do this, let us first remind some algebraic backgrounds established in \cite{v}.
%\begin{lemma}[Varolgunes \cite{v}]\label{algebraiclemma}
  %  \begin{enumerate}[label=(\alph*)]
      %  \item Let $C : C_1 \to C_2 \to C_3 \to \cdots$ be a chain complex such that each $C_i$ is a finitely generated free module. If $C \,\ton \Lambda_{\geq 0} / \Lambda_{>0}$ is acyclic, then $C$ is acyclic.\\
      %  \item Let $\mathcal{C}$ be an $n$-ray whose underlying modules are free. If all the slices of $\mathcal{C}$ are acyclic $(n-1)$-cubes, then $\widehat{\text{tel}}\,\mathcal{C}$ is acyclic.\\
       
   % \end{enumerate}
%\end{lemma}
\begin{proposition}\label{mvs1}
    Let $(M, \omega)$ be a closed symplectic manifold and $K_1, K_2 \subset M$ be a compact domains with contact type boundaries. If $(K_1, K_2)$ is a contact pair and $K_1$ and $K_2$ Poisson-commute, then there exists a Mayer-Vietoris exact triangle for $S^1$-equivariant relative symplectic cohomology pertaining to $K_1$ and $K_2$.
    \begin{align*}
        \includegraphics[scale=1.23]{diagrams/s1mv.pdf}
    \end{align*}
   Moreover, this triangle stays exact if we replace coefficient ring $\Lambda_{\geq 0}$ by $\Lambda$. 
\end{proposition}
\begin{proof}
    Let $\{H_n^{K_1}\}$ and $\{H_n^{K_2}\}$ be cofinal sequences of $\mathcal{H}_{K_1}$ and $\mathcal{H_{K_2}}$, respectively, which are constructed throughout the section 4 of \cite{v}. The Hamiltonian functions $H_n^{K_1}$ are assumed to be chosen in $\mathcal{H}_{K_1}^{\text{Cont}}$ because $\mathcal{H}_{K_1}^{\text{Cont}} \subset \mathcal{H}_{K_1}$. Similarly, we may assume that $H_n^{K_2} \in \mathcal{H}_{K_2}^{\text{Cont}}$. It is proved also in \cite{v} that the following 2-cube is acyclic for all $n$.
    \begin{align}\label{ham2cube}
        \includegraphics[scale=1.23]{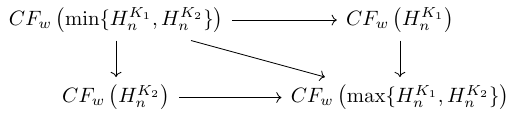}
    \end{align}
    Note that $\{\min\{H_n^{K_1}, H_n^{K_2}\}\}$ is a cofinal sequence of $\mathcal{H}_{K_1 \cup K_2}^{\text{Cont}}$ and $\{\max\{H_n^{K_1}, H_n^{K_2}\}\}$ is a cofinal sequence of $\mathcal{H}_{K_1 \cap K_2}^{\text{Cont}}$. With this in mind, we need to prove that $S^1$-equivariant version of \eqref{ham2cube}, which can be seen below, is also acyclic.
    \begin{align}\label{ham2cubes1}
        \includegraphics[scale=1.23]{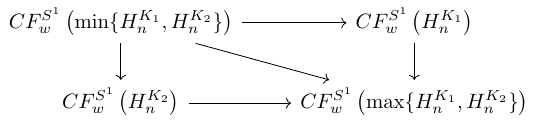}
    \end{align}
    As far as the acyclicity is concerned, it suffices to consider the case when the formal variable $T$ of the Novikov ring $\Lambda_{\geq 0}$ is zero. Hence, the diagonal map of \eqref{ham2cubes1} can be assumed to be the zero map and other maps send $u^k \otimes x$ to $u^k \otimes x$. The maximally iterated cone of \eqref{ham2cubes1} is given by $(\Lambda_{\geq 0}[u] \otimes \mathbf{C}, \text{id}\otimes d)$ where $(\mathbf{C}, d)$ is the  maximally iterated cone of \eqref{ham2cube} with $T = 0$. Since $(\mathbf{C}, d)$ is acyclic, $(\Lambda_{\geq 0}[u] \otimes \mathbf{C}, \text{id}\otimes d)$ is also acyclic. After taking telescope and completion on each vertex of \eqref{ham2cubes1}, the acyclicity is still preserved and thus we can derive a Mayer-Vietoris sequence as desired. Detailed explanation and some algebraic backgrounds can be found in \cite{v}.\\
\end{proof}
\section{Relative symplectic capacities}
In this section, we generalize the Gutt-Hutchings symplectic capacities $c^{GH}_k (k = 1, 2, 3, \dots)$ defined in \cite{gh} and symplectic (co)homology capacity $c^{SH}$ defined in \cite{fhw}.
\subsection{Ingredients} To cook up relative symplectic capacities, we need to get our ingredients ready. Many of them are borrowed and generalized from \cite{gh}. Let's fix our notation first. \\

\noindent\textbf{Notation.} From this point on, unless otherwise stated, $(M, \omega)$ will mean a closed symplectic manifold which is symplectically aspherical and $K \subset M$ will mean a compact domain with contact type and index-bounded boundary. Also, Denote the set $\{ x \in\RR \mid x < 0\}$ by $\RR_{-}$.\\

\begin{proposition}\label{building}
    \begin{enumerate}[label=(\alph*)]
        \item For each $L \in \RR_{-} \cup \{-\infty\}$, there exists a map
        \begin{align*}
            \iota_L : \shn \longrightarrow SH^{S^1, -}_M(K).
        \end{align*}
        Also, if $L_1, L_2 \in \RR_{-} \cup \{-\infty\}$ and $L_1 < L_2$, then there is a map
    \begin{align*}
        \iota_{L_1, L_2} : SH^{S^1,-,>L_2}_M(K) \longrightarrow SH^{S^1,-,>L_1}_M(K). 
    \end{align*}
    These maps form a direct system and the direct limit is 
    \begin{align*}
        \varinjlim_{L \to -\infty} SH^{S^1, - ,> L}_M(K) = SH^{S^1, -}_M(K).
    \end{align*}
    \begin{center}
        \begin{tikzcd}
SH^{S^1,-,>L_2}_M(K)  \arrow{dr}{\iota_{L_2}}  \arrow{rr}{\iota_{L_1,L_2}} & &SH^{S^1,-,>L_1}_M(K) \arrow{dl}{\iota_{L_1}} \\
& SH^{S^1,-}_M(K)\\
\end{tikzcd}

    \end{center}
    \item For each $L \in \RR_{-} \cup \{-\infty\}$, there is a map
    \begin{align*}
        U_L : \shn \longrightarrow \shn.
    \end{align*}
    Moreover, if $L_1 < L_2$, then $U_{L_1} \circ \iota_{L_1, L_2} = \iota_{L_1, L_2} \circ U_{L_2}$ and hence we can define $\displaystyle U = \varinjlim_{L \to -\infty} U_L$. \\
    \begin{center}
        \begin{tikzcd}
            SH^{S^1, -, >L_2}_M(K) \arrow{r}{\iota_{L_1, L_2}} \arrow{d}{U_{L_2}} &SH^{S^1, -, >L_1}_M(K)  \arrow{d}{U_{L_1}}\\
            SH^{S^1, -, >L_2}_M(K) \arrow{r}{\iota_{L_1, L_2}}  &SH^{S^1, -, >L_1}_M(K)     \\        
        \end{tikzcd}
    \end{center}
    \item There is a map
    \begin{align*}
        \delta : SH^{S^1,-}_M(K; \Lambda) \to H(K, \partial K ; \Lambda)  \otimes H(BS^1; \Lambda).\\
    \end{align*}
    \item If $L \in \RR_{-} \cup \{-\infty\}$ and $r > 0$, then there are isomorphisms
    \begin{align*}
        SH^{S^1,-}_{(M, \omega)}(K) \cong SH^{S^1,-}_{(M, r\omega)}(K) \,\,\text{and}\,\,SH^{S^1,-, >L}_{(M, \omega)}(K) \cong SH^{S^1,-, >rL}_{(M, r\omega)}(K).\\ 
    \end{align*}
    \end{enumerate}
    
\end{proposition}
\begin{proof}
(a) Let $H : S^1 \times M \to \RR$ be a contact type $K$-admissible Hamiltonian function. For each $L \in \RR_{-} \cup \{-\infty\}$, there exists an inclusion map
\begin{align*}
    CF_w^{S^1, > L}(H) \lr CF_w^{S^1}(H).
\end{align*}
Also, for $L_1 < L_2$, there is an inclusion map
\begin{align*}
     CF_w^{S^1, >L_2}(H) \longrightarrow CF_w^{S^1, >L_1}(H).
\end{align*}
These inclusion maps induce maps
\begin{align}\label{direct}
   CF_w^{S^1, -, > L}(H)  \lr  CF_w^{S^1, -}(H)
\end{align}
and
\begin{align*}
    CF_w^{S^1, -, >L_2}(H) \longrightarrow  CF_w^{S^1, -, >L_1}(H)
\end{align*}
after factoring out by $CF^{S^1, >-\epsilon}_w(H)$ for small enough $\epsilon > 0$. Furthermore, these maps form a commutative triangle below.
\begin{center}
    \includegraphics[scale=1.23]{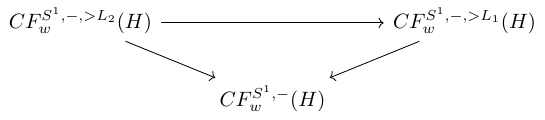}
\end{center}
By the definition of direct limit, we have a canonical map
\begin{align}\label{direct2}
    CF_w^{S^1, -}(H) \lr \varinjlim_{H \in \mathcal{H}^{\text{Cont}}_K} CF_w^{S^1, -}(H).
\end{align}
Combining maps in \eqref{direct} and \eqref{direct2}, we have
\begin{align*}
    CF_w^{S^1, -, > L}(H)  \lr \varinjlim_{H \in \mathcal{H}_K^{\text{Cont}}} CF_w^{S^1, -}(H).
\end{align*}
Then the universal property of direct limit guarantees that there exists a canonical map 
\begin{align}\label{direct3}
    \varinjlim_{H \in \mathcal{H}_K^{\text{Cont}}} CF_w^{S^1, -, >L}(H) \lr \varinjlim_{H \in \mathcal{H}_K^{\text{Cont}}} CF_w^{S^1, -}(H).
\end{align}
Now we prove that the map \eqref{direct3} induces a map on completion. For notational convenience, let $\mathbf{C}^{>L} = \displaystyle\varinjlim_{H \in \mathcal{H}_K^{\text{Cont}}} CF_w^{S^1, -, >L}(H)$ and $\mathbf{C} = \displaystyle\varinjlim_{H \in \mathcal{H}_K^{\text{Cont}}} CF_w^{S^1, -}(H)$. First of all, by tensoring $\Lambda_{\geq 0} / \Lambda_{\geq r}$ over $\Lambda_{\geq 0}$ on each side of \eqref{direct3}, we get
\begin{align}\label{inverse}
   \mathbf{C}^{>L} \ton \Lambda_{\geq 0} / \Lambda_{\geq r} \lr  \mathbf{C} \ton \Lambda_{\geq 0} / \Lambda_{\geq r}.
\end{align}
By the definition of inverse limit, we have a map
\begin{align}\label{inverse1}
    \widehat{\mathbf{C}^{>L}} \lr \mathbf{C}^{>L} \ton \Lambda_{\geq 0} / \Lambda_{\geq r}.
\end{align}
Composing maps \eqref{inverse} and \eqref{inverse1}, we obtain a map
\begin{align*}
    \widehat{\mathbf{C}^{>L}}  \lr  \mathbf{C} \ton \Lambda_{\geq 0} / \Lambda_{\geq r}.
\end{align*}
Moreover, for $r' > r$, we have a following commutative diagram.
\begin{center}
    \includegraphics[scale=1.23]{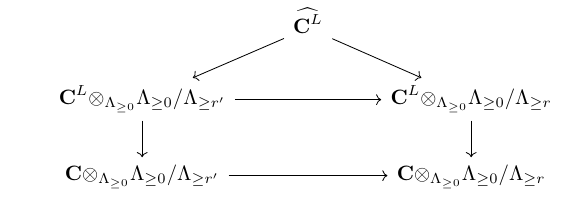}
\end{center}
Hence, by the universal property of inverse limit, we can assert that there exists a map
\begin{align*}
    \widehat{\varinjlim_{H \in \mathcal{H}_K^{\text{Cont}}}} CF_w^{S^1, -, >L}(H) = \widehat{\mathbf{C}^{>L}}  \lr  \widehat{\mathbf{C}} = \widehat{\varinjlim_{H \in \mathcal{H}_K^{\text{Cont}}}} CF_w^{S^1, -}(H).
\end{align*}
Taking homology, we finally get the desired map $\iota_L$
\begin{align*}
   SH^{S^1,-,>L}_M(K) = H \left( \widehat{\varinjlim_{H \in \mathcal{H}_K^{\text{Cont}}}} CF_w^{S^1, -, >L}(H)\right) \lr H \left( \widehat{\varinjlim_{H \in \mathcal{H}_K^{\text{Cont}}}} CF_w^{S^1, -}(H)\right) = SH^{S^1,-}_M(K).
\end{align*}
In the same fashion, we can construct the map $\iota_{L_1, L_2}$.
\\
(b) Let $H : S^1 \times M \to \RR$ be a contact type $K$-admissible Hamiltonian function. Define a map $U : CF^{S^1}_w(H) \,\,\lr\,\, CF^{S^1}_w(H)[+2]$ by 
    \begin{align*}
        U(u^k \otimes x) = \begin{cases}
            u^{k-1} \otimes x &\text{if} \,\,k \geq 1\\
            0 &\text{if} \,\, k=0.
        \end{cases}
    \end{align*}
Since this map $U$ respects the action, we obtain the desired map $U_L$ in the same way that we did in the proof of (a).\\
(c) Let $H : S^1 \times M \to \RR$ be a contact type $K$-admissible Hamiltonian function. We have a short exact sequence of chain complexes
\begin{align*}
    0 \longrightarrow CF_w^{S^1, >-\epsilon}(H) \longrightarrow CF_w^{S^1}(H) \longrightarrow CF_w^{S^1,-}(H) \longrightarrow 0
\end{align*}
for small enough $\epsilon > 0$. Since direct limit is an exact functor, we also have
\begin{align*}
    0 \longrightarrow \varinjlim_{H \in \mathcal{H}^{\text{Cont}}_K} CF_w^{S^1, >-\epsilon}(H) \longrightarrow \varinjlim_{H \in \mathcal{H}^{\text{Cont}}_K} CF_w^{S^1}(H) \longrightarrow 
\varinjlim_{H \in \mathcal{H}^{\text{Cont}}_K} CF_w^{S^1,-}(H) \longrightarrow 0.
\end{align*}
Note that $CF_w^{S^1,-}(H)$ is a free $\Lambda_{\geq 0}$-module generated by nonconstant periodic orbits of $H$ and hence it is a flat module over $\Lambda_{\geq 0}$. Since the direct limit of flat modules is still flat, we know that $\displaystyle \varinjlim_{H \in \mathcal{H}^{\text{Cont}}_K} CF_w^{S^1,-}(H)$ is a flat $\Lambda_{\geq 0}$-module. This flatness implies that 
\begin{align*}
    \text{Tor}_1^{\Lambda_{\geq 0}}\left(\varinjlim_{H \in \mathcal{H}^{\text{Cont}}_K} CF_w^{S^1,-}(H), \Lambda_{\geq 0}/ \Lambda_{\geq r}\right) = 0
\end{align*}
for each $r > 0$. Therefore, we have the following long exact sequence of chain complexes:
\begin{align*}
    0 \longrightarrow \varinjlim_{H \in \mathcal{H}^{\text{Cont}}_K} CF_w^{S^1, >-\epsilon}(H) \ton \Lambda_{\geq 0}/ \Lambda_{\geq r} &\longrightarrow \varinjlim_{H \in \mathcal{H}^{\text{Cont}}_K} CF_w^{S^1}(H) \ton \Lambda_{\geq 0}/ \Lambda_{\geq r}\\&\longrightarrow 
\varinjlim_{H \in \mathcal{H}^{\text{Cont}}_K} CF_w^{S^1,-}(H)\ton \Lambda_{\geq 0}/ \Lambda_{\geq r} \longrightarrow 0.
\end{align*}
For $r' > r$, the projection $\Lambda_{\geq 0}/ \Lambda_{\geq r'} \to \Lambda_{\geq 0}/ \Lambda_{\geq r}$ is surjective and, for any $\Lambda_{\geq 0}$-module $A$, $A \ton \Lambda_{\geq 0}/ \Lambda_{\geq r'} \to A \ton \Lambda_{\geq 0}/ \Lambda_{\geq r}$ is also surjective due to the right exactness of tensor product. Then by the Mittag-Leffler theorem for the inverse limit, we still have a short exact sequence
\begin{align*}
    0 \longrightarrow \widehat{\varinjlim_{H \in \mathcal{H}^{\text{Cont}}_K}} CF_w^{S^1, >-\epsilon}(H) \longrightarrow \widehat{\varinjlim_{H \in \mathcal{H}^{\text{Cont}}_K}} CF_w^{S^1}(H) \longrightarrow 
\widehat{\varinjlim_{H \in \mathcal{H}^{\text{Cont}}_K}} CF_w^{S^1,-}(H) \longrightarrow 0.
\end{align*}
From this short exact sequence, we obtain a long exact sequence
\begin{center}
    \begin{tikzcd}
    SH^{S^1, >-\epsilon}_M(K) \arrow{r} &SH^{S^1}_M(K) \arrow{d}\\
    &SH^{S^1,-}_M(K) \arrow{lu}{[+1]}
\end{tikzcd}

\end{center}
where 
\begin{align}\label{delta}
    SH^{S^1, >-\epsilon}_M(K) = H \left(\widehat{\varinjlim_{H \in \mathcal{H}^{\text{Cont}}_K}}CF^{S^1, >-\epsilon}(H)\right).
\end{align}
To compute \eqref{delta}, note that the completed chain complex $\displaystyle\widehat{\varinjlim_{H \in \mathcal{H}^{\text{Cont}}_K}}CF_w^{S^1, >-\epsilon}(H)$ consists only of (equivalence classes of) critical points of $H$'s. Since the complex $\displaystyle\varinjlim_{H \in \mathcal{H}^{\text{Cont}}_K}CF_w^{S^1, >-\epsilon}(H)$ is complete, we can simply get rid of the hat symbol, that is,
\begin{align*}
    \widehat{\varinjlim_{H \in \mathcal{H}^{\text{Cont}}_K}}CF_w^{S^1, >-\epsilon}(H) \cong \varinjlim_{H \in \mathcal{H}^{\text{Cont}}_K}CF_w^{S^1, >-\epsilon}(H).
\end{align*}
Also, by Remark \ref{convention}, we have
\begin{align*}
    CF_w^{S^1, >-\epsilon}(H) \cong CM_w^{S^1}(-H|_{S^1 \times K})
\end{align*}
where $CM_w^{S^1}(-H|_{S^1 \times K})$ is the weighted $S^1$-equivariant Morse complex of $-H|_{S^1 \times K}$. Since the direct limit commutes with homology, it suffices to show that 
\begin{align}\label{homology}
    H\left(CM_w^{S^1}(-H|_{S^1 \times K}) \ton \Lambda\right) \cong H(K, \partial K; \Lambda) \otimes H(BS^1 ;\Lambda)
\end{align}
 %\footnote{For later use, we note the following: if the first Chern class $c_1(M)$ of $M$ is zero, then $k$-th cohomology of the left hand side of \eqref{homology} corresponds to the $(n-k)$-th cohomology of the right hand side of \eqref{homology}.}.%
By Remark \ref{clw}, 
\begin{align}\label{clwmorse}
    H(CM_w(-H|_{S^1 \times K} ; \Lambda)) \cong H(CM(-H|_{S^1 \times K} ; \Lambda)).
\end{align}
Following the proof in \cite{vi}, it can be verified that
\begin{align}\label{vitiso}
    H\left(CM(-H|_{S^1 \times K} ; \Lambda)\right) \cong H(K, \partial K ; \Lambda).
\end{align}
Since the differential of $CM_w(-H|_{S^1 \times K}) ; \Lambda) \ton \Lambda[u]$ is the tensor product of the Morse differential on $CM(-H|_{S^1 \times K}) ; \Lambda)$ and the identity on $\Lambda[u]$, we have
\begin{align*}
    H\left(CM_w^{S^1}(-H|_{S^1 \times K} ; \Lambda) \ton \Lambda\right) & = H \left(CM_w(-H|_{S^1 \times K} ; \Lambda\right) \ton \Lambda[u]) \\& \cong H \left(CM_w(-H|_{S^1 \times K} ; \Lambda)\right) \ton \Lambda[u] \\& \cong H(CM(-H|_{S^1 \times K} ; \Lambda)) \ton \Lambda[u] &&\text{By}\,\,\eqref{clwmorse} \\&\cong H(K, \partial K ;\Lambda) \ton \Lambda[u] &&\text{By}\,\,\eqref{vitiso} \\& \cong H(K, \partial K ; \Lambda)  \ton H(BS^1; \Lambda).
\end{align*}
(d) Let $H : S^1 \times M \to \RR$ be a contact type $K$-admissible Hamiltonian function and $r > 0$. Let $CF_w^{S^1}(H, (M, \omega))$ denote the Floer complex of $H$ on the symplectic manifold $(M, \omega)$. Then we have an isomorphism of chain complexes
\begin{align*}
    CF_w^{S^1, >L}(H, (M, \omega)) \cong CF_w^{S^1, > rL}(rH, (M, r\omega))
\end{align*}
because
 \begin{itemize}
        \item $x \in CF(H, M, \omega)$ if and only if $x \in CF(rH, M, r\omega)$,\\
       \item $X_{(H,M,\omega)} = X_{(rH,M,r\omega)}$ where $X_{(H,M,\omega)}$ denotes the Hamiltonian vector field of $H$ on $(M, \omega)$,\\
        \item $u : \RR \times S^1 \longrightarrow M$ is a Floer trajectory connecting $x$ and $y$ in $CF(H, M, \omega)$ if and only if it connects $x$ and $y$ in $CF(rH, M, r\omega)$, and\\
        \item $\mathcal{A}_{(H, M, \omega)} ([x, \Tilde{x}]) > L$ if and only if $\mathcal{A}_{(H, M, r \omega)} ([x, \Tilde{x}]) > rL$ where $\mathcal{A}_{(H,M,\omega)}$ denotes the action functional of $H$ defined on $(M, \omega)$.\\
    \end{itemize}
Note that the first three bullet points imply that $HF(rH, M, r \omega) \cong HF(H, M, \omega)$ and combining the last bullet we can prove that $HF^{>rL}(rH, M, r \omega) \cong HF^{>L}(H, M, \omega)$. Taking direct limit and completion, we can conclude the proof.\\
\end{proof}
\begin{corollary}\label{firstgh}
     For each $L \in \RR_{-} \cup \{-\infty\}$, we have a following exact triangle.
      \begin{align}\label{ET}
        \includegraphics[scale=1.23]{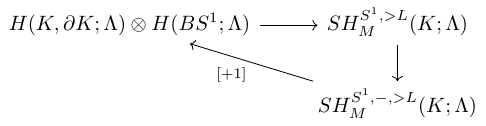}
    \end{align}
    
   \end{corollary}
\begin{proof}
    In view of the proof of (c) of Proposition \ref{building}, we can choose small enough $\epsilon > 0$ such that $SH^{S^1,> - \epsilon}_M(K ; \Lambda) \cong H(K, \partial K; \Lambda) \otimes H(BS^1 ; \Lambda)$. Let $H \in \mathcal{H}_K^{\text{Cont}}$. We may assume that $L < -\epsilon$. Then we have a short exact sequence of chain complexes
    \begin{align*}
        0 \lr CF_w^{S^1,> -\epsilon}(H) \lr CF_w^{S^1,> L}(H) \lr CF_w^{> L}(H)/CF_w^{S^1,> -\epsilon}(H) \lr 0.
    \end{align*}
    It can be verified similarly as in the proof of (c) of Proposition \ref{building} that we can derive the desired exact triangle.
    \\
\end{proof}
\begin{remark}\label{futureuse}
    \begin{enumerate}[label=(\alph*)]
        \item For future use, we label the horizontal map of \eqref{ET} by $j_L^{S^1}$. Explicitly speaking, for each $L \in \RR_- \cup \{-\infty\}$, there exists a map
        \begin{align*}
            j^{S^1}_L : H(K, \partial K; \Lambda) \otimes H(BS^1 ; \Lambda) \lr SH^{S^1, >L}_M(K ; \Lambda).\\
        \end{align*}
        \item One can easily check that the degree 1 map in the exact triangle \eqref{ET} is $\delta \circ \iota_L$.\\
    \end{enumerate}
\end{remark}
Now we shall construct a relative version of transfer morphism. Before doing so, we recall the definition of $\mathcal{H}_{\text{stair}}$ defined in \cite{gutt} and \cite{gh}.
\begin{definition}[Gutt \cite{gutt}, Gutt and Hutchings  \cite{gh}]\label{stair} Let $(V, \omega_V)$ and $(W, \omega_W)$ be Liouville domains. Suppose that the interior of $V$ is symplectically embedded in $W$. Given small $\delta > 0$, there exists a neighborhood $U$ of $\partial V$ in $W - \text{int}(V)$ such that 
\begin{align*}
    (U, \omega_W) \cong ([0, \delta] \times \partial V, d(e^{\rho} \lambda_V)).
\end{align*}
A Hamiltonian function $H : S^1 \times \widehat{W} \to \RR$ is in $\mathcal{H}_{\text{stair}}(V, W)$ if it satisfies the following conditions:
\begin{enumerate}[label=(\alph*)]
    \item H is negative and $C^2$-small on $S^1 \times V$. Moreover, $H > - \epsilon$ on $S^1 \times V$ where $\epsilon = \frac{1}{2} \min \{\text{Spec}(\partial V, \lambda_V), \text{Spec}(\partial W, \lambda_W) \}$.\\
    \item There exists $\eta \in (0, \frac{1}{4}\delta)$ such that $H(t,p, \rho)$ is $C^2$-close to $h_1(e^{\rho})$ on $S^1  \times \partial V \times [0, \eta]$ for some strictly convex increasing function $h_1$.\\
    \item $H(t, p, \rho) = \beta e^{\rho} + \beta'$ on $S^1  \times \partial V \times [\eta, \delta - \eta]$ where $\beta > 0$, $\beta \notin \text{Spec}(\partial V, \lambda_V) \cup \text{Spec}(\partial W, \lambda_W)$ and $\beta' \in \RR$.\\
    \item $H(t,p, \rho)$ is $C^2$-close to $h_2(e^{\rho})$ on $S^1  \times \partial V \times [\delta - \eta, \delta]$ for some strictly concave increasing function $h_2$.\\
    \item $H$ is $C^2$-close to a constant function on $S^1 \times (W - (V \cup U))$.\\
    \item There exists $\eta' > 0$ such that $H(t,p, \rho)$ is $C^2$-close to $h_3(e^{\rho})$ on $S^1  \times \partial W \times [0, \eta']$ for some strictly convex increasing function $h_3$.\\   
    \item $H(t, p, \rho) = \mu e^{\rho} + \mu'$ on $S^1  \times \partial W \times [\eta', \infty]$ where $\mu > 0$, $\mu \notin \text{Spec}(\partial V, \lambda_V) \cup \text{Spec}(\partial W, \lambda_W)$ and $\mu' \in \RR$.\\ 
\end{enumerate}

\end{definition}
In \cite{gutt} and \cite{gh}, transfer morphism is constructed via Hamiltonian functions in the Definition \ref{stair}. To bulid a relative version of transfer morphism, we need to add one more step into the staircase and it can be done with the aid of the Definition \ref{stair}.
\begin{definition}
Let $(M, \omega)$ and $(M', \omega')$ be closed symplectic manifolds not necessarily symplectically aspherical and $K \subset M$ and $K' \subset M'$ be compact domains with contact type boundaries. Let $\phi : (M, \omega) \hookrightarrow (M', \omega')$ be a symplectic embedding with $\phi(\text{int}(K)) \subset K'$. We identify $(M, K)$ with its image $(\phi(M), \phi(K))$ inside $(M', K')$. We identify the neighborhood $(U, \omega)$ of $\partial K$ with $([0, \delta] \times \partial K, d(e^{\rho}\lambda_{K}))$ and the neighborhood $(U', \omega)$ of $\partial K'$ with $([0, \delta'] \times \partial K', d(e^{\rho}\lambda_{K'}))$.
\begin{enumerate}[label=(\alph*)]
    \item A Hamiltonian function $H : S^1 \times M' \to \RR$ is in $\mathcal{H}_{\text{stair}}(K, K', M')$ if it satisfies the following conditions:
    \begin{itemize}
        \item There exists $H' \in \mathcal{H}_{\text{stair}}(K, K')$ such that $H|_{S^1 \times K'} = H'|_{S^1 \times K'}$.\\
        \item There exists $\eta' \in (0, \frac{1}{4}\delta')$ such that $H(t, p, \rho)$ is $C^2$-close to $h_1(e^{\rho})$ on $S^1 \times \partial K' \times [0, \eta']$ for some strictly increasing and convex function $h_1$.\\
        \item $H(t,p, \rho) = \beta e^{\rho} + \beta'$ on $S^1 \times \partial K' \times [\eta', \delta' - \eta']$ where $\beta > 0$, $\beta \notin \text{Spec}(\partial K, \lambda_K) \cup \text{Spec}(\partial K', \lambda_K')$ and $\beta' \in \RR$.\\
        \item $H(t, p, \rho)$ is $C^2$-close to $h_2(e^{\rho})$ on $S^1 \times \partial K' \times [\delta' - \eta', \delta']$ for some strictly increasing and convex function $h_2$.\\
        \item $H$ is $C^2$-close to a constant function on $S^1 \times (M - U')$.\\
    \end{itemize}
    \item For $H \in \mathcal{H}_{\text{stair}}(K, K', M')$, define $H^K : S^1 \times M \to \RR$ as follows:
    \begin{itemize}
        \item $H^K = H$ on $S^1 \times K'$.\\
        \item $H^K$ is $C^2$-close to a constant function on $S^1 \times (M - K')$.\\
    \end{itemize}
    \end{enumerate}
\end{definition}
Now let's get back to our assumption that $(M, \omega)$ and $(M', \omega')$ be closed symplectically aspherical symplectic manifolds and $K \subset M$ and $K' \subset M'$ be compact domains with contact type and index-bounded boundaries. Let $H \in \mathcal{H}_{\text{stair}}(K, K', M')$. Then due to Remark \ref{homkill}, one can observe that
\begin{align}\label{same}
    \widehat{\varinjlim_{H \in \mathcal{H}_{\text{stair}}(K, K', M')}} CF^{S^1}_w(H) = \widehat{\varinjlim_{H \in \mathcal{H}_{\text{stair}}(K, K', M')}} CF^{S^1}_w(H^K) 
\end{align}
and
\begin{align}\label{same2}
    \widehat{\varinjlim_{H \in \mathcal{H}_{\text{stair}}(K, K', M')}} CF^{S^1, >-\epsilon}_w(H) = \widehat{\varinjlim_{H \in \mathcal{H}_{\text{stair}}(K, K', M')}} CF^{S^1, >-\epsilon}_w(H^K). 
\end{align}
Suppose that $H' \in \mathcal{H}^{\text{Cont}}_{K'}$ satisfies $H' \leq H$. Then we have a (weighted) continuation map
\begin{align*}
    CF^{S^1}_w(H') \longrightarrow CF^{S^1}_w(H).
\end{align*}
Then, after taking direct limit and completion, we have 
\begin{align}\label{transfercomp}
    \widehat{\varinjlim_{H' \in \mathcal{H}^{\text{Cont}}_{K'}}} CF^{S^1}_w(H') \lr \widehat{\varinjlim_{H \in \mathcal{H}_{\text{stair}}(K, K', M')}}CF^{S^1}_w(H).
\end{align}
For more detailed exposition of the construction of the map \eqref{transfercomp}, see the proof of (a) of Proposition \ref{building}. By \eqref{same}, the map \eqref{transfercomp} can be rewritten as
\begin{align*}
    \widehat{\varinjlim_{H' \in \mathcal{H}^{\text{Cont}}_{K'}}} CF^{S^1}_w(H') \lr \widehat{\varinjlim_{H \in \mathcal{H}_{\text{stair}}(K, K', M')}} CF^{S^1}_w(H^K).
\end{align*}
Finally, on homology level, we obtain the following map.
\begin{align}\label{map1}
    SH^{S^1}_{M'}(K') = H \left( \widehat{\varinjlim_{H' \in \mathcal{H}^{\text{Cont}}_{K'}}} CF^{S^1}_w(H')\right) \lr H \left(\widehat{\varinjlim_{H \in \mathcal{H}_{\text{stair}}(K, K', M')}}CF^{S^1}_w(H^K) \right) = SH^{S^1}_{M'}(K).
\end{align}
The verification of the last equality in \eqref{map1} can be found in \cite{bo}. It remains to construct a map from $SH_{M'}^{S^1}(K)$ to $SH_{M}^{S^1}(K)$. For $H \in \mathcal{H}^{\text{Cont}}_K$, let $CF_w^{S^1}(H, M')$ be the weighted $S^1$-equivariant Floer complex of $H$ thinking that every Hamiltonian orbit and every Floer trajectory live in $M'$. We can define $CF_w^{S^1}(H, M)$ in the same fashion. There is a projection map
\begin{align}\label{smallbig}
    CF_w^{S^1}(H, M') \longrightarrow CF_w^{S^1}(H, M).
\end{align}
After taking direct limit and completion, we have a map of cochain complexes
\begin{align}\label{embedding}
    \widehat{\varinjlim_{H \in \mathcal{H}^{\text{Cont}}_K}} CF_w^{S^1}(H, M') \longrightarrow
 \widehat{\varinjlim_{H \in \mathcal{H}^{\text{Cont}}_K}} CF_w^{S^1}(H, M).
\end{align}
Therefore, the map \eqref{embedding} induces a map on homology
\begin{align}\label{map2}
    H\left(\widehat{\varinjlim_{H \in \mathcal{H}^{\text{Cont}}_K}} CF_w^{S^1}(H, M')\right)= SH_{M'}^{S^1}(K) \longrightarrow SH_{M}^{S^1}(K) = H\left( \widehat{\varinjlim_{H \in \mathcal{H}^{\text{Cont}}_K}} CF_w^{S^1}(H, M) \right).
\end{align}
Composing maps \eqref{map1} and \eqref{map2}, we finally have a transfer morphism
\begin{align}\label{transfer}
    SH_{M'}^{S^1}(K') \longrightarrow SH_{M}^{S^1}(K).
\end{align}
Note that if $M = M'$, then the map \ref{transfer} is a restriction map $r^{K'}_{K}$ introduced in \cite{v}. Looking back the construction of the map \eqref{transfer}, every linking map increases or preserves the action and hence we also have maps
\begin{align*}
     \Phi: SH_{M'}^{S^1,-}(K') \longrightarrow SH_{M}^{S^1,-}(K) 
\end{align*}
and
\begin{align*}
    \Phi^L : SH_{M'}^{S^1,-,>L}(K') \longrightarrow SH_{M}^{S^1,-,>L}(K)
\end{align*}
for each $L \in \RR_{-} \cup \{-\infty\}$. We can extend these maps to cohomology with coefficients in $\Lambda$ and we still denote their extensions by $\Phi$ and $\Phi^L$, respectively.

\begin{proposition}\label{building2}
     Let $(M, \omega)$ and $(M', \omega')$ be closed symplectically aspherical symplectic manifolds and $K \subset M$ and $K' \subset M'$ be compact domains with contact type and index-bounded boundaries. Let $\phi : (M, \omega) \hookrightarrow (M', \omega')$ be a symplectic embedding with $\text{int}\,(\phi(K)) \subset K'$. Then the maps $\Phi$ and $\Phi^L$ have the following properties:
     \begin{enumerate}[label=(\alph*)]
        \item  For each $L \in \RR_{-} \cup \{-\infty\}$, there exists a map $\Phi^L : SH^{S^1,-,>L}_{M'}(K') \to \shn$. Also, if $L_1, L_2 \in \RR_{-} \cup \{-\infty\}$ and $L_1 < L_2$, then $\Phi^{L_1} \circ \iota_{L_1, L_2} = \iota_{L_1, L_2} \circ \Phi^{L_2}$ and $\displaystyle\Phi = \varinjlim_{L \to -\infty} \Phi^L$. 
        \begin{center}
          \begin{tikzcd}
   SH^{S^1,-,>L_2}_{M'}(K') \arrow{r}{\Phi^{L_2}} \arrow{d}{\iota_{L_1, L_2}} &SH^{S^1,-,>L_2}_M(K) \arrow{d}{\iota_{L_1, L_2}}\\
   SH^{S^1,-,>L_1}_{M'}(K') \arrow{r}{\Phi^{L_1}} &SH^{S^1,-,>L_1}_M(K)\\
\end{tikzcd}

        \end{center}
        \item For each $L \in \RR_{-} \cup \{-\infty\}$, $\iota_L \circ \Phi^L = \Phi \circ \iota_L$.\\
        \begin{center}
            \begin{tikzcd}
    SH^{S^1, -, >L}_{M'}(K') \arrow{r}{\Phi^L} \arrow{d}{\iota_L} &SH^{S^1, -, >L}_M(K) \arrow{d}{\iota_L}\\
    SH^{S^1, -}_{M'}(K') \arrow{r}{\Phi}&SH^{S^1, -}_M(K)\\
\end{tikzcd}
        \end{center}
        \item  For each $L \in \RR_{-} \cup \{-\infty\}$, $U_L \circ \Phi^L = \Phi^L \circ U_L$.\\
        \begin{center}
            \begin{tikzcd}
    SH^{S^1, -, >L}_{M'}(K') \arrow{r}{\Phi^L} \arrow{d}{U_L} &SH^{S^1, -, >L}_M(K) \arrow{d}{U_L}\\
    SH^{S^1, -, >L}_{M'}(K') \arrow{r}{\Phi^L}&SH^{S^1, -, >L}_M(K)\\
\end{tikzcd}
\end{center}
        \item Let $\rho : H(K', \partial K' ; \Lambda) \to H(K, \partial K; \Lambda)$ be the map induced by restriction. Then $\delta \circ \Phi = (\rho \otimes \text{id}) \circ \delta$.  
    \begin{center}
         \end{center}
        \end{enumerate}
    \end{proposition}
    \begin{proof}
    (a), (b), (c) Clear from the construction of $\Phi$ and $\Phi^L$. For details, see Chapter 8 of \cite{gh}.\\
    (d) Let $H \in \mathcal{H}_{\text{stair}}(K,K',M')$ and $H' \in \mathcal{H}_{K'}^{\text{Cont}}$. Suppose that $H' \leq H$. Consider the following commutative diagram
    \begin{center}
        \begin{tikzcd}
0 \arrow{r} & CF_w^{S^1, >-\epsilon}(H', M') \arrow{r}\arrow{d} &CF_w^{S^1}(H', M') \arrow{r}\arrow{d} &CF_w^{S^1,-}(H', M') \arrow{r}\arrow{d} &0\\
0 \arrow{r} &CF_w^{S^1, >-\epsilon}(H, M') \arrow{r}\arrow{d} &CF_w^{S^1}(H, M') \arrow{r}\arrow{d} &CF_w^{S^1,-}(H, M') \arrow{r}\arrow{d} &0\\
0 \arrow{r} &CF_w^{S^1, >-\epsilon}(H, M) \arrow{r} &CF_w^{S^1}(H, M)\arrow{r} &CF_w^{S^1,-}(H, M) \arrow{r} &0
\end{tikzcd}
     \end{center}
    where each row is exact and vertical maps are those in the construction of $\Phi$. More precisely, the first vertical maps are continuation maps and the sencond vertical maps are projection maps. By the same reason presented in the proof of (c) Proposition \ref{building}, each row remains exact after taking $\displaystyle\widehat{\varinjlim_{H' \in \mathcal{H}_{K'}^{\text{Cont}}}}$ in the first row and taking $\displaystyle\widehat{\varinjlim_{H \in \mathcal{H}_{\text{stair}(K, K', M')}}}$ in the second and third rows. So, we have a commutative diagram of exact sequences of chain complexes as below. 
     \begin{center}
        \begin{tikzcd}
0 \arrow{r} & \displaystyle\widehat{\varinjlim_{H'}}CF_w^{S^1, >-\epsilon}(H', M') \arrow{r}\arrow{d} &\displaystyle\widehat{\varinjlim_{H'}}CF_w^{S^1}(H', M') \arrow{r}\arrow{d} &\displaystyle\widehat{\varinjlim_{H'}}CF_w^{S^1,-}(H', M') \arrow{r}\arrow{d} &0\\
0 \arrow{r} &\displaystyle\widehat{\varinjlim_{H}}CF_w^{S^1, >-\epsilon}(H, M') \arrow{r}\arrow{d} &\displaystyle\widehat{\varinjlim_{H}}CF_w^{S^1}(H, M') \arrow{r}\arrow{d} &\displaystyle\widehat{\varinjlim_{H}}CF_w^{S^1,-}(H, M') \arrow{r}\arrow{d} &0\\
0 \arrow{r} &\displaystyle\widehat{\varinjlim_{H}}CF_w^{S^1, >-\epsilon}(H, M) \arrow{r} &\displaystyle\widehat{\varinjlim_{H}}CF_w^{S^1}(H, M)\arrow{r} &\displaystyle\widehat{\varinjlim_{H}}CF_w^{S^1,-}(H, M) \arrow{r} &0
\end{tikzcd}
    \end{center}
    Note that in the second and third row of the diagram above, we can replace $H$ by $H^K$ due to \eqref{same} and \eqref{same2} and each row remains exact after tensoring $\Lambda$ over $\Lambda_{\geq 0}$. Recall from the proof of that (c) of Proposition \ref{building} that
    \begin{align*}
       H \left(\displaystyle\widehat{\varinjlim_{H'}}CF_w^{S^1, >-\epsilon}(H', M') \ton \Lambda\right) = H(K', \partial K';\Lambda )\otimes H(BS^1 ; \Lambda)
    \end{align*}
    and 
    \begin{align*}
        H \left( \displaystyle\widehat{\varinjlim_{H}}CF_w^{S^1, >-\epsilon}(H^K, M) \ton \Lambda \right) =  H(K, \partial K;\Lambda )\otimes H(BS^1; \Lambda).
    \end{align*}
     The long exact sequences of the first row and the third row (after tensoring with $\Lambda$ over $\Lambda_{\geq 0}$) yield a commutative diamgram as follows:
    \begin{center}
        \begin{tikzcd}
    SH_{M'}^{S^1,-}(K', \Lambda) \arrow{r}{\delta} \arrow{d}{\Phi} & H \left(\displaystyle\widehat{\varinjlim_{H}}CF_w^{S^1, >-\epsilon}(H, M')\right) = H(K', \partial K',\Lambda )\otimes H(BS^1, \Lambda) \arrow{d}{\rho \otimes \text{id}}\\
    SH_{M'}^{S^1,-}(K, \Lambda) \arrow{r}{\delta} & H \left( \displaystyle\widehat{\varinjlim_{H}}CF_w^{S^1, >-\epsilon}(H^K, M) \right) =  H(K, \partial K,\Lambda )\otimes H(BS^1, \Lambda).
\end{tikzcd}
    \end{center}
    \end{proof}
    \begin{remark}\label{nons1}
    Notice that the statements (a), (b) and (d) in Proposition \ref{building} and the statements (a) and (b) in Proposition \ref{building2} deal with cohomologies with coefficients in $\Lambda_{\geq 0}$. But they remain true even though we replace their coefficients by $\Lambda$. Moreover, one can observe that every statement in Proposition \ref{building}, Corollary \ref{firstgh} and Proposition \ref{building2} still holds in the non $S^1$-equivariant case, that is, we can drop the superscript $S^1$ in all cases. We will use the same notation such as $\iota_L, \iota_{L_1, L_2}, \delta, \Phi$ and so on for the non $S^1$-equivariant case and this notation will be clear from the context. For example, for the non $S^1$-equivariant case, the map $\delta$ defined in Proposition \ref{building} should be thought of as
    \begin{align*}
        \delta : SH^{-}_M(K; \Lambda) \lr H(H, \partial K; \Lambda)
    \end{align*}
    and hence equation in (d) of Proposition \ref{building2} might as well be transformed into $\delta \circ \Phi = \rho \circ \delta$. One map that will be frequently used is the non $S^1$-equivariant version of $j^{S^1}_L$ introduced in Remark \ref{futureuse} and we denote this map by $j_L$ dropping the superscript $S^1$. Analogous proof of Corollary \ref{firstgh} can be used to construct an exact triangle 
       \begin{align}\label{triforsh}
            \includegraphics[scale=1.2]{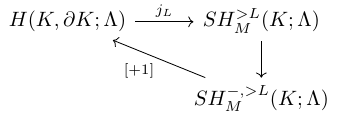}
        \end{align}
        and a map 
        \begin{align*}
            j_L : H(K, \partial K ; \Lambda) \lr SH^{>L}_M(K ; \Lambda).
        \end{align*}
  The degree 1 map in \eqref{triforsh} is also $\delta \circ \iota_L$.       \\
    
\end{remark}
\subsection{Relative Gutt-Hutchings capacities}
    In this subsection, we define a relative version of Gutt-Hutchings capacities. One more thing to note is that $[\omega^n|_K]$ defines a fundamental class of $H^*(K, \partial K; \Lambda)$. Here, $n = \frac{1}{2} \dim M$. We have set up all the background that we need to define the relative Gutt-Hutchings capacities. 
    \begin{definition}
         For each $k \in \NN$, define the $k$-th relative Gutt-Hutchings capacity $c_k^{GH}(M,K)$ of $K$ in $M$  by
\begin{align*}
    c_k^{GH}(M,K) = - \sup\left\{L < 0 \bigmid \begin{array}{cc}
         & \delta U^{k-1} \iota_L(\alpha) = [\omega^n|_K] \otimes [\text{pt}] \in H(K, \partial K ; \Lambda) \otimes H(BS^1 ; \Lambda) \\
         & \,\,\text{for some}\,\, \alpha \in SH^{S^1, -, >L}_M(K ; \Lambda)
    \end{array}
    \right\}.\\
\end{align*}
    \end{definition}
\begin{proposition}\label{ghc} $c^{GH}_k(K, M)$ of $K$ in $M$ has the following properties.
    \begin{enumerate}[label=(\alph*)]
        \item For $r>0$, $c_k^{GH}(M,r \omega, K) = rc_k^{GH}(M, \omega, K)$.\\
        \item $c_k^{GH}(M,K) \leq c_{k+1}^{GH}(M,K)$ for all $k \geq 1$.\\
        \item If there exists a symplectic embedding $\phi : (M, \omega) \hookrightarrow (M', \omega')$ with $\text{int} \,(\phi(K)) \subset K'$, then $c_k^{GH}(M, K) \leq c_k^{GH}(M', K')$ for all $k \geq 1$. \\
    \end{enumerate}
\end{proposition}
\begin{proof}
(a) \begin{align*}
c_k^{GH}(M, r\omega, K)& = -\sup\left\{ L < 0 \bigmid  \begin{array}{cc}
     &   \delta U^{k-1} \iota_L(\alpha) =[\omega^n|_K]\otimes [\text{pt}] \in H(K, \partial K ; \Lambda) \otimes H(BS^1 ; \Lambda) \\
     & \,\,\text{for some}\,\, \alpha \in SH^{S^1, - ,>L}_{(M, r \omega)}(K)
\end{array} \right\}\\
&\overset{(*)}{=} - \sup\left\{ L< 0 \bigmid  \begin{array}{cc}
     &   \delta U^{k-1} \iota_{\frac{L}{r}}(\alpha) =[\omega^n|_K]\otimes [\text{pt}] \in H(K, \partial K ; \Lambda) \otimes H(BS^1 ; \Lambda) \\
     & \,\,\text{for some}\,\, \alpha \in SH^{S^1, - ,>\frac{L}{r}}_{(M,  \omega)}(K)
\end{array} \right\}  \\
&=-\sup\left\{ r L < 0 \bigmid  \begin{array}{cc}
     &   \delta U^{k-1} \iota_L(\alpha) =[\omega^n|_K]\otimes [\text{pt}] \in H(K, \partial K ; \Lambda) \otimes H(BS^1 ; \Lambda) \\
     & \,\,\text{for some}\,\, \alpha \in SH^{S^1, - ,>L}_{(M,  \omega)}(K)
\end{array} \right\}\\
&= -r \sup\left\{  L < 0 \bigmid  \begin{array}{cc}
     &   \delta U^{k-1} \iota_L(\alpha) =[\omega^n|_K]\otimes [\text{pt}] \in H(K, \partial K ; \Lambda) \otimes H(BS^1 ; \Lambda) \\
     & \,\,\text{for some}\,\, \alpha \in SH^{S^1, - ,>L}_{(M,  \omega)}(K)
\end{array} \right\}\\
&= r c_k^{GH}(M, \omega, K).
\end{align*}
Note that the second equality labelled by $(*)$ follows from (d) of Proposition \ref{building}.\\
(b) Let $\alpha \in SH^{S^1, -, >L}_M(K ; \Lambda)$ with $\delta U^{k} \iota_L(\alpha) = [\omega^n|_K] \otimes [\text{pt}]$. Let $\beta = U_L (\alpha) \in SH^{S^1, -, >L}_M(K ; \Lambda)$. Then
\begin{align*}
    \delta U^{k-1} \iota_L(\beta)& = \delta U^{k-1} \iota_L(U_L(\alpha))  \\& = \delta U^k \iota_L(\alpha) &&\text{(b) of Proposition \ref{building}} \\& = [\omega^n|_K] \otimes [\text{pt}].
\end{align*}
(c) Let $\alpha' \in SH^{S^1, -, >L}_{M'}(K' ; \Lambda)$ with $\delta U^{k-1} \iota_L(\alpha') = [\omega^n|_{K'}] \otimes [\text{pt}]$. Let $\alpha = \Phi^L(\alpha')$. Then
\begin{align*}
    \delta U^{k-1} \iota_L (\alpha)& = \delta U^{k-1} \iota_L (\Phi^L(\alpha'))\\& = \delta \iota_L U^{k-1}_L \Phi^L(\alpha') &&\text{(a) of Proposition \ref{building2}} \\& = \delta \iota_L \Phi^L  U^{k-1}_L (\alpha') &&\text{(b) of Proposition \ref{building2}} \\&= \delta \Phi \iota_L U^{k-1}_L(\alpha') &&\text{(a) of Proposition \ref{building2}} \\&= (\rho \otimes \text{id}) \delta U^{k-1} \iota_L(\alpha') &&\text{(c) of Proposition \ref{building2}}\\&= (\rho \otimes \text{id}) ([\omega^n|_{K'}]\otimes [\text{pt}]) \\&= [\omega^n|_K] \otimes [\text{pt}].
\end{align*}
\end{proof}

\begin{proposition}\label{disjoint}
    Let $K_1, K_2 \subset M$ be disjoint compact domains with contact type and index-bounded boundaries. Then
    \begin{align*}
        c_k^{GH}(M, K_1 \cup K_2) = \max\left\{c_k^{GH}(M, K_1), c_k^{GH}(M, K_2)\right\}
    \end{align*}
    for each $k = 1,2,3,\cdots$. 
\end{proposition}
\begin{proof}
    Since $K_1$ and $K_2$ are disjoint, we can identify $H(K_1 \cup K_2, \partial(K_1 \cup K_2);\Lambda)$ with $H(K_1, \partial K_1; \Lambda) \oplus H(K_1, \partial K_1; \Lambda)$. Let \begin{align*}
        &a = \sup\left\{L \bigmid \begin{array}{cc}
         & \delta U^{k-1} \iota_L(\alpha) = [\omega^n|_{K_1 \cup K_2}] \otimes [\text{pt}] \in H(K_1 \cup K_2, \partial (K_1 \cup K_2) ; \Lambda) \otimes H(BS^1 ; \Lambda) \\
         & \,\,\text{for some}\,\, \alpha \in SH^{S^1, -, >L}_M(K_1 \cup K_2 ; \Lambda)
    \end{array}
    \right\},\\
    &b = \sup\left\{L \bigmid \begin{array}{cc}
         & \delta U^{k-1} \iota_L(\alpha) = [\omega^n|_{K_1}] \otimes [\text{pt}] \in H(K_1, \partial K_1 ; \Lambda) \otimes H(BS^1 ; \Lambda) \\
         & \,\,\text{for some}\,\, \alpha \in SH^{S^1, -, >L}_M(K_1 ; \Lambda)
    \end{array}
    \right\}
    \end{align*}
    and 
    \begin{align*}
        c = \sup\left\{L \bigmid \begin{array}{cc}
         & \delta U^{k-1} \iota_L(\alpha) = [\omega^n|_{K_2}] \otimes [\text{pt}] \in H(K_1, \partial K_2 ; \Lambda) \otimes H(BS^1 ; \Lambda) \\
         & \,\,\text{for some}\,\, \alpha \in SH^{S^1, -, >L}_M(K_2 ; \Lambda)
    \end{array}
    \right\}.
    \end{align*}
    We shall prove $a = \min\{b, c\}$. First, we prove that $a \geq \min\{b, c\}$. Suppose that $\alpha_1 \in SH^{S^1,-,>L_1}_M(K_1; \Lambda)$  and $\alpha_2 \in SH^{S^1,-,>L_2}_M(K_2; \Lambda)$ satisfy
    \begin{align*}
        \delta U^{k-1}\iota_{L_1}(\alpha_1) = [\omega^n|_{K_1}] \otimes [\text{pt}] \,\,\text{and}\,\, \delta U^{k-1}\iota_{L_1}(\alpha_1) = [\omega^n|_{K_2}] \otimes [\text{pt}].
    \end{align*}
    Denote $\min\{L_1, L_2\}$ by $L$. Let 
    \begin{align*}
        \alpha_1' = \iota_{L, L_1}(\alpha_1) \in SH^{S^1,-,>L}_M(K_1; \Lambda) \,\,\text{and}\,\,\alpha_2' = \iota_{L, L_2}(\alpha_2) \in SH^{S^1,-,>L}_M(K_2; \Lambda).
    \end{align*}
    We may think that $\alpha_1', \alpha_2' \in SH^{S^1, -, >L}_M(K_1 \cup K_2; \Lambda)$. Then 
    \begin{align*}
        &\delta U^{k-1} \iota_L(\alpha_1' + \alpha_2') &&\text{Maps are defined over $K_1 \cup K_2$}\\& = \delta U^{k-1} \iota_L(\alpha_1') + \delta U^{k-1} \iota_L(\alpha_2')&&\text{Maps are defined over $K_1$ and $K_2$}\\&=\delta U^{k-1} \iota_{L} (\iota_{L, L_1}(\alpha_1)) + \delta U^{k-1} \iota_{L} (\iota_{L, L_2}(\alpha_2)) \\&= \delta U^{k-1}\iota_{L_1}(\alpha_1) + \delta U^{k-1}\iota_{L_2}(\alpha_2)\\&= [\omega^n|_{K_1}]\otimes [\text{pt}] + [\omega^n|_{K_2}] \otimes [\text{pt}]\\&= [\omega^n|_{K_1 \cup K_2}] \otimes [\text{pt}].
    \end{align*}
    This implies that $a \geq \min\{b, c\}$. Now it remains to show that $a \leq \min\{b, c\}$. Suppose that $\alpha \in SH^{S^1,-,>L}_M(K_1 \cup K_2; \Lambda)$ satisfies 
    \begin{align}\label{decomposition}
        \delta U^{k-1} \iota_L(\alpha) = [\omega^n|_{K_1 \cup K_2}] \otimes [\text{pt}] = [\omega^n|_{K_1}] \otimes [\text{pt}] + [\omega^n|_{K_2}] \otimes [\text{pt}].
    \end{align}
    We can decompose $\alpha = \alpha_1 + \alpha_2$ where
    \begin{align*}
        \alpha_1 \in SH^{S^1,-,>L}_M(K_1; \Lambda)\,\,\text{and} \,\,\alpha_2 \in SH^{S^1,-,>L}_M(K_2; \Lambda).
    \end{align*}
    Comparing terms in \eqref{decomposition}, we have
    \begin{align}\label{comparison}
    \delta U^{k-1} \iota_L(\alpha_1) = [\omega^n|_{K_1}] \otimes [\text{pt}]\,\,\text{and}\,\,\delta U^{k-1} \iota_L(\alpha_2) = [\omega^n|_{K_2}] \otimes [\text{pt}].
    \end{align}
    From the first equality of \eqref{comparison}, $L \leq b$ and hence $a \leq b$. Analogously, the second equality in \eqref{comparison} implies $a \leq c$. Therefore, $a \leq \min\{b, c\}$.\\
\end{proof}
For the first relative Gutt-Hutchings capacity $c_1^{GH}(M,K)$, we have an easier description and the main idea to see this is Corollary \ref{firstgh}.
\begin{proposition}\label{firstghpro}
    The first relative Gutt-Hutchings capacity $c_1^{GH}(M,K)$ can be given by
    \begin{align}\label{newgh}
        c_1^{GH}(M,K) = - \sup \left\{ L < 0 \bigmid j_L^{S^1}([\omega^n|_K] \otimes [\text{pt}]) = 0 \right\}
    \end{align}
    where the map $j_L^{S^1}$ is given in Remark \ref{futureuse}.
\end{proposition}
\begin{proof}
    Recall that
    \begin{align*}
        c_1^{GH}(M, K) = - \sup \left\{ L < 0 \bigmid \delta  \iota_L (\alpha) = [\omega^n|_K] \otimes [\text{pt}] \,\, \text{for some} \,\, \alpha \in SH^{S^1, -, >L}_M(K ; \Lambda) \right\}.
    \end{align*}
    Suppose that $\delta  \iota_L (\alpha) = [\omega^n|_K] \otimes [\text{pt}]$ for some $\alpha \in SH^{S^1, -, >L}_M(K ; \Lambda)$. Then by the exactness of the triangle \eqref{ET} in the proof of Corollary \ref{firstgh}, $j_L^{S^1}([\omega^n|_K] \otimes [\text{pt}]) = j_L^{S^1}(\delta  \iota_L (\alpha)) = 0$. Conversely, suppose that $j_L^{S^1}([\omega^n|_K] \otimes [\text{pt}]) = 0$. Then also by the exactness of the triangle \eqref{ET}, there exists $\alpha \in SH^{S^1,-,>L}_M(K; \Lambda)$ such that $\delta \iota_L (\alpha) = [\omega^n|_K] \otimes [\text{pt}]$. Thus far we proved that 
    \begin{align*}
       &\left\{ L < 0 \bigmid \delta  \iota_L (\alpha) = [\omega^n|_K] \otimes [\text{pt}] \,\, \text{for some} \,\, \alpha \in SH^{S^1, -, >L}_M(K ; \Lambda) \right\} \\&=  \left\{ L < 0 \bigmid j_L^{S^1}([\omega^n|_K] \otimes [\text{pt}]) = 0 \right\}
    \end{align*}
    and this concludes the proof.\\
\end{proof}
\subsection{Relative symplectic (co)homology capacity}
Now we shall define another relative symplectic capacity exteding the definition of symplectic (co)homology capacity defined in \cite{fhw}. Background for defining this is the non $S^1$-equivariant version of Proposition \ref{building}, Corollary \ref{firstgh} and Proposition \ref{building2}. See Remark \ref{nons1}. 
\begin{definition}
    Define the relative symplectic (co)homology capacity by
    \begin{align*}
        c^{SH}(M, K) = - \sup \left\{ L < 0 \bigmid j_L([\omega^n|_K]) = 0 \right\}
    \end{align*}
    where the map $j_L$ is defined in Remark \ref{nons1}.\\
    %More precisely, the map $j_L$ used to define $c^{SH}(M, K)$ is $j_L : H^{2n}(K, \partial K ; \Lambda) \lr SH^{L,-n}_M(K ; \Lambda)$.
\end{definition}
\begin{remark}\label{altsh}
The relative symplectic (co)homology capacity $c^{SH}(M, K)$ has also alternative description similar to $c_1^{GH}(M, K)$. In a 
similar way that we proved Proposition \ref{firstghpro}, we can show that
    \begin{align*}
        c^{SH}(M, K) = - \sup \left\{ L < 0 \bigmid \delta \iota_L (\alpha) = [\omega^n|_K]  \,\, \text{for some} \,\, \alpha \in SH^{-, >L}_M(K ; \Lambda) \right\}.\\
    \end{align*}

\end{remark}
\begin{proposition}\label{shc}
    $c^{SH}(M, K)$ satisfies the following properties.
    \begin{enumerate}[label=(\alph*)]
        \item For $r>0$, $c^{SH}((M,r \omega), K) = rc^{SH}((M, \omega), K)$.\\
        \item If there exists a symplectic embedding $\phi : (M, \omega) \hookrightarrow (M', \omega')$ with $\text{int}\,\phi(K) \subset K'$, then $c^{SH}(M, K) \leq c^{SH}(M', K')$. \\
    \end{enumerate}
\end{proposition}
\begin{proof}
    (a) Denote the map $j_L$ on $(M, \omega)$ by $j_{L}^{(M, \omega)}$. Then by (the non $S^1$-equivariant version of) (d) of Proposition \ref{building} and Remark \ref{nons1}, $j_L^{(M, r \omega)} = j_{\frac{L}{r}}^{(M, \omega)}$ for each $r > 0$ and hence
    \begin{align*}
        c^{SH}((M, r \omega), K) &= - \sup \left\{ L < 0 \bigmid j_L^{(M, r \omega)}([\omega^n|_K]) = 0 \right\} \\&= - \sup \left\{ L < 0 \bigmid j_{\frac{L}{r}}^{(M,  \omega)}([\omega^n|_K]) = 0 \right\}\\& = - \sup \left\{ rL < 0 \bigmid j_{L}^{(M,  \omega)}([\omega^n|_K]) = 0 \right\} \\& = - r \sup \left\{ L < 0 \bigmid j_L^{(M, \omega)}([\omega^n|_K]) = 0 \right\} \\&= r c^{SH}((M, \omega), K).
    \end{align*}
    (b) Denote the map $j_L$ for $(M, K)$ by $j_L^{(M, K)}$. Suppose that $j_{L}^{(M', K')} ([\omega^n|_{K'}]) = 0$. Consider the following commutative diagram. 
    \begin{align}\label{shmono}
     \includegraphics[scale=1.25]{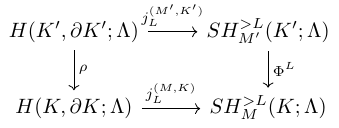}   
    \end{align}
   Suppose that $j_L^{(M', K')}([\omega^n|_{K'}]) = 0$. Since 
    \begin{align*}
        j_L^{(M, K)}([\omega^n|_K]) &= j_L^{(M, K)}\left(\rho([\omega^n|_{K'}])\right) &&\text{Definition of}\,\,\rho \\&= \Phi^L\left(j_L^{(M',K')}([\omega^n|_{K'}])\right) &&\text{Commutativity of}\,\,\eqref{shmono}\\&= 0,
    \end{align*}
    the monotonicity property of $c^{SH}$ is proved.\\
\end{proof}
The relative symplectic capacity $c^{SH}$ satisfies the same disjoint union property of $c_k^{GH}$ described in Proposition \ref{disjoint}.
\begin{proposition}\label{disjointsh}
    Let $K_1, K_2 \subset M$ be disjoint compact domains with contact type and index-bounded boundaries. Then
    \begin{align*}
        c^{SH}(M, K_1 \cup K_2) = \max\left\{c^{SH}(M, K_1), c^{SH}(M, K_2)\right\}.
    \end{align*}
    
  \end{proposition}
  \begin{proof}
        Analogous to the proof of Proposition \ref{disjoint}.\\
    \end{proof}
   For relative symplectic (co)homology capacity, we can say something when compact domains are not disjoint. In the following lemma, we use the notation $j_L^{K}$, for clarification, to denote the map $j_L : H(K, \partial K ; \Lambda ) \to SH^{>L}_M(K ;\Lambda)$ introduced in Remark \ref{nons1}.
        \begin{lemma}\label{contactpair}
            Let $(K_1, K_2)$ be a contact pair in $M$. Assume that $K_1, K_2, K_1 \cup K_2$ and $K_1 \cap K_2$ are compact domains with contact type and index-bounded boundaries and that $K_1$ and $K_2$ satisfy descent.
           \begin{enumerate}[label=(\alph*)]
                \item If $j_L^{K_1}([\omega^n|_{K_1}])=0$, then $ j_L^{K_1 \cap K_2}([\omega^n|_{K_1\cap K_2}]) = 0$. Moreover, if the restriction map 
                \begin{align*}
                    r^{K_1}_{K_1 \cap K_2} : SH_M(K_1 ; \Lambda) \lr SH_M(K_1 \cap K_2; \Lambda) 
                \end{align*}
                is injective, then the converse is also true.\\
               % \item If $j_L^{K_2}([\omega^n|_{K_2}])=0$, then $ j_L^{K_1 \cap K_2}([\omega^n|_{K_1\cap K_2}]) = 0$. Moreover, if the restriction map 
              %  \begin{align*}
               %     r^{K_2}_{K_1 \cap K_2} : SH_M(K_2 ; \Lambda) \lr SH_M(K_1 \cap K_2; \Lambda) 
              %  \end{align*}
               % is injective, then the converse is also true. \\
                \item If $ j_L^{K_1 \cup K_2}([\omega^n|_{K_1\cup K_2}]) = 0$, then $j_L^{K_1}([\omega^n|_{K_1}])=0$. Moreover, if the restriction map 
                \begin{align*}
                    r^{K_1 \cup K_2}_{K_1} : SH_M(K_1 \cup K_2; \Lambda)  \lr SH_M(K_1 ; \Lambda)  
                \end{align*}
                is injective, then the converse is also true.\\
               % \item If $ j_L^{K_1 \cup K_2}([\omega^n|_{K_1\cup K_2}]) = 0$, then $j_L^{K_2}([\omega^n|_{K_1}])=0$. Moreover, if the restriction map 
               %\begin{align*}
                   % r^{K_1 \cup K_2}_{K_2} : SH_M(K_1 \cup K_2; \Lambda)  \lr SH_M(K_2 ; \Lambda)  
               % \end{align*}
               % is injective, then the converse is also true. \\
            \end{enumerate}
       \end{lemma}
           \begin{proof}
           Consider the following commutative diagram.
           \begin{align}\label{sandwich}
               \includegraphics[scale=0.97]{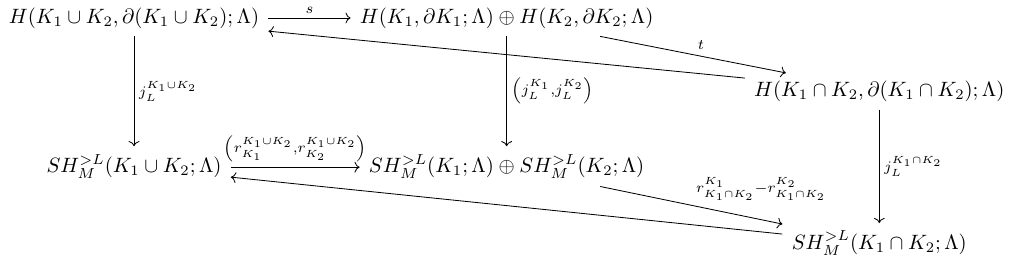}
           \end{align}
The arrows of the upper triangle of \eqref{sandwich} are induced by restrictions. More precisely,
\begin{align*}
s(\alpha) = (\alpha|_{K_1}, \alpha|_{K_2})\,\,\text{and}\,\, t(a \alpha, b \beta) = a \alpha|_{K_1 \cap K_2} - b \beta|_{K_1 \cap K_2}
   \end{align*}
for $a, b \in \Lambda$. The lower triangle of \eqref{sandwich} is the Mayer-Vietoris exact triangle shown in Proposition \ref{mv}. \\

\noindent
(a) Assume that $j_L^{K_1}([\omega^n|_{K_1}])=0$. Then
\begin{align*}
    j_L^{K_1 \cap K_2}([\omega^n|_{K_1\cap K_2}]) &= j_L^{K_1 \cap K_2}(t([\omega^n|_{K_1}], 0)) \,\,&&\text{Definition of}\,\, t \\&=\left( r^{K_1}_{K_1 \cap K_2} - r^{K_2}_{K_1 \cap K_2} \right) \left( j_L^{K_1}, j_L^{K_2}\right) ([\omega^n|_{K_1}], 0) \,\,&&\text{Commutativity of}\,\, \eqref{sandwich}\\&= r^{K_1}_{K_1 \cap K_2}   \left(j_L^{K_1}([\omega^n|_{K_1}])\right )\\&= 0.
\end{align*}
          It remains to prove the converse assuming that the restriction map \begin{align}\label{resta}
              r^{K_1 \cap K_2}_{K_1} : SH_M(K_1 \cup K_2; \Lambda)  \lr SH_M(K_1 ; \Lambda)
          \end{align} is injective. Note that if the map \eqref{resta} is injective, then its restriction to action filtration, also denoted by $r^{K_1 \cap K_2}_{K_1}$, 
          \begin{align*}
              r^{K_1 \cap K_2}_{K_1} : SH_M^{>L}(K_1 \cup K_2; \Lambda)  \lr SH_M^{>L}(K_1 ; \Lambda)
          \end{align*}
is still injective. We compute $ j_L^{K_1 \cap K_2} \left(t ([\omega^n|_{K_1}], 0)\right)$. On one hand,
\begin{align*}
    j_L^{K_1 \cap K_2} \left(t ([\omega^n|_{K_1}], 0)\right) &= j_L^{K_1 \cap K_2} ([\omega^n|_{K_1 \cap K_2}]) &&\text{Definition of}\,\, t\\&=0. &&\text{Assumption of the converse}
\end{align*}
On the other hand,
\begin{align*}
    j_L^{K_1 \cap K_2} \left(t ([\omega^n|_{K_1}], 0)\right) &=\left( r^{K_1}_{K_1 \cap K_2} - r^{K_2}_{K_1 \cap K_2} \right) \left( j_L^{K_1}, j_L^{K_2}\right) ([\omega^n|_{K_1}], 0) &&\text{Commutativity of}\,\, \eqref{sandwich}\\&= r^{K_1}_{K_1 \cap K_2}\left( j_L^{K_1}([\omega^n|_{K_1}])\right).
\end{align*}
Hence, we have $r^{K_1}_{K_1 \cap K_2}\left( j_L^{K_1}([\omega^n|_{K_1}])\right) = 0$ and $j_L^{K_1}([\omega^n|_{K_1}]) = 0$ by the injectivity of $r^{K_1}_{K_1 \cap K_2}$.\\
(b) Assume that $ j_L^{K_1 \cup K_2}([\omega^n|_{K_1\cup K_2}]) = 0$. Then
\begin{align*}
    j_L^{K_1}([\omega^n|_{K_1}])&= j_L^{K_1}(s([\omega^n|_{K_1 \cup K_2}]) &&\text{Definition of}\,\, s \\&= \left(j_L^{K_1}, j_L^{K_2} \right)(s([\omega^n|_{K_1\cup K_2}], 0)) \\&= \left( r^{K_1 \cup K_2}_{K_1}, r^{K_1 \cup K_2}_{K_2} \right) (j_L^{K_1 \cup K_2}([\omega^n|_{K_1\cup K_2}])) &&\text{Commutativity of}\,\,\eqref{sandwich}\\&=0.
\end{align*}
Now assume that $j_L^{K_1}([\omega^n|_{K_1}])=0$ and that the restriction map 
                \begin{align*}
                    r^{K_1 \cup K_2}_{K_1} : SH_M(K_1 \cup K_2; \Lambda)  \lr SH_M(K_1 ; \Lambda)  
                \end{align*}
                is injective. Similarly as before, we compute $\left( j_L^{K_1}, j_L^{K_2} \right) (s ([\omega^n|_{K_1 \cup K_2}]))$ in two ways. It can be easily seen from the commutative diagram \eqref{sandwich} that
                \begin{align*}
                    \left( j_L^{K_1}, j_L^{K_2} \right) (s ([\omega^n|_{K_1 \cup K_2}])) = \left( 0, j_L^{K_2}([\omega^n|_{K_2}])\right)
                \end{align*}
                and 
                \begin{align*}
                  \left( j_L^{K_1}, j_L^{K_2} \right) (s ([\omega^n|_{K_1 \cup K_2}])) = \left( r^{K_1 \cup K_2}_{K_1}\left(j_L^{K_1 \cup K_2} ([\omega^n|_{K_1\cup K_2}]) \right), r^{K_1 \cup K_2}_{K_2} \left( j_L^{K_1 \cup K_2}([\omega^n|_{K_1\cup K_2}])\right)\right).     
                \end{align*}
              Since $ r^{K_1 \cup K_2}_{K_1}\left(j_L^{K_1 \cup K_2} ([\omega^n|_{K_1\cup K_2}]) \right) = 0$, we have $j_L^{K_1 \cup K_2} ([\omega^n|_{K_1\cup K_2}]) = 0$ due to the injectivity of $r^{K_1 \cup K_2}_{K_1}$.\\

           \end{proof}    
           \begin{proposition}\label{nondijoint}
               Let $(K_1, K_2)$ be a contact pair in $M$. Assume that $K_1, K_2, K_1 \cup K_2$ and $K_1 \cap K_2$ are compact domains with contact type and index-bounded boundaries and that $K_1$ and $K_2$ satisfy descent.
               \begin{enumerate}[label=(\alph*)]
                   \item If the restriction map 
                \begin{align*}
                    r^{K_1}_{K_1 \cap K_2} : SH_M(K_1 ; \Lambda) \lr SH_M(K_1 \cap K_2; \Lambda) 
                \end{align*}
                is injective, then $c^{SH}(M, K_1) = c^{SH}(M, K_1 \cap K_2) $.\\
               % \item If the restriction map 
               % \begin{align*}
               %     r^{K_2}_{K_1 \cap K_2} : SH_M(K_2 ; \Lambda) \lr SH_M(K_1 \cap K_2; \Lambda) 
              %  \end{align*}
               % is injective, then $c^{SH}(M, K_2) = c^{SH}(M, K_1 \cap K_2) $.\\
                \item If the restriction map 
                \begin{align*}
                    r^{K_1 \cup K_2}_{K_1} : SH_M(K_1 \cup K_2; \Lambda)  \lr SH_M(K_1 ; \Lambda)  
                \end{align*}
                is injective, then $c^{SH}(M, K_1) = c^{SH}(M, K_1 \cup K_2)$.\\
               % \item If the restriction map 
               % \begin{align*}
                %    r^{K_1 \cup K_2}_{K_2} : SH_M(K_1 \cup K_2; \Lambda)  \lr SH_M(K_2 ; \Lambda)  
               % \end{align*}
               % is injective, then $c^{SH}(M, K_2) = c^{SH}(M, K_1 \cup K_2)$.\\
               \end{enumerate}
           \end{proposition}
      \begin{proof}
      We shall prove the first statement since the other can be proved in an analogous way.
 \begin{align*}
        c^{SH}(M, K_1) &= - \sup \left\{ L <0 \bigmid j_L^{K_1}([\omega^n|_{K_1}]) = 0 \right\} \\& = - \sup \left\{ L <0 \bigmid j_L^{K_1 \cap K_2}([\omega^n|_{K_1 \cap K_2}]) = 0 \right\} &&\text{By (a) of Lemma}\,\, \ref{contactpair} \\& = c^{SH}(M, K_1 \cap K_2).
    \end{align*}  
      \end{proof}
\begin{proposition}\label{finitesh}
   If $SH_M(K ; \Lambda) = 0$, then $c^{SH}(M, K) < \infty$.
\end{proposition}
\begin{proof}
 %   Suppose that $j_L([\omega^n|_K]) \neq 0$ for all $L \leq 0$. By Remark \ref{shcap}, we have a map $j:= j_{-\infty} : H(K, \partial K; \Lambda) \to SH_M(K;\Lambda)$. Note that $j$ can be identified with  $ \displaystyle\varinjlim_{L \to -\infty} j_L : H(K, \partial K; \Lambda) \to \varinjlim_{L \to -\infty} SH^{>L}_M(K; \Lambda) \cong SH_M(K;\Lambda)$. Since $j([\omega^n|_K]) = \displaystyle\varinjlim_{L \to -\infty} j_L([\omega^n|_K]) \neq 0$ in $SH_M(K; \Lambda)$, we conclude that $SH_M(K; \Lambda)$ is not trivial. 
 If $SH_M(K; \Lambda) = 0$, then we have an isomorphism 
 \begin{align}\label{sh0}
     H(K, \partial K ; \Lambda) \cong SH^{-}_M(K; \Lambda)
 \end{align}
 from the exact triangle \eqref{triforsh} in Remark \ref{nons1} when $L = - \infty$. Let $\alpha \in SH^{-}_M(K ; \Lambda)$ be the element corresponding to $[\omega^n|_K]$ under the isomorphism \eqref{sh0}. We may assume that $\alpha \in SH^{-,>L}_M(K ; \Lambda)$ for some $L < 0$. Since $\delta \iota_L (\alpha) = [\omega^n|_K]$, it follows from the alternative description of $c^{SH}(M,K)$ in Remark \ref{altsh} that $c^{SH}(M, K) \leq - L < \infty$.
 %Conversely, assume that $SH_M(K; \Lambda) \neq 0$. Then $SH^{>L}_M(K; \Lambda) \neq 0$ for some $L < 0$. If the map $j_L : H(K, \partial K ; \Lambda) \to SH^{>L}_M(K; \Lambda)$ is zero, then 
 %\begin{align*}
    % SH^{>L}_M(K;\Lambda) \cong SH^{-, >L}_M(K;\Lambda)
 %\end{align*}
 %from the exactness of \eqref{triforsh}. This implies furthermore that $ H(K, \partial K ; \Lambda) = 0$, which is not true. Hence, $c^{SH}(M, K) \geq -L$. Since this holds for all $L' \leq L$, we have $c^{SH}(M, K) = \infty$ by letting $L \to -\infty$. 
  
\end{proof}

%\begin{remark}
   % Similar proof of Proposition \ref{finitesh} can show that if $SH^{S^1}_M(K; \Lambda) = 0$, then $c_1^{GH}(M,K)$ is finite. Furthermore, since $SH^{S^1}_M(K ; \Lambda) = 0$ if and only if $SH_M(K; \Lambda) = 0$ by Corollary \ref{bothdie}, we can state that if $SH_M(K; \Lambda) = 0$, then $c_1^{GH}(M,K)$ is finite. However, analogous argument cannot show the other direction of $c_1^{GH}(M,K)$ version of Proposition \ref{finitesh} because $H(K, \partial K ; \Lambda) \otimes H(BS^1 ; \Lambda)$ is not 1-dimensional over $\Lambda$. More precisely, $j_L^{S^1}([\omega^n|_K] \otimes [\text{pt}]) = 0$ is not equivalent to saying that the map $j_L^{S^1}$ is a zero map.\\
%\end{remark}
\subsection{Comparison of $c_1^{GH}$ and $c^{SH}$}
Now we shall compare the first relative Gutt-Hutchings capacity $c_1^{GH}(M, K)$ and the symplectic (co)homology capacity $c^{SH}(M, K)$. The following lemma tells us that in general $c^{SH}(M, K)$ is bigger than or equal to $c_1^{GH}(M, K)$.
\begin{lemma}\label{comm}
        $c_1^{GH}(M, K) \leq c^{SH}(M, K)$.  
        
\end{lemma}
\begin{proof}
    Recall from Proposition \ref{firstghpro} that
    \begin{align*}
        c_1^{GH}(M,K) = - \sup \left\{ L < 0 \bigmid j^{S^1}_L([\omega^n|_K] \otimes [\text{pt}]) = 0 \right\}.
    \end{align*}
    Let 
    \begin{align*}
        \zeta = \sup \left\{ L < 0 \bigmid j^{S^1}_L([\omega^n|_K] \otimes [\text{pt}]) = 0 \right\}
    \end{align*}
    and
    \begin{align*}
        \xi = \sup \left\{L < 0 \bigmid j_L([\omega^n|_K]) = 0 \right\}.
    \end{align*}
    It suffices to show that $\zeta \geq \xi$. Consider the following commutative diagram.
    \begin{align}\label{ladder}
        \includegraphics[scale=1.2]{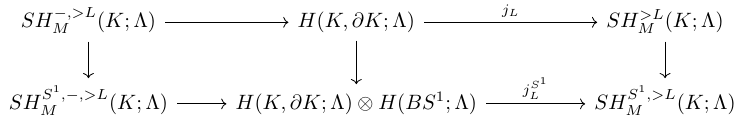}
    \end{align}
    The first and the second row of \eqref{ladder} come from exact triangles \eqref{ET} and \eqref{triforsh}, respectively. The first and the third vertical maps are induced by $x \mapsto 1 \otimes x$ and the second vertical map is induced by $x \mapsto x \otimes [\text{pt}]$. Suppose that $j_L([\omega^n|_K]) = 0$ for some $L < 0$. Then by the commutativity of the second rectangle of the diagram \eqref{ladder}, we have $j^{S^1}_L([\omega^n|_K] \otimes [\text{pt}]) = 0$ and this proves that $\xi \leq \zeta$.\\
\end{proof}
  With some assumption on the Conley-Zehnder indices of Hamiltonian orbits, we can also prove the opposite inequality of Lemma \ref{comm}. We recall the following definition.
\begin{definition}
    Let $(\Sigma, \xi, \alpha)$ be a contact manifold of dimension $2n - 1$ where $\xi$ is a contact structure and $\alpha$ is a contact form. Assume that the first Chern class $c_1(\xi)$ vanishes. A contact form $\alpha$ is called \textbf{dynamically convex} if every contractible periodic Reeb  orbit $\gamma$ of $\alpha$ satisfies $CZ(\gamma) \geq n+1$.
    \end{definition}
Following lemma is the main ingredient to prove the opposite inequality of Lemma \ref{comm}.

\begin{lemma}\label{-n}
   Assume that $\dim M = 2n$ and the contact form $\lambda_K$ of $\partial K$ is dynamically convex. If $\ell \geq -n$, then $SH^{S^1,-, >L, \ell}_M(K) = 0$ for each $L \in \RR_- \cup \{-\infty\}$.
\end{lemma}

To prove Lemma \ref{-n} above, we need a slight modification of the definition of $\mathcal{H}^{MB}$ introduced in Definition \ref{mb}. 
\begin{definition}
    Let $(M, \omega)$ be a closed symplectic manifold and $K \subset M$ be a compact domain with contact type boundary. A \textbf{$K$-admissible Morse-Bott Hamiltonian function} is an autonomous function $H : M \to \RR$ satisfying the following conditions:
    \begin{enumerate}[label=(\alph*)]
        \item $H$ is negative and $C^2$-small on $K$\\
        \item There exists $\rho_0 > 0$ such that $H(p, \rho)$ is $C^2$-close to $h_1(e^{\rho})$ on $\partial K \times [0, \frac{1}{3}\rho_0]$ for some convex and increasing function $h_1$ satisfying $h_1^{''} - h_1^' >0$.\\
        \item  $H(p, \rho) = \beta e^{\rho} + \beta'$ on $\partial K \times [\frac{1}{3}\rho_0, \frac{2}{3}\rho_0]$ where $\beta \notin \text{Spec}(\partial K, \lambda_K)$ and $\beta' \in \RR$.\\
        \item $H(p, \rho)$ is $C^2$-close to $h_2(e^{\rho})$ on $\partial K \times [\frac{2}{3}\rho_0, \rho_0]$ for some concave and increasing function $h_2$.\\
        \item $H$ is $C^2$-close to a constant function on $M - K \cup (\partial K \times [0, \rho_0])$.
       
    \end{enumerate}
    We denote the set of all $K$-admissible Morse-Bott Hamiltonian functions by $\mathcal{H}_K^{MB}$.\\
\end{definition}

Theorem \ref{reeb} still holds for this newly defined Hamiltonian functions in $\mathcal{H}_K^{MB}$. Following theorem is a rewriting of Theorem \ref{reeb} and a detailed proof of it can be found in \cite{bo09} and \cite{cfhw}.

\begin{theorem}\label{Kversion}
     A $K$-admissible Morse-Bott Hamiltonian function $H \in \mathcal{H}_K^{MB}$ can be perturbed to a contact type $K$-adimissible Hamiltonian function $H' \in \mathcal{H}_K^{\text{Cont}}$ whose 1-periodic orbits satisfy the following:
    \begin{enumerate}[label=(\alph*)]
        \item The constant periodic orbits of $H'$ are the critical points of $H$.\\
        \item For each Reeb orbit $\gamma$ of $(\partial M, \lambda_M)$, there are two nondegenerate Hamiltonian orbits $\gamma_{\text{Max}}$ and $\gamma_{\text{min}}$ of $H'$ such that for a given symplectic trivialization $\tau$ along $\gamma$,
        \begin{align*}
            - \text{CZ}_{\tau}(\gamma_{\text{Max}}) = \text{CZ}_{\tau}(\gamma) \,\, \text{and}\,\, - \text{CZ}_{\tau}(\gamma_{\text{min}}) = \text{CZ}_{\tau}(\gamma) +1 
        \end{align*}
        where $\text{CZ}_{\tau}$ stands for the Conley-Zehnder index calculated under the trivialization $\tau$.\\
    \end{enumerate}
\end{theorem}

\noindent\textit{Proof of Lemma \ref{-n}}.
   Let us consider the case where $L = -\infty$ first. Let $H \in \mathcal{H}^{\text{Cont}}_K$. We may assume that $H$ is perturbed from a $K$-admissible Morse-Bott Hamiltonian function. Then by Theorem \ref{Kversion}, the generating set of $CF^{S^1,\ell}_w(H) / CF^{S^1, >-\epsilon, \ell}_w(H)$ is 
    \begin{align*}
        \{ u^k \otimes x_{\text{Max}} \mid \text{CZ}(x_{\text{Max}}) - 2k = \ell, \,k=0,1,2,\cdots \} \\\cup\,\, \{ u^k \otimes x_{\text{min}} \mid \text{CZ}(x_{\text{min}}) - 2k = \ell, \,k=0,1,2,\cdots \}
    \end{align*}
where $x$ is a Reeb orbit of $(\partial K, \lambda_K)$ and $x_{\text{Max}}$ and $x_{\text{min}}$ are Hamiltonian orbits of $H$ described in Theorem \ref{Kversion}. It follows from the proof of Theorem 1.1 of \cite{gutt} that the $SH^{S^1,-,\ell}_M(K) = H \left(CF^{S^1,\ell}_w(H) / CF^{S^1, >-\epsilon, \ell}_w(H) \right)$ is generated by the (equivalence classes) of $u^0 \otimes x_{\text{Max}}$. Then
\begin{align*}
    \ell& = \text{CZ}(x_{\text{Max}}) \\& = - \text{CZ}(x)&&\text{Theorem \ref{Kversion}}\\& \leq -n-1. &&(\partial K, \lambda_K)\,\,\text{is dynamically convex}
\end{align*}
Therefore, $SH^{S^1,-,\ell}_M(K) = 0$ for $\ell \geq -n$. This proof actually shows that if $\ell \geq - n$, then there is no generator of grading $\ell$. Since there is still no generator of grading $\ell$ with action greater than $L \in \RR_-$, we also have $SH^{S^1,-,>L, \ell}_M(K) = 0$ for each $L \in \RR_- $.\\
\qed   
  
\begin{theorem}\label{shgheq}
    If $\dim M = 2n$ and the contact form $\lambda_K$ of $\partial K$ is dynamically convex, then
    \begin{align*}
        c_1^{GH}(M, K) = c^{SH}(M, K).
    \end{align*}
\end{theorem}
\begin{proof}
Recall that
    \begin{align*}
        c_1^{GH}(M, K) = - \sup \left\{ L < 0 \bigmid \delta  \iota_L (\alpha) = [\omega^n|_K] \otimes [\text{pt}] \,\, \text{for some} \,\, \alpha \in SH^{S^1, -, >L}_M(K ; \Lambda) \right\}.
    \end{align*}
    and 
    \begin{align*}
        c^{SH}(M, K) = - \sup \left\{ L < 0 \bigmid \delta \iota_L (\alpha') = [\omega^n|_K]  \,\, \text{for some} \,\, \alpha' \in SH^{-, >L}_M(K ; \Lambda) \right\}.
    \end{align*}
More precisely, the class $\alpha$ defining $c_1^{GH}(M, K)$ is in $SH^{S^1, -, >L, -n-1}_M(K ; \Lambda)$ and the class $\alpha'$ defining $c^{SH}(M, K)$ is in $SH^{-, >L, -n-1}_M(K ; \Lambda)$. For the convention on the grading, see Remark \ref{convention}. Let
\begin{align*}
    \zeta = \sup \left\{ L < 0 \bigmid \delta  \iota_L (\alpha) = [\omega^n|_K] \otimes [\text{pt}] \,\, \text{for some} \,\, \alpha \in SH^{S^1, -, >L}_M(K ; \Lambda) \right\}
\end{align*}
and let
\begin{align*}
    \xi = \sup \left\{ L < 0 \bigmid \delta \iota_L (\alpha') = [\omega^n|_K]  \,\, \text{for some} \,\, \alpha' \in SH^{-, >L}_M(K ; \Lambda) \right\}.
\end{align*}
This time, we need to prove $\zeta \leq \xi$. We can extract a part of the Gysin sequence constructed in (b) of Corollary \ref{gysinaction} as follows.
    \begin{align*}
        SH^{S^1,-,>L, -n}_M(K ; \Lambda) &\lr SH^{-,>L, -n-1}_M(K ; \Lambda) \\&\lr SH^{S^1,-,>L,-n-1}_M(K ; \Lambda) \lr SH^{S^1,-,>L,-n+1}_M(K ; \Lambda).
    \end{align*}
    It follows from Lemma \ref{-n} that $SH^{S^1,-,>L, -n}_M(K ; \Lambda) = 0$ and $SH^{S^1,-,>L,-n+1}_M(K ; \Lambda) = 0$. Therefore,
    \begin{align}\label{iso}
        SH^{-,>L, -n-1}_M(K ; \Lambda) \cong SH^{S^1,-,>L,-n-1}_M(K ; \Lambda).
    \end{align}
    Let us denote this isomorphism \eqref{iso} by $\eta$. Following commutative diagram is a part of \eqref{ladder} magnifying what we need.
    \begin{align}\label{ladder2}
        \includegraphics[scale=1.1]{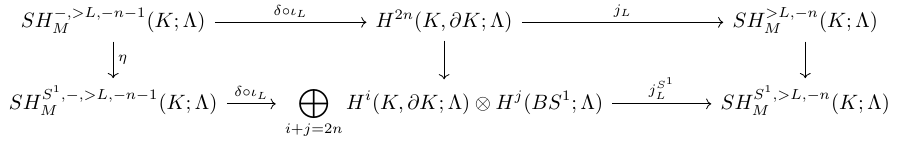}
    \end{align}
    As before, the vertical map in the middle is given by $x \mapsto x \otimes [\text{pt}]$. Suppose $\delta \iota_L(\alpha) = [\omega^n|_K] \otimes [\text{pt}]$ for some $\alpha \in SH^{S^1, -, >L, -n-1}_M(K ; \Lambda)$. Since $\eta$ is an isomorphism, there exists $\alpha' \in SH^{-,>L, -n-1}_M(K ; \Lambda)$ corresponding to $\alpha$. Then by the commutativity of the first rectangle in \eqref{ladder2}, we should have $\delta \iota_L (\alpha') = [\omega^n|_K]$ and we can conclude the proof.
   \\
\end{proof}

\Addresses


\begin{thebibliography} {99}
    \bibitem{ak} A. Abbondandolo, J, Kang, Symplectic homology of convex domains and Clarke’s duality, Duke Mathematical Journal, 171(3), 739 - 830.

    \bibitem{as} M. Abouzaid, P. Seidel, An open string analogue of Viterbo functoriality, Geometry and Topology, 14  (2010) 627 - 718.

    \bibitem{bo09} F. Bourgeois, A. Oancea, Symplectic homology, autonomous Hamiltonians, and Morse-Bott moduli spaces, Duke Mathematical Journal, 146(1) (2009), 71 - 174. 

    \bibitem{bo13} F. Bourgeois, A. Oancea, The Gysin exact sequence for $S^1$-equivariant symplectic homology, 
    Journal of Topology and Analysis, 5(4), 2013, 361 – 407. 
    
    \bibitem{bo} F. Bourgeois, A. Oancea, $S^1$-equivariant symplectic homology and linearized contact homology, International Mathematics Research Notices, 13, (2017), 3849 - 3937.

  %  \bibitem{co} K. Cieliebak, A. Oancea, Symplectic homology and the Eilenberg–Steenrod axioms, Algebraic and Geometric Toplogy, 18 (2018), 1953 - 2130. 

    \bibitem{cfh} K. Cieliebak, A. Floer, H. Hofer, Symplectic homology II. A general construction, Mathematische Zeitschrift, 218 (1995), 103 - 122.

    \bibitem{cfhw} K. Cieliebak, A. Floer, H. Hofer, K. Wysocki, Applications of symplectic homology II: Stability of the action spectrum, Mathematische Zeitschrift, 223 (1996), 27 - 45.


    \bibitem{cz} C. Conley, E. Zehnder, Morse-type index theory for flows and periodic solutions for Hamiltonian Equations, Communications on Pure and Applied Mathematics, 37(2) (1984), 207 - 253.
    
     \bibitem{dgpz} A. Dickstein, Y. Ganor, L. Polterovich, F. Zapolsky, Symplectic topology and ideal-valued measures, Selecta Mathematica, to appear.

    \bibitem{eh} I. Ekeland, H. Hofer, Symplectic topology and Hamiltonian dynamics, Mathematische Zeitschrift, 200 (1989), 355 - 378.

    \bibitem{ep09} M. Entov, L. Polterovich, Rigid subsets of symplectic manifolds, Compositio Mathematica, 145 (2009), 773 - 826.

   \bibitem{fh} A. Floer, H. Hofer, Symplectic homology I. Open sets in $\CC^n$, Mathematische Zeitschrift, 215 (1994), 37 - 88.
       
    \bibitem{fhw} A. Floer, H. Hofer, K. Wysocki, Applications of symplectic homology \MakeUppercase{\romannumeral 1}, Mathematische Zeitschrift, 217 (1994), 577 - 606.

    \bibitem{g} Y. Groman, Floer theory and reduced cohomology on open manifolds, Geometry and Topology, 27 (2023), 1273 – 1390.

    \bibitem{gutt} J. Gutt, The positive equivariant symplectic homology as an invariant for some contact manifolds, Journal of Symplectic Geometry, 15 (2017), 1019 - 1069. 

    \bibitem{gh} J. Gutt, M. Hutchings, Symplectic capacities from $S^1$-equivariant symplectic homology, Algebraic and Geometric Toplogy, 18 (2018), 3537 - 3600.

    \bibitem{gs} V. Ginzburg, J. Shon, On the filtered symplectic homology of prequantization bundles, International Journal of Mathematics, 29 (11) (2018), 1850071.

    \bibitem{gt} Y. Ganor, S. Tanny, Floer theory of disjointly supported Hamiltonians on symplectically aspherical manifolds, Algebraic and Geometric Topology, 23(2) (2023), 645–732.
    
    \bibitem{hs} H. Hofer, D. Salamon, Floer homology and Novikov rings, Progress in Mathematics, 133 (1995), 483 - 524.

    \bibitem{i} K. Irie, Symplectic homology of fiberwise convex sets and homology of loop spaces, Journal of Symplectic Geometry, 20 (2022), 417 - 470.
    
    
    \bibitem{msv} C.Y. Mak, Y. Sun, U. Varolgunes, A characterization of heaviness in terms of relative symplectic cohomology, Journal of Toplogy, to appear.

    \bibitem{osurv} A. Oancea, A survey of Floer homology for manifolds with contact type boundary or symplectic homology, Ensaios Matemáticos, Sociedade Brasileira de Matemática, Rio de Janeiro (2004), 51 - 91.
   
    \bibitem{p} J. Pardon, Contact homology and virtual fundamental cycles, Journal of American Mathmatical Society, 32 (2019), 825 - 919.

    \bibitem{pss} S. Piunikhin, D. Salamon, M. Schwarz, Symplectic Floer-Donaldson theory and quantum cohomology, Contact and symplectic geometry, Publication of Newton Institute, 8, Cambridge University Press (1996), 171–200. 

    \bibitem{r} A. Ritter, Topological quantum field theory structure on symplectic cohomology, Journal of Topology, 6 (2013), 391 - 489.

    \bibitem{rs} J. Robbin, D. Salamon, The Maslov index for paths, Topology, 32 (1993), 827 - 844.

    \bibitem{s} Y. Sun, Index bounded relative symplectic cohomology, Algebraic and Geometric Topology, to appear.

    \bibitem{s24} Y. Sun, Heavy sets and index bounded relative symplectic cohomology, Journal of Fixed Point Theory and Applications, 26 (2024), 21.

    \bibitem{v} U. Varolgunes, Mayer–Vietoris property for relative symplectic cohomology, Geometry and Topology, 25 (2021), 547 - 642.

    \bibitem{vt} U. Varolgunes, Mayer–Vietoris property for relative symplectic cohomology, Ph.D. thesis, Massachusetts Institute of Technology.

    \bibitem{vc} U. Varolgunes, personal communication.
    
    \bibitem{ven} S. Venkatesh, The quantitative nature of reduced Floer theory, Advances in Mathematics, 383 (2021), 107682.
    
    \bibitem{vi} C. Viterbo, Functors and computations in Floer homology with applications, Geometric and Functional Analysis, 9 (1999), 985 - 1033
    
\end{thebibliography}
\end{document}